\documentclass[11pt]{article}
\usepackage{amsmath,amsfonts,amsthm,amssymb,mathabx,bbm}
\usepackage{setspace}
\usepackage{Tabbing}
\usepackage{lastpage}
\usepackage{extramarks}
\usepackage{soul,color}
\usepackage{graphicx,float,wrapfig}
\usepackage{xypic}
\usepackage{nicefrac}
\usepackage{multicol}

\newcommand{\bbE}{{\mathbb{E}}}
\newcommand{\bbN}{{\mathbb{N}}}
\newcommand{\bbR}{{\mathbb{R}}}
\newcommand{\bbC}{{\mathbb{C}}}

\newcommand{\bbZ}{{\mathbb{Z}}}
\newcommand{\bbQ}{{\mathbb{Q}}}
\newcommand{\bbT}{{\mathbb{T}}}

\newcommand{\mcA}{{\mathcal{A}}}
\newcommand{\mcB}{{\mathcal{B}}}

\newcommand{\mcE}{{\mathcal{E}}}
\newcommand{\mcF}{{\mathcal{F}}}

\newcommand{\mcH}{{\mathcal{H}}}
\newcommand{\mcI}{{\mathcal{I}}}
\newcommand{\mcK}{{\mathcal{K}}}
\newcommand{\mcL}{{\mathcal{L}}}
\newcommand{\mcM}{{\mathcal{M}}}

\newcommand{\mcP}{{\mathcal{P}}}

\newcommand{\mcR}{{\mathcal{R}}}
\newcommand{\mcS}{{\mathcal{S}}}
\newcommand{\mcT}{{\mathcal{T}}}

\newcommand{\mcV}{{\mathcal{V}}}
\newcommand{\mcW}{{\mathcal{W}}}

\newcommand{\mcZ}{{\mathcal{Z}}}

\newcommand{\mfp}{{\mathfrak{p}}}

\newcommand{\ba}{{\mathbf{a}}}
\newcommand{\bb}{{\mathbf{b}}}
\newcommand{\bc}{{\mathbf{c}}}
\newcommand{\be}{{\mathbf{e}}}
\newcommand{\bh}{{\mathbf{h}}}
\newcommand{\bi}{{\mathbf{i}}}
\newcommand{\bj}{{\mathbf{j}}}
\newcommand{\bk}{{\mathbf{k}}}

\newcommand{\bq}{{\mathbf{q}}}

\newcommand{\bu}{{\mathbf{u}}}
\newcommand{\bv}{{\mathbf{v}}}
\newcommand{\bw}{{\mathbf{w}}}
\newcommand{\bx}{\mathbf{x}}

\newcommand{\bz}{\mathbf{z}}

\newcommand{\bA}{{\mathbf{A}}}
\newcommand{\bB}{{\mathbf{B}}}
\newcommand{\bC}{{\mathbf{C}}}
\newcommand{\bD}{{\mathbf{D}}}

\newcommand{\bH}{{\mathbf{H}}}
\newcommand{\bI}{{\mathbf{I}}}

\newcommand{\bL}{{\mathbf{L}}}
\newcommand{\bM}{{\mathbf{M}}}

\newcommand{\bP}{{\mathbf{P}}}
\newcommand{\bQ}{{\mathbf{Q}}}
\newcommand{\bR}{{\mathbf{R}}}
\newcommand{\bS}{{\mathbf{S}}}

\newcommand{\bZero}{{\mathbf{0}}}
\newcommand{\bOne}{{\mathbf{1}}}

\newcommand{\IND}{{\mathbbm{1}}}

\newcommand{\supp}{{\text{supp}}}

\newcommand{\stpos}{\tbump{-1px}{{\ensuremath{\underline{\gg}}}}}
\newcommand{\Zd}{\ensuremath{\mathbb{Z}^d}}

 \newcommand{\adot}{{\cdot}}

\newcommand{\fatcup}{{\scalebox{1.7}{$\cup$}}}
\newcommand{\fatsqcup}{{\scalebox{1.7}{$\sqcup$}}}

\newcommand{\AZd}{{\mathcal{A}^{\mathbb{Z}^d}}}

\newcommand{\sigmaX}{{\sigma_{\text{max}}}}

\newcommand{\btheta}{\boldsymbol\theta}

\newcommand{\bomega}{\boldsymbol\omega}
\newcommand{\bdelta}{\boldsymbol\delta}
\newcommand{\bzeta}{\boldsymbol\zeta}

\newcommand{\bnu}{\boldsymbol\nu}

\newcommand{\qq}[1]{\ensuremath{\left\lfloor #1\right\rfloor}}
\newcommand{\qr}[1]{\ensuremath{\left\lceil #1 \right\rceil}}

\newtheorem{thm}{Theorem}[section]

\newtheorem{lem}[thm]{Lemma}
\newtheorem{prop}[thm]{Proposition}
\newtheorem{cor}[thm]{Corollary}

\newtheorem*{theorem*}{Theorem}

\theoremstyle{definition}
\newtheorem{defn}[thm]{Definition}
\newtheorem{exmp}[thm]{Example}

\newcommand{\snote}[1]{}

\newcommand{\bump}{{\hspace{0.05cm}}}

\newcommand{\tbump}[2]{\ensuremath{\hspace{#1} \text{ #2 } \hspace{#1}}}

\begin{document}
\large\textbf{\centerline{Spectral Theory of $\Zd$ Substitutions}}

\vspace{4px}
\centerline{Alan Bartlett}
\normalsize\centerline{University of Washington}

{\small
\begin{abstract}
	In this paper, we generalize and develop results of Queff\'elec allowing us to characterize the spectrum of an aperiodic $\Zd$ substitution.  Specifically, we describe the Fourier coefficients of mutually singular measures of pure type giving rise to the maximal spectral type of the translation operator on $L^2$, without any assumptions on primitivity or height, and show singularity for aperiodic bijective commutative $\Zd$ substitutions.  Moreover, we provide a simple algorithm to determine the spectrum of aperiodic $\bq$-substitutions, and use this to show singularity of Queff\'elec's noncommutative bijective substitution, as well as the Table tiling, answering an open question of Solomyak.  Finally, we show that every ergodic matrix of measures on a compact metric space can be diagonalized, which we use in the proof of the main result.\let\thefootnote\relax\footnote{
This research is part of a Ph.D. Thesis (under the direction of Boris Solomyak) at the University of Washington, was supported by NSF grants RTG-0838212 and DMS-1361424.}
\end{abstract}
}
\normalsize

Substitutions are of interest to us as models of aperiodic phenomena - for example, mathematical quasicrystals, see [\ref{TAO}].  Here, we consider higher dimensional analogues of substitutions of constant length, or substitutions on $\Zd$-indexed sequences which replace letters in a finite alphabet with a rectangular block of letters, and which we call $\bq$-substitutions.   As usual, we study the translation operator on the hull of a substitution, or the collection of all $\Zd$-indexed sequences whose local patterns can be produced by the substitution.  Our primary interest is the spectrum of the translation action, and our analysis is based on the work of Queff\'elec, which describes the spectrum of one-sided primitive and aperiodic constant length substitutions of trivial height.   Our formulation emphasizes the arithmetic properties of $\bq$-substitutions and the recognizability properties afforded by aperiodicity, describing the spectrum of such a substitution by identifying its discrete, singular continuous, and absolutely continuous spectral components as a sum of mutually singular measures of pure type\snote{\footnote{That is, these measures are either purely discrete, purely singular continuous, or purely absolutely continuous, with respect to Lebesgue measure on the $d$-Torus.}} with readily computable Fourier coefficients.  This is done using a representation consistent with the abelianization (which gives rise to the substitution matrix) and separates the analysis into two parts: the \textit{correlation measures}, which depend on how letters are arranged upon substitution or their \textit{configuration}, and the \textit{spectral hull}, a convex set which depends only on substituted letters and not their arrangement within the blocks.

An important notion in tiling theory is that of \textit{local rules} which force global (aperiodic) order. A result of Moz\'es [\ref{mozes}] asserts that self-similar tilings corresponding to $\Zd$ substitutions for $d>1$ can be obtained from local pattern matching rules (by increasing the size of the alphabet),  whereas this assertion is false for $d=1$, see [3] for details and more recent work in this direction.
Moreover, it is known that mathematical diffraction spectrum is included in the dynamical spectrum (that of the translation operator on $L^2$, which we study here) as a consequence of Dworkin's argument [\ref{dworkin}] (see also [\ref{LMS}], [\ref{BLV}] and references therein), and so $\bq$-substitutions are of interest as they allow us to describe dynamical and diffractive properties of higher dimensional mathematical quasicrystals.

Although our main result is the extension of Queff\'elec's work in [\ref{queffelec}] to aperiodic $\bq$-substitutions, our primary goal is an algorithm for determining their spectrum, and so we summarize the paper while discussing the main results in this context.  For us, this means characterization of the maximal spectral type $\sigmaX$ (see \S \ref{spectral theory section}) of the translation operator on $L^2$ in terms of its discrete, singular continuous, and absolutely continuous components.  Loosely speaking, the algorithm relies on a representation of a substitution $\mcS$ in matrix algebras, by associating \textit{instruction, substitution, and coincidence matrices} $\mcR_\bj, M_\mcS \in \bM_\mcA(\bbC)$ and $C_\mcS \in \bM_{\mcA^2}(\bbC)$ given by
$$\scalebox{0.9}{$\left(\mcR_\bj\right)_{\alpha, \beta} :=  \scalebox{0.9}{${\small \begin{cases} \,1 & \text{if } \alpha = \left(\mcS \beta\right)(\bj) \\ \, 0 & \text{if } \alpha \neq \left(\mcS \beta\right)(\bj) \end{cases}}$} \hspace{0.75in} M_\mcS = {\sum}_{\bj \in [\bZero,\bq)} \mcR_\bj \hspace{0.75in} C_\mcS := {\sum}_{\bj \in [\bZero,\bq)} \mcR_\bj \otimes \mcR_\bj$}$$
\noindent respectively, where $[\bZero,\bq) = \{ \bj \in \Zd \, : \, 0 \leq j_i < q_i \, \forall i \, \}$ and $\otimes$ is the Kronecker product, see \S \ref{configurations}.  

Due to results of Michel [\ref{michel}] extended by Cortez and Solomyak [\ref{cortez and solomyak}], we know that the invariant measures of the substitution subshift $X_\mcS$ correspond to the \textit{core of $M_\mcS$}, a convex subset of $\bbC^\mcA$ consisting of strictly positive linear combinations of the Perron vectors of $M_\mcS$.  Here, the extreme points of the core correspond to ergodic invariant measures supporting minimal components of the subshift, see \S \ref{invariant measures section}. Thus, fixing a positive linear combination $\bu \in \bbC^\mcA$ of the Perron vectors of the primitive components of $\mcS$  has the effect of fixing an invariant measure $\mu \in \mcM(X_\mcS,T)$ on the subshift supported by all of its ergodic invariant measures.

Our main result, theorem \ref{queffelec two} in \S \ref{queffelec section} establishes a similar representation for the spectrum of $\mcS$.   Specifically, it identifies the spectrum of $\mcS$ with the \textit{spectral hull} $\mcK$, a convex subset of $\bbC^{\mcA^2}$ consisting of \textit{strongly positive} left $Q$-eigenvectors of $C_\mcS$, a positivity condition relating to positive definiteness in $\bM_\mcA(\bbC)$, see \S \ref{spectral hull}.  The extreme points $\mcK^\ast$ correspond to measures ergodic for the $\bq$-shift on $\bbT^d \subset \bbC^d$ (conjugate to multiplication by $\bq$ on $\bbR^d / \Zd$), which gives us a decomposition of the spectrum into mutually singular measures of pure, and readily identifiable, type.

The correspondence here is explicitly provided by taking the vectors in the spectral hull to corresponding linear combinations of \textit{correlation measures} $\sigma_{\alpha \beta}$, or spectral measures for indicator functions of sequences with fixed value at the origin.  Using arithmetic properties of $\bq$-substitutions (proposition \ref{qsub}) and recognizability properties of aperiodic $\bq$-substitutions (Moss\'e's theorem \ref{aperiodicity implies recognizability}, extended by Solomyak), we can explicitly compute the Fourier coefficients of the correlation measures via a linear recursion (theorem \ref{fourier recursion}) in the instruction matrices.  By the main result, this allows us to compute the Fourier coefficients of mutually singular measures of pure type giving rise to the spectrum of $\mcS$ via convolution with a $\bq$-adic support measure on $\bbT^d$, the $d$-torus.

For an aperiodic $\bq$-substitution $\mcS$ on $\mcA$, compute the maximal spectral type $\sigmaX$ by finding
{\small
\begin{enumerate}
	\item the matrices $\mcR_\bj$ for $\bj \in [\bZero,\bq)$, $M_\mcS$ and $C_\mcS,$ and the primitive reduced forms (\ref{matrix PRF})  of $\mcS$ and $\mcS \otimes \mcS$
	
	\item a strictly positive linear combination $\bu \in \bbC^\mcA$ of the Perron vectors of $M_\mcS$
	
	\item the spectral hull, or the strongly positive $Q$-eigenspace of $C_\mcS^t$, and its extreme points $\mcK^\ast$
	
	\item the Fourier coefficients $\widehat{\sigma_{\alpha \beta}}(\bk)$, computed via theorem \ref{fourier recursion} for $\alpha \beta \in \mcA^2$ and $\bk \in \Zd.$
	
	\item the Fourier coefficients of the measures $\lambda_\bv = \sum_{\alpha \beta \in \mcA^2} v_{\alpha \beta} \sigma_{\alpha \beta}$ for $\bv \in \mcK^\ast \subset \bbC^{\mcA^2}$
\end{enumerate}
}
\noindent so that, by theorem \ref{queffelec two}, the maximal spectral type $\sigmaX$ of $\mcS$ is the sum of $\bomega_\bq \ast \lambda_\bv$ as $\bv$ ranges over the extreme points of the spectral hull, where $\bomega_\bq$ is a support measure for the $\bq$-adic roots of unity in $\bbT^d$ and $\ast$ is convolution of measures.  We work out the above algorithm for the Thue-Morse substitution as we develop our results, in addition to several detailed examples in \S \ref{examples}.

We organize the paper as follows: in \S \ref{sds section}, we discuss some preliminaries for the theory of $\Zd$-indexed substitution dynamical systems, including their invariant measures and spectrum.  In \S \ref{analysis of q sub} we develop arithmetic properties of $\bq$-substitutions which, along with an aperiodicity result of Moss\'e extended by Solomyak, allow us to prove a recursive identity related to invariant measures which arise when studying their spectral theory.  In \S \ref{spectral theory}, we separate the spectral characterization problem of an aperiodic $\bq$-substitution into a study of its correlation measures and spectral hull; see theorems \ref{queffelec one} and \ref{queffelec two}.  There, we use the measure recursion to explicitly compute Fourier coefficients of the correlation measures (theorem \ref{fourier recursion}), and describe an algorithm for characterizing the spectral hull (proposition \ref{characterization of K}) of a $\bq$-substitution.  These combine via theorem \ref{queffelec two} to give explicit formulae for the Fourier coefficients of measures of pure type which determine the spectrum of an aperiodic $\bq$-substitution.   Further, we show that aperiodic bijective commutative $\bq$-substitutions have purely singular (to Lebesgue) spectrum, generalizing a result of Baake and Grimm in [\ref{baake and grimm}].  Following this, we compute several examples in \S \ref{examples}, including the spectrum of the Table substitution tiling [\ref{robinson}], and correct an example of Queff\'elec which was mistakingly identified as having Lebesgue component in its spectrum (see example \ref{queffelec example}).  In \S \ref{appendix} we prove a diagonalization result for operator valued measures ergodic for a continuous transformation of a compact metric space, allowing us to prove theorem \ref{queffelec two}, whose proof is delayed as it is rather involved and not relevant to the algorithm or statement of the result.  Finally, at the end we include a brief notation and terminology reference for the benefit of the reader.

Before beginning, we make a remark on a particularly important notational convention.  We will be working extensively with the ring $\Zd,$ and wish to do so by interpreting all operations and relations on $\Zd$ coordinatewise.  We represent $\Zd$ integers in boldface $\bi, \bj, \bk,$ etc, and denote the components of $\bk$ with $k_i$ for $1 \leq i \leq d.$  The symbols $\bZero,\bOne$ represent the $\Zd$ integers all of whose coordinates are $0, 1$ respectively, and for $1 \leq i \leq d,$ let $\bOne_i \in \Zd$ be the integer $0$ in all coordinates but the $i$-th, where it is $1,$ so that $\bOne = \sum_1^d \bOne_i.$  For $\ba, \bb \in \Zd,$ the inequalities $\ba < \bb$ or $\ba \leq \bb$ should be interpreted as holding in each coordinate simultaneously, i.e. $a_i < b_i,$ for $1 \leq i \leq d,$ and so defines a partial order on $\Zd.$ Additionally, whenever $\ba \leq \bb,$ the interval notation $[\ba,\bb]$ or $[\ba, \bb)$ should be interpreted componentwise in the usual way, giving rise to (semi-)rectangles in $\Zd.$   For $t \in \bbZ,$ we have $t\ba = (ta_1,\ldots, ta_d),$ with $\ba + \bb$ and $\ba \bb$ representing the usual sum and componentwise product, and we define $\frac{\ba}{\bh} \in \bbQ^d$ as the componentwise quotient for $\bh \geq \bOne.$  The dot product of two vectors will be denoted $\bu^t \bv$.  Finally, for $\bz \in \bbT^d$, $\bk \in \Zd$ write $\bz^n = (z_1^n,\ldots, z_d^n)$ and $\bz^\bk = (z_1^{k_1}, \ldots, z_d^{k_d})$, so that $\bz^n, \bz^\bk \in \bbT^d.$  With these notations established, we proceed to the next section where we describe the full $\Zd$-shift on a finite alphabet and substitutions of constant length in $\Zd,$ which we call $\bq$-substitutions.


\section{Substitution Dynamical Systems}\label{sds section}

Fix a dimension $d \geq 1.$  An \textit{alphabet} is a finite set $\mcA$ consisting of at least $2$ \textit{letters} which we will frequently denote with symbols $\alpha, \beta, \gamma, \delta.$ Often, we will consider alphabets of the form $\{0,1,\ldots, s-1\},$ and $s$ will always refer to the size of the alphabet.  Consider the collection of all functions from $\Zd \to \mcA,$ which can be identified with the product space $\mcA^{\Zd}$ of all $\Zd$-indexed sequences with values in $\mcA.$  The elements of $\AZd$ are \textit{sequences}, denoting them with letters $\bA, \bB, \bC,$ etc.   Endowing $\mcA$ with the discrete topology, consider the topology of pointwise convergence on $\AZd,$ or equivalently the product topology when viewed as a sequence space.  Addition on $\Zd$ gives rise to a $\Zd$-action of commuting automorphisms, which act by translation on $\AZd$ sending $\bk \mapsto T^\bk$ where $T^\bk : \AZd \to \AZd$ sends $\bA$ to the sequence $T^\bk \bA$ defined by $T^\bk \bA(\bj) := \bA(\bj + \bk).$ We call this action \textit{the shift} and denote it by $T.$  The pair $(\AZd, T)$ is an invertible topological dynamical system, \textit{the full shift,} and we let $\mcB$ denote the $\sigma$-algebra of its Borel measurable sets.

By a \textit{block} (or \textit{word} in the $d=1$ case) we mean a map $\omega$ from a finite subset of $\Zd$ into $\mcA,$ denoting the domain of a block by $\text{supp}(\omega)$, and letting $\mcA^+$ denote the collection of all blocks in $\AZd.$ Here, blocks differ from convention in two significant ways: they need not be contiguous, and they need not start at $\bZero$ index.  We will routinely identify blocks with their \textit{graphs,} or the image of the map $\bj \mapsto (\bj, \omega(\bj)) \in \Zd \times \mcA,$ as this is often convenient and coincides with the tiling perspective of substitutions, as treated by Radin [\ref{radin}]. For $\omega \in \mcA^+$ and $\bA \in \mcA^+ \cup \AZd,$ we say $\omega$ is \textit{extended by} $\bA$ (writing $\omega \leq \bA$) whenever $\bA$ extends $\omega$ as a function into $\mcA$; we say $\omega$ is a \textit{subblock} (or \textit{subword} in the $d=1$ case) of $\bA$ if $T^\bk \omega$ is extended by $\bA,$ for some $\bk \in \Zd.$  If $\omega \in \mcA^+$ is a block, the \textit{cylinder over $\omega$} is the collection
$$[\omega] := \{ \bC \in \AZd : \bC(\bj) = \omega(\bj) \text{ for } \bj \in \text{supp}(\omega) \} = \{ \bC \in \AZd : \omega \leq \bC \}$$
\noindent so that cylinders over blocks $\mcA^+$ correspond to the standard basis for the topology of $\AZd.$  The alphabet $\mcA$ is a discrete compact set, and so $\AZd$ is totally disconnected as every cylinder is both open and closed, and compact by the Tychonoff theorem.  As every $\bA \in \AZd$ is in the intersection of the cylinders $[\bA_n]$ where $\bA_n$ is the subblock of $\bA$ supported in $[-n \bOne, n\bOne],$ this implies $\AZd$ is a perfect set.  Thus, the full shift $(\AZd, T)$ is an invertible topological dynamical system, consisting of a $\Zd$-action of commuting automorphisms acting on a Cantor set of functions $\Zd \to \mcA.$  

A \textit{substitution} is a map $\mcS : \mcA \to \mcA^+,$ replacing each letter by a block.  In the $d=1$ case, substitutions of many types are considered, but we wish to consider only the class of substitutions which generalize substitutions of constant length.  Fix $\bq > \bOne,$ and write $Q := q_1 \cdots q_d = \text{Card}[\bZero,\bq).$ A substitution $\mcS$ is a \textit{$\bq$-substitution} if all the blocks $\mcS(\alpha)$ have $[\bZero,\bq)$ as their common support. 
For such substitutions, we can inflate about $\bZero$ by $\bq$ and subdivide mod $\bOne$ by identifying $\Zd$ with $\bq \Zd + [\bZero,\bq)$, allowing $\mcS$ to replace the letter at each index with the $[\bZero,\bq)$-block associated with it.  This inflate and subdivision process extends $\mcS$ to sequences over arbitrary subsets of $\Zd,$ and gives rise to many arithmetic properties of $\bq$-substitutions, which we discuss in \S \ref{analysis of q sub}.  We now show how to associate a subshift to a given substitution, after which we detail some essential results and preliminaries for the spectral theory of substitution subshifts.

The language $\mcL_\mcS \subset \mcA^+$ of a substitution is the collection of all blocks $\omega \in \mcA^+$ which appear as subwords of $\mcS^n(\gamma)$ for some $n \in \bbN$ and $\gamma \in \mcA.$  Associating blocks with their graphs, this can be thought of as the collection of all finite patterns attainable by the substitution; by definition, it is both shift-invariant and closed under the action of $\mcS.$  The \textit{substitution subshift of $\mcS$} is the collection $X_\mcS$ consisting of those sequences in $\AZd,$ all of whose subblocks appear in the language of $\mcS$.  By the \textit{reduced language} of $\mcS,$ we mean the collection of all blocks in $\mcA^+$ that appear in some word $\bA \in X_\mcS,$ and can be strictly smaller than the language; as this has no effect on the measure theoretic or topological structure of $X_\mcS,$ we will often assume our language is reduced.  

Let $\mcS$ be a $\bq$-substitution on $\mcA$.  As the alphabet is finite, we can always find some $h > 0$ and $\eta: [-\bOne,\bZero] \to \mcA$ so that $(\mcS^h\eta)(\bc) = \eta(\bc)$ for $\bc \in [-\bOne,\bZero]$.  Then $\eta$ acts as a \textit{seed} for the subshift: by construction, the sequence of blocks $\mcS^{nh}\eta$ defined on $[-\bq^{nh},\bq^{nh})$ are nested, as $\mcS^{nh}\eta$ extends $\mcS^{mh}\eta$ whenever $n \geq m.$  As $\AZd$ is endowed with the topology of point-wise convergence and as $\bq > \bOne$, the sequence $\mcS^{nh}\eta$ converges to a unique limit point $\bD_\eta \in \AZd$ which is also in $X_\mcS$ by construction, and so $X_\mcS$ is nonempty for any $\bq$-substitution on $\mcA$.

Thus $X_\mcS$ is a nonempty closed and shift-invariant (by definition) subset of $\AZd$ and $(X_\mcS,T)$ forms a topological subshift of the full shift $\AZd$ called a \textit{substitution dynamical system.}  A useful property of substitution subshifts is their independence of the iterate used: for $n > 0,$ $X_{\mcS^n} = X_\mcS,$ sometimes referred to as \textit{telescope invariance,} see [\ref{cortez and solomyak}, Lemma 2.9].  Note that $\mcS$ restricts from $\AZd$ to a map on $X_\mcS$, and the Borel $\sigma$-algebra for $X_\mcS$ is generated by the cylinders over blocks in the reduced language $\mcL_\mcS.$  As our goal is to study the spectral theory of substitution dynamical systems, we adopt the convention that all cylinders are intersected with $X_\mcS.$  Finally, let $\mcM(X_\mcS,T)$ denote the space of $T$-invariant ($\mu = \mu \circ T^\bk$ for all $\bk \in \Zd$) Borel probability measures on $X_\mcS;$ note that $\mcM(X_\mcS,T)$ is a compact convex set, the extreme points of which are ergodic, see [\ref{walters}].

\subsection{Invariant Measures}\label{invariant measures section}

Consider the vector space $\bbC^\mcA$ of formal linear combinations in the letters of $\mcA,$ represented in the basis $\be_\alpha,$ for $\alpha \in \mcA$, and $\bM_\mcA(\bbC)$ the space of $\mcA \tbump{-4px}{$\times$} \mcA$-indexed matrices.  Given a substitution $\mcS$ on $\mcA,$ its \textit{substitution matrix} $M_\mcS \in \bM_\mcA(\bbC)$ is the nonnegative matrix whose $\alpha,\gamma$ entry is the number of times $\alpha$ appears in the word $\mcS(\gamma).$  This representation is often called the \textit{abelianization of $\mcS$} as it represents the symbols produced by the substitution irrespective of position.  The \textit{expansion} of a substitution is the spectral radius of its substitution matrix; in the case of $\bq$-substitutions, the expansion is $Q = \text{Card}[\bZero,\bq)$ as the substitution matrix is $Q$-column stochastic, see \S \ref{configurations}.  A substitution is \textit{primitive} if for some $n > 0,$ $\alpha$ appears in the block $\mcS^n(\gamma),$ for every $\alpha, \gamma \in \mcA.$   Thus, the substitution matrix of a primitive substitution is a primitive matrix: there exists some $n > 0$ so that $M_{\mcS^n} = M_\mcS^n$ has strictly positive entries.  By the Perron-Frobenius theorem,  the spectral radius of a primitive matrix is a simple eigenvalue with a strictly positive eigenvector $\bu,$ normalized to a probability vector and called the \textit{Perron vector} of $\mcS,$ see [\ref{gantmacher}, Theorem 8.1,2].   

The following result of Michel [\ref{michel}] relates the invariant measures of a primitive substitution subshift to the $Q$-eigenspace of its substitution matrix, see [\ref{radin}, Lemma 1.5] for the $\Zd$ setting.
\begin{thm}[Michel]\label{michel theorem} 
	If $\mcS$ is a primitive $\bq$-substitution, then $\mcM(X_\mcS,T) = \{ \mu \}$ is uniquely ergodic, $\mu \circ \mcS = \frac{1}{Q}\mu$, and if $\bu \in \bbC^\mcA$ is the Perron vector of $M_\mcS,$ then $\mu[\alpha] = u_\alpha$ for $\alpha \in \mcA$.
\end{thm}
\noindent As the collection of cylinders over the language of a substitution generate the Borel $\sigma$-algebra of its subshift, an invariant measure $\mu$ is uniquely determined by the measures of the \textit{initial cylinders} $[\alpha] = \{ \bA \in X_\mcS : \bA(\bZero) = \alpha \}$ for $\alpha \in \mcA$.  In the case of non-primitive substitutions, there will be a nonempty proper subset $\mcA_0 \subsetneq \mcA$ such that $\mcS^h$ restricts to a substitution on the subalphabet $\mcA_0,$ for some $h > 0.$   Using this, we can express a substitution in its \textit{primitive reduced form}:
\begin{prop}
\label{PRF}
	If $\mcS$ is a $\bq$-substitution on $\mcA$, then there is an $h > 0$ and a partition of the alphabet $\mcA = \mcE_1 \sqcup \cdots \sqcup \mcE_K \sqcup \mcT$, where $\sqcup$ denotes disjoint union, so that
	\begin{multicols}{2}
	\begin{itemize}
		\item $\mcS^h: \mcE_j \to \mcE_j^+$ is primitive for $1 \leq j \leq K$
		\item $\gamma \in \mcT$ implies $\mcS^h(\gamma) \notin \mcT^+$
	\end{itemize}
	\end{multicols}
\end{prop}
\noindent Clearly, $K=1$ and $\mcT = \emptyset$ if and only if $\mcS$ is primitive.  We call the partition $\{\mcE_1,\ldots, \mcE_K, \mcT\}$ the \textit{ergodic decomposition of $\mcS,$} its members $\mcE_j$ the \textit{ergodic classes of $\mcS,$} and $\mcT$ its \textit{transient part}, compare [\ref{queffelec}, \S 10.1].  As the transient part $\mcT$ is equal to $\mcA \smallsetminus \fatsqcup \mcE_j,$ it suffices to specify the ergodic classes $\mcE = \{ \mcE_1, \ldots, \mcE_K \},$ and we will do this frequently.  Restricting $\mcS$ to each ergodic class gives the \textit{primitive components of $\mcS,$} and we refer to their collective Perron vectors as the \textit{Perron vectors of $\mcS$}.  The minimal such $h > 0$ satisfying the above proposition is the \textit{index of imprimitivity}; as telescoping has no effect on the subshift, we will assume that all of our substitutions have index of imprimitivity $1$ as this directly impacts the existence of several limits in our analysis.
\begin{exmp}
 Consider the $2$-substitution $\mcS$ on the alphabet $\mcA = \{ 1,2,3,4,5, 6\}$ given by
 $$\mcS : {\tiny \scalebox{0.8}{$\begin{cases} 1 \mapsto 25 & \hspace{12px} 4 \mapsto 43 \\ 2 \mapsto 13 & \hspace{12px} 5 \mapsto 31 \\ 3 \mapsto 52 & \hspace{12px} 6 \mapsto 25 \end{cases}$}}  \tbump{4px}{with} M_{\mcS} = {\tiny \scalebox{0.55}{$\begin{pmatrix} 0 & 1 & 0 & 0 & 1 & 0 \\ 1 & 0 & 1 & 0 & 0 & 0\\ 0 & 1 & 0 & 0 & 1 & 0 \\ 0 & 1 & 1 & 1 & 1 & 0 \\ 1 & 0 & 1 & 0 & 0 & 0 \\ 0 & 1 & 0 & 0 & 1 & 0\end{pmatrix}$}} \tbump{12px}{ and } {\tiny \mcS^2 :  \scalebox{0.8}{$\begin{cases} 1 \mapsto 1331 & \hspace{12px}  4 \mapsto 4352 \\ 2 \mapsto 2552 & \hspace{12px} 5 \mapsto 5225 \\ 3 \mapsto 3113 & \hspace{12px}  6 \mapsto 1331 \end{cases}$}} \tbump{4px}{with} M_{\mcS}^2 = {\tiny \scalebox{0.55}{$\begin{pmatrix} 2 & 0 & 2 & 0 & 0 & 0 \\ 0 & 2 & 0 & 0 & 2 & 0\\ 2 & 0 & 2 & 0 & 0 & 0 \\ 0 & 1 & 1 & 1 & 1 & 0 \\ 0 & 2 & 0 & 0 & 2 & 0 \\ 2 & 0 & 2 & 0 & 0 & 0\end{pmatrix}$}}$$
Thus, its index of imprimitivity is $2$ with ergodic classes  $\{1,3\}$ and $\{2,5\}$ and transient part $\{4,6\}.$  Note that reordering the basis consistent with the ergodic decomposition puts $M_\mcS^2$ into block upper triangular form - the diagonal blocks corresponding to the primitive components.
\end{exmp}
We comment briefly on a few aspects of the substitution matrix and Perron-Frobenius theory, see [\ref{gantmacher}] noting that our matrices are $Q$-column (not row) stochastic.  The spectral radius $\rho$ of $M_\mcS$ is the expansion of the substitution, and Frobenius theory [\ref{gantmacher}, III \S 2 Theorem 2] tells us that all eigenvalues of modulus $\rho$ correspond to roots of unity (times $\rho$).  The index of imprimitivity is then the smallest $h>0$ such that $\lambda^h = Q^h$ for every eigenvalue $\lambda$ of modulus $Q$ for $M_\mcS$, allowing one to find the index of imprimitivity from the eigenvalues of the substitution matrix, see also [\ref{gantmacher}, III \S 5].   A normal form for reducible matrices is described in  [\ref{gantmacher}, III \S 4 (69)], allowing us to write $M_\mcS$ via a permutation of the basis in block form 
\begin{equation}\label{matrix PRF}	M_\mcS ={\tiny \scalebox{0.7}{$\begin{pmatrix} M_1 & & & \tilde{M_1} \\ & \ddots & & \vdots \\ & & M_K & \tilde{M}_K \\ & & & M_\mcT \end{pmatrix}$}} \tbump{0.25in}{with matrices} {\tiny \scalebox{0.9}{$\begin{cases} M_i \in \bM_{\mcE_i}(\bbC)  & \text{for } 1 \leq i \leq K  \\ \tilde{M_i} \in \bM_{\mcE_i \tbump{-2px}{$\times$} \mcT}(\bbC) & \text{for } 1 \leq i \leq K \\ M_\mcT \in \bM_{\mcT}(\bbC) \end{cases}$}}\end{equation}
where all unrepresented blocks are $\bZero.$  Moreover, the diagonal blocks $M_1, \ldots, M_K$ are all primitive when exponentiated by the index of imprimitivity, and so this serves as a proof of the above proposition.   As $M_\mcS$ is $Q$-stochastic by column, an application of [\ref{gantmacher}, III \S 4 Theorem 6] shows the spectral radius of the $M_1,\ldots, M_K$ are all $Q$, whereas that of $M_\mcT$ is strictly less than $Q.$  As the matrices $M_i^h$ are primitive, their $Q^h$-eigenvectors are determined by the Perron vectors of $\mcS$ which by the above representation of $M_\mcS$ will be mutually orthogonal, and proves the following:
\begin{lem}\label{perron decomposition}
	If $\mcS$ is a $\bq$-substitution, the Perron vectors of $\mcS$ span the $Q$-eigenvectors of $M_\mcS$.  All other eigenvalues of the substitution matrix are strictly smaller than $Q$ in modulus.
\end{lem}

A substitution $\mcS$ is \textit{aperiodic} if $X_\mcS$ contains no shift-periodic elements, or if $T^\bk \bA = \bA$ implies $\bk = \bZero$ for all $\bA \in X_\mcS.$  In [\ref{cortez and solomyak}], Cortez and Solomyak use results of Bezuglyi and others in [\ref{BKMS}] characterizing the invariant measures on aperiodic and stationary Bratteli diagrams to extend Michel's theorem \ref{michel theorem} to aperiodic $\bq$-substitutions; our version is a corollary of [\ref{cortez and solomyak}, Theorem 3.8]:
\begin{thm}[Cortez and Solomyak]\label{invariant measures}  The ergodic measures of an aperiodic $\bq$-substitution are the uniquely ergodic measures corresponding to its primitive components.
\end{thm}
\begin{proof}
	 By lemma \ref{perron decomposition} and (\ref{matrix PRF}) above, the eigenvalues corresponding to the transient letters $M_\mcT$ are strictly dominated by the expansion of the substitution.  Thus, the \textit{distinguished eigenvalues,} see [\ref{BKMS}], are those corresponding to $M_\mcS$ restricted to its ergodic classes, which is the substitution matrix of $\mcS$ when restricted to its ergodic classes, and the result follows from [\ref{BKMS}, Corollary 5.6]
\end{proof}
In particular, there are at most $\text{Card}(\mcA)$ ergodic measures for any substitution subshift and the study of invariant measures for $\bq$-substitutions reduces to the primitive case.  By the above, one can express any invariant measure as a convex sum of the uniquely ergodic measures for the primitive components and thus extend the scaling property of theorem \ref{michel theorem} to nonprimitive subshifts.  Moreover, the vector $\bu = (\mu([\alpha]))_{\alpha \in \mcA}$ will be the same convex combination of mutually orthogonal Perron vectors corresponding to the primitive components of $\mcS.$ 

\subsection{Spectral Theory}\label{spectral theory section}

Given a measure $\mu \in \mcM(X_\mcS,T),$ we are interested in the (unitary) translation operator $f \mapsto f \circ T$ on $L^2(\mu),$ the space of complex valued square integrable functions on $X_\mcS,$ the substitution subshift.  For each pair $f,g \in L^2(\mu),$ there is a complex Borel measure $\sigma_{f,g}$ on the $d$-Torus $\bbT^d$ called \textit{the spectral measure for $f, g,$} with Fourier coefficients satisfying the identity
\begin{equation}\label{spectral coefficients} \widehat{\sigma_{f,g}}(\bk) := \int_{\bbT^d} z_1^{-k_1}\cdots z_d^{-k_d} \bump d\sigma_{f,g} = \int_X f \circ T^{-\bk} \bump \overline{g} \bump d\mu  \end{equation}
\noindent Note that $\sigma_{f} := \sigma_{f,f}$ is a positive measure for every nonzero $f \in L^2(\mu)$ by Bochner's theorem; one extends to arbitrary pairs $f,g$ using polarization.  By the spectral theorem for unitary operators, there is a \textit{maximal function} $F \in L^2(\mu)$ such that for any $g \in L^2(\mu),$ the spectral measure for $g$ is absolutely continuous with respect to that of $F,$ or $\sigma_g \ll \sigma_F,$ and every measure $\lambda \ll \sigma_F$ is the spectral measure of some $g \in L^2(\mu).$  Although the maximal function is not unique, their spectral measures have the same \textit{type}: they are mutually absolutely continuous.  We denote the type of $\sigma_F$ by $\sigmaX(\mu),$ the \textit{maximal spectral type} of $(X_\mcS,T,\mu)$.   In light of theorem \ref{invariant measures}, we define the \textit{spectrum of $\mcS$} as the sum of the maximal spectral types of the primitive components, so that if $\mcE_\mcS$ denotes the ergodic measures of $X_\mcS$:
\begin{equation}\label{sigmaX def} \textstyle \sigmaX := \sum_{\mu \in \mcE_\mcS}  \sigmaX(\mu) \end{equation}
\noindent noting that this is well defined up to measure equivalence.  Using identity (\ref{spectral coefficients}) and the Krein-Milman theorem, one can see that $\sigma_{f,g}(\mu) \ll \sigmaX$ for every $\mu \in \mcM(X_\mcS,T)$ and $f,g \in L^2(\mu),$ justifying the definition.   By Lebesgue decomposition, $\sigmaX$ can be separated into \textit{pure types}, or its \textit{discrete, singular} and \textit{absolutely continuous} components on $\bbT^d$ and is essentially our goal.

In the case of $\bq$-substitutions and their substitution subshifts, the discrete component is well understood: it is a multiplicative subgroup of $\bbT^d$ corresponding to $\bq$ and the \textit{height} of a substitution, as described by Dekking in [\ref{dekking}] (and extended to $\Zd$ by Frank in [\ref{frank zd}]) which also includes a complete classification of the pure discrete case: see section \ref{height} for more details.  Study of the continuous spectrum is based on the work of Queff\'elec, which relates the maximal spectral type of a substitution to \textit{correlation measures} $\sigma_{\alpha \beta}$ for $\alpha, \beta \in \mcA,$ the spectral measures with respect to translation for the indicator functions of initial cylinders $[\alpha]$ and $[\beta]$.  

Due to Dekking's work, our interest is primarily in the continuous spectrum of $\mcS,$ and distinguishing the purely singular case from those with Lebesgue components in their spectrum.   Note that this is a nontrivial problem: the  Thue-Morse substitution is an example of a substitution with purely singular spectrum, possessing both discrete and continuous components, whereas the Rudin-Shapiro substitution has a discrete and absolutely continuous component, with no singular continuous spectrum.  Moreover, one can form a \textit{substitution product} of Thue-Morse and Rudin-Shapiro to obtain a substitution with discrete, singular continuous, and absolutely continuous components in its spectrum, which we consider in example \ref{substitution product example}; see also [\ref{BGG}, \S 2].   For higher dimensional examples, we have the work of Baake and Grimm which shows a large class of substitutions on $2$ symbols to be purely singular to Lebesgue spectrum, see [\ref{baake and grimm}], as well as Frank's paper [\ref{frank lebesgue}] describing a collection of substitutions with Lebesgue spectral component.  

\section{Aperiodic $\bq$-Substitutions}\label{analysis of q sub}
	
	The goal of this section is to represent $\bq$-substitutions as a \textit{configuration} of \textit{instructions}, describing how the substitution replaces symbols by location, and exposing their arithmetic properties.  Moreover, we describe a representation of $\bq$-substitutions in matrix algebras which, along with recognizability properties afforded by aperiodicity, allows us to iteratively compute the invariant measure of an arbitrary cylinder of $X_\mcS,$ or the frequency with which any finite pattern appears in the subshift.  This allows us to explicitly compute Fourier coefficients (\ref{fourier coefficients}) of the correlation measures, as $\widehat{\sigma_{\alpha \beta}}(\bk)$ is the measure of a cylinder specified by the block $\bZero \to \alpha$ and $\bk \mapsto \beta$.


\subsection{Arithmetic Base $\bq$ in $\Zd$}\label{arithmetic}

We take a moment to establish some basic arithmetic notions, most of which are consequences of the classical division algorithm on $\bbZ$ applied to each coordinate.  Essentially, we are just formalizing base $q$ arithmetic in each coordinate separately.  Recall that for $\ba \geq \bZero$
$$[\bZero,\ba) = \big\{ \bj  \in \Zd : 0 \leq j_i < a_i \text{ for } 1 \leq i \leq d \big\}$$
\noindent Fix $\bq > \bOne,$ then for every $n \in \bbN$ and each $\bk \in \Zd,$ the division algorithm modulo $\bq^n$ applied componentwise in $\Zd$ provides a unique pair $\qr{\bk}_n  \in [\bZero,\bq^n) \text{ and }\qq{\bk}_n \in \Zd$ satisfying
$$\bk = \qr{\bk}_n + \qq{\bk}_n \bq^n$$
\noindent so that $\qr{\cdot}_n: \Zd \to [\bZero,\bq^n)$ is the \textit{remainder,} and $\qq{\cdot}_n: \Zd \to \Zd$ the \textit{quotient}, mod $\bq^n$.  Writing
$$\bk_n := \qr{\qq{\bk}_n}_1 = \lfloor\qr{\bk}_{n+1}\rfloor_n$$
\noindent for $\bk \in \Zd$ and $n \geq 0$ gives a unique \textit{digit sequence} $(\bk_j)_{j \in \bbN} \in [\bZero,\bq)^\bbN$ such that for $n \geq 0$
$$\bk = \bk_0 + \bk_1 \bq + \cdots + \bk_{n-1}\bq^{n-1} + \qq{\bk}_n \bq^n$$
\noindent referred to as the \textit{$n$-th $\bq$-adic expansion} of $\bk,$ and we call $\bk_n$ the \textit{$n$-th digit} of $\bk.$ Thus, for $\bk \in \Zd,$ $\qr{\bk}_n$ can be represented by the first $n$ digits of $\bk$, and $\qq{\bk}_n$ by the digits starting at the $n$-th place and above.  The \textit{power} $\mfp:= \mfp(\bk)$ of $\bk$ is the minimal $p \geq 0$ with $\bk \in (-\bq^p,\bq^p)$ and is such that $\qq{\bk}_n = \qq{\bk}_{\mfp(\bk)}$ and $\bk_n = \bk_{\mfp(\bk)}$ for $n \geq \mfp(\bk)$.   For $n \geq 0,$ and $\bk \in \Zd,$ the \textit{$n$-carry set for $\bk$} is 
\begin{equation}\label{carry sets}
	\Delta_n(\bk) := \big\{ \bj \in [\bZero,\bq^n) : \qq{\bj+\bk}_n  \neq \bZero \big\}
\end{equation}
\noindent Thus, for $n \geq \mfp(\bk),$ the collection $\Delta_n(\bk) + \bq^n \Zd$ consists of those integers which require a \textit{carry or borrow} operation at the $n$-th place when $\bq$-adicly added to $\bk$, and has a useful statistical property:
\begin{lem}\label{small carries}
The frequency of carries at the $n$-th place of $\bq$-adic addition with $\bk$ goes to $0,$ or
$$\lim_{n \to \infty} \frac{1}{Q^n} \text{Card} \bump \Delta_n(\bk) = 0$$
\noindent for every $\bk \in \Zd,$ and this convergence is exponentially fast.
\end{lem}
\begin{proof}
	As the cardinality of $\Delta_n(\bk)$ does not depend on the signs of the components of $\bk$ (they merely switch the side of the semirectangle thru which the translation occurs) we can assume without loss of generality that $\bk \in \bbN^d$.  Then for $n \geq \mfp := \mfp(\bk),$ we have $\Delta_n(\bk) \subset \Delta_n(\bq^\mfp)$ and as $\Delta_n(\bq^\mfp) = [\bZero,\bq^n) \smallsetminus [\bZero,\bq^n - \bq^\mfp)$ we compute the cardinality and we obtain	
	$$\textstyle \text{Card}\bump\Delta_n(\bk) = Q^n - {\prod}_{i=1}^n(q_i^n - q_i^\mfp) \leq Q^n  - Q^n \prod_{i=1}^d \big( 1 - \nicefrac{q_i^\mfp}{q_i^n} \big) $$	
\noindent as $Q = q_1 \cdots q_d$.  Dividing by $Q^n$ and letting $n \to \infty$ gives the desired result.
\end{proof}

We briefly describe the above in the context of a $\bq$-substitution $\mcS$.  For every $n \geq 0,$ tile $\Zd$ with the superblocks $[\bZero,\bq^n)$ placed at each location of the $\bq^n \Zd$ lattice; these represent the domains of the translated superblocks $T^\bj \mcS^n\gamma$. The $n$-th quotient $\qq{\bk}_n$ indicates the superblock in the $\bq^n\Zd$ lattice containing $\bk$, and the $n$-th remainder $\qr{\bk}_n$ tells us where $\bk$ sits inside that superblock.  The power $\mfp(\bk)$ represents the smallest $p$ such that $\bk$ falls into a superblock of size $\bq^p$ attached to the origin at one of its corners.   As $\qq{\bj + \bk}_n$ represents the superblock of size $\bq^n$ containing $\bj + \bk$, the carry set $\Delta_n(\bk)$ corresponds to the locations within a superblock which will leave that superblock when translated by $\bk$, the proportion of which is negligible as $n \to \infty$.


\subsection{Instructions and Configurations}\label{configurations}

Let $\mcS$ be a $\bq$-substitution on $\mcA.$  For each $\bj \in [\bZero,\bq),$ the map sending $\gamma$ to $\mcS\gamma(\bj),$ the $\bj$-th letter of the word $\mcS(\gamma),$ is a map $\mcA \to \mcA$ called the \textit{$\bj$-th instruction of $\mcS,$} and is denoted by $\mcR_\bj$ for $\bj \in [\bZero,\bq).$    For $n \geq 0$ and $\bj \in [\bZero,\bq^n)$, denote the $\bj$-th instruction of $\mcS^n$ by $\mcR_\bj^{(n)}$ (distinct from $\mcR_\bj^n$), extending to $\bj \in \Zd$ by reducing modulo $\bq^n$ as this has no effect on the substitution or its iterates.  We call the $\mcR_\bj^{(n)}$ for $\bj \in \Zd$ and $n \geq 0$ the \textit{generalized instructions} of $\mcS$ as they permit an alternate characterization of $\bq$-substitutions which exposes their arithmetic properties, see also [\ref{queffelec}, \S 5.1].  
\begin{prop}\label{qsub} 

Let $\mcS$ be a $\bq$-substitution on $\mcA.$  For every $n \in \bbN, \bj \in \Zd$, and $\bA \in \AZd$
$$(\mcS^n\bA)(\bj)  = \mcR_{\bj_0}  \mcR_{\bj_1} \cdots \mcR_{\bj_{n-1}}(\bA(\qq{\bj}_n)) \tbump{12px}{ and }  \mcR_\bj^{(n)} = \mcR_{\bj_0}  \mcR_{\bj_1} \cdots \mcR_{\bj_{n-1}}$$
\noindent where $\qq{\cdot}_n$ is the quotient mod $\bq^n$, and $\bj_i \in [\bZero,\bq)$ the $\bq$-adic digits of $\bj,$ as is our convention, and
$$T^\bj \mcS^n = T^{\bj_0} \mcS \cdots T^{\bj_{n-1}} \mcS  T^{\qq{\bj}_n} = T^{\qr{\bj}_n} \mcS^n T^{\qq{\bj}_n}$$
\end{prop} 
\begin{proof}
\noindent Fix a $\bq$-substitution $\mcS$ and $\bA \in \AZd.$  The proof of the first statement follows by a simple inductive argument on the $n=1$ case.  The sequence $\mcS \bA$ is obtained by concatenating the blocks $\mcS(\bA(\ba))$ at the coordinates $\ba \bq$ for $\ba \in \Zd.$  It follows that the letter in the $\bb + \ba \bq$-th position of $\mcS \bA$ comes from the $\bb$-th letter of $\mcS(\bA(\ba))$, so that as $\mcR_\bb(\alpha) = \mcS(\alpha)_\bb$ for $\bb \in [\bZero,\bq),$ we have $\mcS\bA(\bb+\ba \bq) = \mcR_\bb(\bA(\ba)).$  This proves the $n=1$ case, as $\bb = \qr{\bb + \ba \bq}_1$ and $\ba = \qq{\bb + \ba \bq}_1.$  Writing $\mcS^n\bA = \mcS(\mcS^{n-1}\bA)$ gives the inductive step necessary to prove the first result.   The second follows by taking $\bj \in [\bZero,\bq^n)$, so that $\mcR_\bj^{(n)}(\bA(\bZero)) = \mcS^n(\bA(\bZero))(\bj) = \mcR_{\bj_0} \circ \cdots \circ \mcR_{\bj_{n-1}}(\bA(\bZero))$.

Finally, fix $n \geq 0$ and $\bj \in \Zd$.  For each $\bk \in \Zd$ and $\bA \in \AZd$, use the above to write
$$\big( \mcS^n T^\bj \bA \big)(\bk) = \mcR_{\bk}^{(n)}\big( T^\bj \bA(\qq{\bk}_n) \big) = \mcR_{\bk}^{(n)}\big( \bA(\bj + \qq{\bk}_n) \big) = \big( \mcS^n \bA \big)(\bj \bq^n + \bk) = \big( T^{\bj\bq^n} \mcS \bA \big)(\bk)$$
so that $\mcS^n T^\bj = T^{\bj \bq^n} \mcS^n$.  Writing $\bj = \bj_0 + \qq{\bj}_1 \bq$ we can commute $T$ with $\mcS$ one at a time writing
$$T^{\bj_0 + \qq{\bj}_1 \bq}\mcS^n = T^{\bj_0} \mcS \circ T^{\qq{\bj}_1} \mcS^{n-1} = T^{\bj_0} \mcS \circ T^{\qr{\qq{\bj}_1}_1} \mcS \circ T^{\qq{\bj}_2} \mcS^{n-2} = T^{\bj_0} \mcS \circ T^{\bj_1} \mcS \circ T^{\qq{\qq{\bj}_1}_1} \mcS^{n-2}$$
as $\qr{\qq{\bj}_1}_1 = \bj_1$ and $\qq{\qq{\bj}_1}_1 = \qq{\bj}_2$.  One then iterates the above to complete the proof.
\end{proof}
As we are only concerned with $\bq$-substitutions, we prefer proposition \ref{qsub} over the inflate and subdivide perspective, as it conveniently describes how translation and substitution interact.  If the instructions are all bijections on $\mcA$, we say $\mcS$ is a \textit{bijective substitution;} if they all commute with each other, we say $\mcS$ is a \textit{commutative} substitution.  The following example is classical and is an example of an aperiodic bijective commutative $\bq$-substitution; in theorem \ref{abc theorem}, we show that all such substitutions have spectrum singular to Lebesgue measure on $\bbT^d$.
\begin{exmp}\label{TM example}
	The Thue-Morse substitution is a $2$-substitution on the alphabet $\{0,1\}$ given by
	$$\small{\tau: \begin{cases} 0 \longmapsto 01 \\ 1 \longmapsto 10 \end{cases} \tbump{12px}{ with instructions } \mcR_0 = \begin{cases} 0  \longmapsto  0 \\ 1  \longmapsto 1\end{cases} \tbump{12px}{ and } \mcR_1 = \begin{cases} 0  \longmapsto 1 \\ 1  \longmapsto 0 \end{cases}}$$
	\noindent as $\tau$ is the identity in the first position, and the transposition in the second.  As $12 = 0 \cdot 2^0 + 0 \cdot 2^1 + 1 \cdot 2^2 + 1 \cdot 2^3,$ we have $\mcR_{12}^{(4)} = \mcR_0 \mcR_0 \mcR_1 \mcR_1 = \mcR_0$ as $\mcR_1^2 = \mcR_0$ is the identity, and so $(\tau^4\alpha)_{12} = \alpha$ for $\alpha = 0,1$ - note that $\tau^4(0) = 01101001100101101001011001101001$.   In this way, one can see that the generalized instructions $\mcR_\bj^{(n)}$ for $\tau$ are also the identity or transposition according to the parity of $1$'s in the binary expansion of $\bj$ below the $n$-th place.
\end{exmp}

We would like to distinguish those aspects of a $\bq$-substitution which depend on its collection of instructions (counted with multiplicity), from the arrangement of those instructions within the substitution by position. A \textit{$\bq$-configuration of instructions on $\mcA$} is a map $\mcR : [\bZero,\bq) \to \mcA^\mcA$ which assigns to every $\bj \in [\bZero,\bq)$ an instruction $\mcR_\bj: \mcA \to \mcA.$  By proposition \ref{qsub}, every $\bq$-configuration determines a $\bq$-substitution, and conversely by writing $\mcR_\bj(\gamma) := \mcS\gamma(\bj)$.  We will always represent the instructions of a substitution with the symbol of its configuration $\mcR: \bj \mapsto \mcR_\bj,$ the $\bj$-th instruction.  Using the generalized instructions as in proposition \ref{qsub}, we can extend any $\bq$-configuration $\mcR$ representing $\mcS$ to a $\bq^n$-configuration $\mcR^{(n)}$ representing $\mcS^n$.  Two substitutions $\mcS$ and $\tilde{\mcS}$, with configurations $\mcR$ and $\tilde{\mcR}$, are \textit{configuration equivalent} if there is a bijection $\iota$ of their domains for which $\tilde{\mcR}\circ \iota = \mcR$, so that they have the same instructions with different arrangements.  Properties of a substitution which are shared by all configuration equivalent substitutions are called \textit{configuration invariants}.  
	
For a matrix $\bR \in M_\mcA(\bbC),$ the space of $\mcA\times\mcA$-matrices with complex coefficients, observe that, if we write $\bR = (R_{\alpha, \gamma} )_{\alpha, \gamma \in \mcA}$ then $R_{\alpha, \gamma} =  \be_\alpha^t \bR \be_\gamma$  where $\be_\alpha$ is the $\alpha^\text{th}$-coordinate vector.  As we are representing $\mcA$ in $\bbC^\mcA$ by sending $\alpha \mapsto \be_\alpha,$ this induces a representation of an instruction $\mcR : \mcA \to \mcA$ by its \textit{instruction matrix} $\mcR \in \bM_\mcA(\bbC)$.   Formally, we can view an instruction as a $\bOne$-substitution on $\mcA$ and compute its substitution matrix $\mcR \in \bM_\mcA(\bbC),$ this is equivalent.  One can compute the instruction matrix from its instruction (or conversely) via the relation 
$$\alpha = \mcR(\gamma) \iff  \be_\alpha^t \bR \be_\gamma =\mcR_{\alpha,\gamma} =  1 \tbump{12px}{equivalently} \alpha \neq \mcR(\gamma) \iff   \be_\alpha^t \bR \be_\gamma =\mcR_{\alpha,\gamma} = 0$$
\noindent as instruction matrices have coefficients $0$ or $1$.   Moreover, this representation carries composition of functions into matrix multiplication so that, as in proposition \ref{qsub},  the generalized instruction matrices are products of the instruction matrices.  Note that the substitution matrix is the sum of its instruction matrices: $M_\mcS = \sum_{\bj \in [\bZero,\bq)} \mcR_\bj,$ and so $M_{\mcS^n} = M_\mcS^n$.  As they represent functions, the instruction matrices are naturally column-stochastic, and thus substitution matrices of $\bq$-substitutions are $Q$-column stochastic: their column sums are $Q$.  Note that, as matrix addition is commutative, the substitution matrix is also a configuration invariant, as are the measures of the \textit{initial cylinders} $[\alpha]$ for $\alpha \in \mcA,$ by Michel's theorem \ref{michel theorem}.
\begin{exmp}\label{TM instructions}
	For the Thue-Morse substitution, the instruction and substitution matrices are 
	$$\small{\mcR_0 = \begin{pmatrix} 1 & 0 \\ 0 & 1 \end{pmatrix} \tbump{12px}{ and } \mcR_1 = \begin{pmatrix} 0 & 1 \\ 1 & 0 \end{pmatrix} \tbump{24px}{ so that } M_\tau = \begin{pmatrix} 1 & 1 \\ 1 & 1 \end{pmatrix}}$$
	\noindent and the substitution matrix is (row) stochastic, typical in the bijective case.  As a positive (thus primitive) matrix, it has a unique probability eigenvector $\bu = (\nicefrac{1}{2}, \nicefrac{1}{2})$, its Perron vector.  
\end{exmp}
We will always conflate the notation for an instruction with its instruction matrix, extending this to the generalized instructions as well, and so $\mcR_\bj^{(n)}$ can be interpreted as belonging to $\mcA^\mcA$ or $\bM_\mcA(\bbC)$ interchangeably based on context.   We now describe a product on $\bq$-substitutions: given two alphabets $\mcA$ and $\tilde{\mcA},$ let $\mcA \tilde{\mcA}$ denote their \textit{product alphabet}, or those pairs $\alpha \tilde{\gamma}$ with $\alpha \in \mcA, \tilde{\gamma} \in \tilde{\mcA}.$  
\begin{defn}\label{substitution product}
	Let $\mcS, \tilde{\mcS}$ be $\bq$-substitutions on alphabets $\mcA, \tilde{\mcA}$ respectively.  Their \textit{substitution product} $\mcS \otimes \tilde{\mcS}$ is the $\bq$-substitution on $\mcA \tilde{\mcA}$ with configuration $\mcR \otimes \tilde{\mcR}$ whose $\bj$-th instruction is
	$$(\mcR \otimes \tilde{\mcR})_\bj : \mcA \tilde{\mcA} \to \mcA  \tilde{\mcA} \tbump{8px}{ with } (\mcR \otimes \tilde{\mcR})_\bj : \alpha \tilde{\gamma} \longmapsto \mcR_\bj(\alpha) \tilde{\mcR}_\bj(\tilde{\gamma}),$$
	\noindent obtained by concatenating the configurations $\mcR$ and $\tilde{\mcR}$ of the respective substitutions.
\end{defn}
Whenever $\tilde{\mcA} = \mcA,$ we write $\mcA \mcA = \mcA^2$ and, if $\tilde{\mcS} = \mcS$, the substitution $\mcS \otimes \mcS$ is called the \textit{bisubstitution} of $\mcS$ and its substitution matrix $C_\mcS := M_{\mcS \otimes \mcS}$ the \textit{coincidence matrix} of $\mcS$, after Queff\'elec, see [\ref{queffelec}, \S 10].  The substitution product is always aperiodic stable as periodic sequences in the hull of a substitution product would necessarily be periodic in both factors.   On the other hand, it is not always primitive stable: the bisubstitution of a primitive substitution is in general itself not primitive.  Using the primitive reduced form of proposition \ref{PRF}, however, it can always be made primitive on its ergodic classes by telescoping appropriately.  In \S  \ref{substitution product section} we discuss an example of a primitive substitution product of Thue-Morse and Rudin-Shapiro, illustrating an interesting relationship that can hold between the spectrum of a substitution product and that of its factors.  

We offer some justification for the notation: given two matrices $A, B \in M_\mcA(\bbC),$ their \textit{Kronecker product} is the matrix $A \otimes B \in M_{\mcA^2}(\bbC)$  whose $(\alpha \beta, \gamma \delta) \in \mcA^2 \times \mcA^2$ entry is $A_{\alpha \gamma} B_{\beta \delta}$ and $\be_\alpha \otimes \be_\beta = \be_{\alpha \beta}$ is the standard unit vector in $\bbC^{\mcA^2}$ corresponding to the word $\alpha \beta.$   With this notation, observe that the coefficients of the Kronecker product become
\begin{equation}\label{tensor conjugation} (A \otimes B)_{\alpha \beta, \gamma \delta} = \be_{\alpha \beta}^t A \otimes B \be_{\gamma \delta} = ( \be_\alpha^t A \be_\gamma ) ( \be_\beta^t B \be_\delta ) = A_{\alpha \gamma} B_{\beta \delta}\end{equation}
\noindent This implies the Kronecker products \textit{mixed product property} 
\begin{equation}\label{mixed product property} (A \otimes B)(C \otimes D) = (AC)\otimes(BD) \end{equation}
\noindent Using this, one checks that the (generalized) instruction matrices of a substitution product are Kronecker products of the corresponding (generalized) instruction matrices of its factors.  Moreover, for every $n \geq 0,$ integers $\bj,\bk \in \Zd$ and letters $\alpha, \beta, \gamma, \delta \in \mcA$, the above gives
\begin{equation}\label{tensor selector}
	\alpha  = \mcR_{\bj}^{(n)}(\gamma) \tbump{4px}{and} \beta = \mcR_\bk^{(n)}(\delta) \tbump{16px}{$\iff$}  \be_{\alpha \beta}^t \mcR_{\bj}^{(n)} \otimes \mcR_{\bk}^{(n)} \be_{\gamma \delta} = 1
\end{equation}
\noindent so that for us the Kronecker product is just formalizing the conjunction \textit{and} within a linear algebraic context, and will be used as above to represent simultaneous conditions on substitutions.   For more discussion of the Kronecker product, see [\ref{HJ}, \S 4].  We work out some of the above details in the case of the Thue-Morse example, for clarification.
\begin{exmp}\label{TM bisub}
	For the Thue-Morse substitution $\tau$, the bialphabet and bisubstitution are
	$$\mcA^2 = \{00, 01, 10, 11\} \tbump{12px}{ and } \tau \otimes \tau : \scalebox{0.8}{${\small\begin{cases} 00 \mapsto 00 \, 11, & 01 \mapsto 01 \, 10 \\ 11 \mapsto 11 \, 00, & 10 \mapsto 10 \, 01\end{cases}}$}$$
	\noindent and so the instruction matrices for $\tau \otimes \tau$ (with basis of $\bbC^{\mcA^2}$ ordered $\be_{00}, \be_{01}, \be_{10}, \be_{11}$) are
	$${\tiny \mcR_0 \otimes \mcR_0 = \begin{pmatrix} 1 & \\ & 1 \end{pmatrix} \otimes \begin{pmatrix} 1 & \\ & 1 \end{pmatrix} = \scalebox{0.75}{ $ \begin{pmatrix} 1 & & & \\ & 1 & & \\ & & 1 & \\ & & & 1 \end{pmatrix}$} \tbump{12px}{ {\normalsize and} } \mcR_1 \otimes \mcR_1 = \begin{pmatrix} & 1 \\ 1 & \end{pmatrix} \otimes \begin{pmatrix}  & 1\\ 1 &  \end{pmatrix}  = \scalebox{0.75}{$\begin{pmatrix} & & & 1 \\ & & 1 & \\ & 1 & & \\ 1 & & & \end{pmatrix}$} }$$
	and the tensor products $\mcR\otimes \tilde{\mcR}$ of instruction matrices can be computed by \textit{placing a copy} of $\tilde{\mcR}$ at every $1$ in $\mcR,$ and $\bZero \in \bM_\mcA(\bbC)$ at every $0$ in $\mcR$, see also example \ref{TM sigma}.  The coincidence matrix is
	$$C_{\tau} = M_{\tau \otimes \tau} = \mcR_0 \otimes \mcR_0 + \mcR_1 \otimes \mcR_1 = \scalebox{0.75}{$\tiny{\begin{pmatrix} 1 & & & 1 \\ & 1 & 1 & \\ & 1 & 1 & \\ 1 & & & 1 \end{pmatrix}}$} \tbump{12px}{ which is similar to } \scalebox{0.75}{$\tiny{\begin{pmatrix} 1 & 1  & & \\ 1 & 1 & & \\ &  & 1 & 1\\  & & 1 & 1 \end{pmatrix}}$}$$
	\noindent  by reordering the basis to $\be_{00}, \be_{11}, \be_{01}, \be_{10}$.  Thus, the primitive reduced form of $\tau \otimes \tau$ is two copies of $\tau,$ one each on the alphabets $\mcE_0 = \{00,11\}$ and $\mcE_1=\{01,10\}$ which together form the ergodic classes of the bisubstitution; there is no transient part, as is always the case for bijective substitutions - their instructions are invertible, so every letter is in a closed orbit.
\end{exmp}
We now describe how aperiodicity of the subshift gives rise to useful topological structure allowing us to exploit the arithmetic properties of $\bq$-substitutions to study their measure theoretic structure and ultimately their spectral theory.


\subsection{Aperiodicity and the Subshift}\label{aperiodicity}

Recall that a $\bq$-substitution is aperiodic if $X_\mcS$ contains no periodic points; in the $d=1$ case, there is a useful criteria for verifying aperiodicity in primitive $q$-substitutions, based on a result of Pansiot [\ref{pansiot}, Lemma 1].  The specific advantage in the $\bbZ$ setting is provided by the following equivalence: a primitive substitution ($d=1$) is periodic if and only if its subshift is finite, and hence has a word of finite length which generates the entire language.  By a ($\mcS$-)\textit{neighborhood of $\alpha,$} we mean a word $\gamma \alpha \delta$ which appears in the reduced language of $\mcS,$ or the set of all subwords appearing in some substitution sequence $\bA \in X_\mcS.$ 
\begin{lem}[Pansiot's Lemma]\label{pansiot lemma}
	 A primitive $q$-substitution (the $\bbZ$ case) which is one-to-one on $\mcA$ is aperiodic if and only if $\mcS$ has a letter with at least two distinct neighborhoods.
\end{lem}

As $\mcS: \mcA \to \mcA^+$ is a map on a finite set, injectivity is trivial to verify, and primitivity can easily be checked by taking powers of the substitution matrix, or by iterating the substitution.  One then examines the iterates $\mcS^{nh}(\gamma)$ for neighbor pairs ($h$ is the index of imprimitivity); as soon as a letter with distinct neighborhoods appears, aperiodicity is verified.  Unfortunately, we have no reasonable criteria for extending Pansiot's result to the case $d > 1,$ as it relies on the fact that periodicity is equivalent to finiteness of the hull for $d = 1$, which is false in higher dimensions.
\begin{exmp}\label{TM aperiodic}
	Consider the Thue-Morse substitution of example \ref{TM example}.  It is primitive, and as $\tau^2(0) = 0110$ and $\mcR_0$ is the identity, $1$ has multiple neighborhoods appearing in the language: $011$ and $110$.  As it is evidently injective on $\mcA$, it is therefore aperiodic.  
\end{exmp}	
We are interested in aperiodic substitutions as they allow for \textit{unique local desubstitution} - a notion made precise by the following result.  Originally proven by Moss\'e [\ref{mosse}] in the one-dimensional case, and extended by Solomyak [\ref{solomyak aperiodicity}] to self-similar $\bbR^d$ tilings, it is stated here for $\bq$-substitutions. 
\begin{thm}[Moss\'e (96), Solomyak (98)]
\label{aperiodicity implies recognizability}
	A $\bq$-substitution $\mcS$ on $\mcA$ is aperiodic if and only if for every $\bA \in X_\mcS$ there exists a unique $\bk \in [\bZero,\bq)$ and $\bB \in X_\mcS$ with $T^{\bk} \mcS(\bB) = \bA.$
\end{thm} 
\noindent Known as \textit{recognizability}, the above property gives conditions on the language permitting a convenient topological description of the subshift.  Iterating theorem \ref{aperiodicity implies recognizability} gives us cancellation laws
\begin{equation}\label{ifs cancellation} T^\bj \mcS^n(\bA) = T^\bk \mcS^n(\bB) \tbump{20px}{$\iff$} \qr{\bj}_n = \qr{\bk}_n \tbump{4px}{ and } T^{\qq{\bj}_n} \bA = T^{\qq{\bk}_n} \bB \end{equation}
for $n \geq 0,$ $\bj, \bk \in \Zd$ and $\bA, \bB \in X_\mcS$, using proposition \ref{qsub}, and letting us express the subshift  as
\begin{equation}\label{ifs attractor}
X_\mcS = {\bigsqcup}_{\bj_0 \in [\bZero,\bq)} T^{\bj_0}\mcS(X_\mcS) = {\bigsqcup}_{\bj \in [\bZero,\bq^n)} T^{\bj_0}  \mcS \circ  \cdots \circ T^{\bj_{n-1}}\mcS(X_\mcS)  = {\bigsqcup}_{\bj \in [\bZero,\bq^n)} T^\bj \mcS^n (X_\mcS) 
\end{equation}
\noindent where ${\sqcup}$ denotes a disjoint union.   Consider the substitutions $T^\bk \mcS$, each mapping $\mcA \to \mcA^{\bk + [\bZero,\bq^n)}$ for some $\bk \in \Zd$.  They are $\bq$-substitutions in every sense \textit{except} that their substituted blocks lie at $\bk$ as opposed to the origin, and as $\mcS$ is $\bq$-expansive on the supports of blocks, these substitutions are all contractions (nonstrict, consider distinct fixed points) on the full shift.  Thus, we can interpret the above as describing $X_\mcS$ as the attractor for an iterated function system whose $n$-th iterates are given by $T^\bj \mcS^n = T^{\bj_0} \mcS \circ \cdots \circ T^{\bj_{n-1}} \mcS$ for $\bj \in [\bZero,\bq^n)$ with expansion $\bj_0 + \cdots + \bj_{n-1} \bq^{n-1}$.

We now consider the orbit of the cylinder sets $[\eta]$ for $\eta \in \mcA^+$ (a topological basis) under this iterated function system in order to study the Borel structure of the substitution subshift.  We would like to arrange them by degree, and so refer to the sets $T^\bj \mcS^n[\eta]$ for $\eta \in \mcA^+$, and $\bj \in [\bZero,\bq^n)$ as the \textit{$n$-th iterated cylinders} of $X_\mcS$.   We are interested in how these iterated cylinders are nested by degree, and as $T^\bj \mcS^n[\eta] \subset [\omega] \iff \omega \leq T^\bj \mcS^n(\eta)$, this question can also be answered using blocks: if $\omega \leq T^\bj \mcS^n(\eta)$ for $\bj \in [\bZero,\bq^n)$ and $\eta \in \mcA^+$, then for $\bk \in \supp(\omega)$ proposition \ref{qsub} gives us
$$\omega(\bk) = (T^\bj\mcS^n\eta)(\bk)  = (\mcS^n\eta)(\bj+\bk) = \mcR_{\bj+\bk}^{(n)} \big(\eta(\, \qq{\bj+\bk}_n)\big)$$
\noindent so that $\omega \leq T^\bj \mcS^n(\eta)$ implies $\qq{\bj + \supp(\omega)}_n \subset \supp(\eta).$  If $\omega \leq T^\bj \mcS^n(\eta)$ and $\eta \leq \eta'$, then $\omega \leq T^\bj \mcS^n(\eta'),$ and so it suffices to specify the smallest $\eta \in \mcA^+$ with $\omega \leq T^\bj \mcS^n(\eta)$ for some $\bj \in [\bZero,\bq^n)$.
\begin{defn}\label{inverse substitution}
For $\omega \in \mcA^+$ and $n \geq 0,$ the \textit{$n$-th desubstitutes} of $\omega$ are given by
$$\mcS^{-n}(\omega) := \left \{ (\bj, \eta) \in [\bZero,\bq^n)\times \mcA^+ \tbump{1px}{$:$} \supp(\eta) = \qq{\,\bj + \supp(\omega)}_n \text{ and } \omega \leq T^{\bj}\mcS^n(\eta) \right \}$$
\end{defn}
\noindent As $\eta \leq \eta' \implies [\eta'] \subset [\eta]$, the pairs $(\bj, \eta) \in \mcS^{-n}(\omega)$ give the largest cylinders $[\eta]$ which are carried into $[\omega]$ in $n$-steps, with $\bj$ giving the sequence of maps as in proposition \ref{qsub}, and so describe how the iterated cylinders fit within one another, along with the Borel structure of $X_\mcS$.
\begin{cor}\label{topological structure}
	Let $\mcS$ be an aperiodic $\bq$-substitution on $\mcA$.  For every $n \geq 0$
	\begin{enumerate}
		\item the maps $T^\bj\mcS : X_\mcS \longrightarrow X_\mcS$ are embeddings for all $\bj \in \Zd$
		\item the \textit{$n$-th iterated initial cylinders} $\mcP_n = \big\{T^\bj \mcS^n[\alpha] \tbump{1px}{$:$} \bj \in [\bZero,\bq^n), \hspace{2px} \alpha \in \mcA \big\}$ partition $X_\mcS$
		\item iterates of the $n$-th desubstitutes partition the cylinders: $[\omega] = {\bigsqcup}_{(\bj,\eta) \in \mcS^{-n}(\omega)} \,T^\bj \mcS^n[\eta]$
	\end{enumerate} 
\end{cor}
\begin{proof}
As $\mcS: X_\mcS \to \mcS(X_\mcS)$ is a continuous surjection from a compact space to a Hausdorff space, the first claim follows as $T^\bj \mcS^n$ is injective: each $\bB \in X_\mcS$ has at most one $\bA \in X_\mcS$ with $\mcS(\bA) = \bB,$ by theorem \ref{aperiodicity implies recognizability}.  Thus $\mcS : X_\mcS \xrightarrow{\sim} \mcS(X_\mcS)$ is a homeomorphism, and so $T^\bj \mcS \vert_{X_\mcS}$ is an embedding.  

As $(\mcS\bB)(\bZero) = \mcS(\bB(\bZero))$ depends only on the definition of $\bB$ at $\bZero$, the second statement follows immediately from theorem \ref{aperiodicity implies recognizability}: the existence of $\bj \in [\bZero,\bq^n)$ and $\bB$ implies $\mcP_n$ covers, and uniqueness implies its members are disjoint.

For the third and final claim, fix $\omega \in \mcA^+,$ and $n \geq 0.$  That the union is disjoint follows immediately from theorem \ref{aperiodicity implies recognizability}, and we already know that the right hand side is contained in $[\omega]$, so we complete the proof by showing the reverse containment; to that end, let $\bA \in [\omega].$  By theorem \ref{aperiodicity implies recognizability}, there is a unique $\bB \in X_\mcS$ and $\bj \in [\bZero,\bq^n)$ such that $\bA = T^\bj \mcS^n(\bB)$, and so $T^\bj \mcS^n(\bB) \in [\omega].$  If $\omega$ is defined at $\bk$ then, using proposition \ref{qsub}, this implies that
$$\omega(\bk) = (T^\bj \mcS^n\bB)(\bk) = (\mcS^n\bB)(\bj+\bk) = \mcR_{\bj+\bk}^{(n)}(\bB(\qq{\bj+\bk}_n)) = \mcR_{\bj+\bk}^{(n)}(\eta(\qq{\bj+\bk}_n))$$
\noindent where $\eta \in \mcA^+$ is the block obtained by restricting $\bB$ to $\qq{\bj+\supp(\omega)}_n$.    Thus, we have $\bB \in [\eta]$ and so $\bA = T^\bj \mcS^n(\bB) \in T^\bj \mcS^n[\eta]$, showing that $\bA$ is also an element of the left-hand side.  \end{proof}

If we let $\alpha \in \mcA$ denote the block $\bZero \mapsto \alpha$, one checks that $\qq{\bj+\supp(\alpha)}_n = \bZero$ for $\bj \in [\bZero,\bq^n)$ as $\supp(\alpha)$ is $\bZero,$ so that the support of any $\eta$ coming from  $\mcS^{-n}(\alpha)$ is $\bZero$ also.    For more general $\omega$ with $\Omega := \supp(\omega)$, if $\omega \leq T^\bj \mcS^n(\eta)$ for some $\eta \in \mcA^+$ and $\bj \notin \Delta_p(\Omega) := \fatcup_{\bk \in \Omega} \Delta_p(\bk)$ then $\omega$ is extended by $T^\bj \mcS^n(\gamma)$ for $\gamma = \eta(\bZero).$  Thus, in this context, lemma \ref{small carries} shows that the proportion of $(\bj,\eta) \in \mcS^{-n}(\omega)$ corresponding to $\eta$ not defined entirely at the origin goes to $0$ as $n \to \infty$.  

Before we proceed to the spectral theory of aperiodic $\bq$-substitutions, we give a final application of aperiodicity which will allow us to relate the correlation measures $\sigma_{\alpha \beta}$ to the spectrum of our substitution.   The \textit{iterated initial indicators} for a $\bq$-substitution $\mcS$ on $\mcA$ are the indicator functions $\IND_{T^\bk \mcS^n[\gamma]} \in L^2(X_\mcS)$ of the iterated initial cylinders $T^\bk \mcS^n[\gamma]$ for $n \geq 0,$ $\bk \in \Zd,$ and $\gamma \in \mcA$. 
\begin{prop}\label{iterated indicator}
	If $\mcS$ is an aperiodic $\bq$-substitution on $\mcA$, then the iterated initial indicators $\IND_{T^\bk \mcS^n[\gamma]}$ for $\bk \in \Zd, n \geq 0,$ and $\gamma \in \mcA,$ have dense span in $L^2(\mu)$, for all $\mu \in \mcM(X_\mcS,T).$
\end{prop}
\begin{proof}
	As the functions $\IND_{[\omega]}$ for $\omega \in \mcA^+$ have dense span in $L^2(\mu)$, we can accomplish this by showing that every $\IND_{[\omega]}$ can be approximated by the linear combinations of $\IND_{T^\bk \mcS^n[\alpha]}$, the iterated initial indicators.  Note that corollary \ref{topological structure}.3 can be interpreted for indicator functions, giving
$$\IND_{[\omega]} = {\sum}_{(\bj,\eta) \in \mcS^{-p}(\omega)} \IND_{T^\bj \mcS^p([\eta])}$$	
\noindent Separate the above sum into two parts: those terms with $\bj \notin \Delta_p(\Omega)$ which correspond to blocks $\eta$ defined only at $\bZero$, and those with $\bj \in \Delta_p(\Omega)$ which correspond to $\eta$ with larger support, and so
$$\IND_{[\omega]} - {\sum}_{(\bj,\eta) \in \mcS^{-p}(\omega) \text{ and } \bj \notin \Delta_p(\Omega)} \IND_{T^\bj \mcS^p([\eta])} = {\sum}_{(\bj,\eta) \in \mcS^{-p}(\omega) \text{ and } \bj \in \Delta_p(\Omega)} \IND_{T^\bj \mcS^p([\eta])}$$
\noindent so that taking absolute values and integrating against $\mu,$ we can apply lemma \ref{small carries} and the right (and thus left) side goes to $0$.  This in turn shows that $\IND_{[\omega]}$ can be approximated by linear combinations of the iterated initial indicators, as those terms over $\bj \notin \Delta_p(\Omega)$ are supported entirely at the origin.  Thus, linear combinations of the iterated initial indicators approximate the indicators of the standard cylinders, and thus themselves have dense span in $L^2(\mu).$  
\end{proof}


\section{Spectral Theory}\label{spectral theory}

The goal of this section is to state the central result of this paper, theorem \ref{queffelec two}, which allows us to identify the spectrum of an aperiodic $\bq$-substitution $\mcS$ with finitely many measures ergodic for the $\bq$-\textit{shift}, discussed below.  This is significant as these mutually singular measures of pure type are given by linear combinations of the correlation measures $\sigma_{\alpha \beta} := \sigma_{\IND_{[\alpha]}, \IND_{[\beta]}}$ over its \textit{spectral hull} $\mcK$, a convex subset of $\bbC^{\mcA^2}$ determined by the coincidence matrix $C_\mcS$ of the $\bq$-substitution.  First, we discuss aspects of $\bq$-substitutions relating to their discrete spectrum.

\subsection{Height and the $\bq$-Shift}\label{height}
Most of our understanding of the discrete spectrum of constant length substitutions is due to the work of Dekking and the notion of \textit{height} of a substitution, see [\ref{dekking}].   The following description of height in the $\Zd$ case is based on [\ref{frank zd}, \S 3.1] and [\ref{queffelec}, \S 6.1.1].  Given a primitive and aperiodic $\bq$-substitution $\mcS$ on $\mcA,$ fix a substitution sequence $\bA$ in the hull $X_\mcS.$ 

For a sublattice $\mcL \subset \Zd$ (a subgroup of $\Zd$ spanning $\bbR^d$) let $\mcL_0$ be a set of class representatives, so that $\Zd$ is the disjoint union of $\bj + \mcL$ for $\bj \in \mcL_0.$   Let $\bA(\bj + \mcL)$ denote the letters appearing in positions $\bj + \bh$ for $\bh \in \mcL.$  By primitivity we know $\mcA = \bigcup_{\bj \in \mcL_0} \bA(\bj + \mcL),$ though in some cases the $\bA(\bj + \mcL)$ may form a partition of $\mcA.$  When that happens, one can identify all the letters in each $\bA(\bj + \mcL)$ with a single representative, and this gives a map $\mcA \to \mcA.$  When extended pointwise to $\AZd,$ this map takes $\bA$ to a periodic sequence which, by primitivity, must therefore take $X_\mcS$ onto a finite set.  If we let $\mcL = \Zd,$ this reduces $X_\mcS$ to a singleton.   The \textit{height lattice} of $\mcS$ is the smallest lattice (in the subset order) satisfying the conditions
\begin{equation}\label{height lattice}\mcA = {\bigsqcup}_{\bj \in \mcL_0} \bA(\bj + \mcL) \tbump{0.5in}{and} \mcL + \bq \Zd = \Zd\end{equation}
\noindent Moreover, the above map sending the $\bA(\bj + \mcL)$ to a representative above gives rise to the maximal equicontinuous factor of $(X_\mcS, T, \mu),$ see [\ref{frank zd}].  Given a lattice $\mcL \subset \Zd$ , its dual lattice $\mcL^\ast \subset \bbR^d$ is the collection of all $\btheta \in \bbR^d$ for which $\btheta^t \bh = \sum \theta_k h_k \in \bbZ$.  One checks that $\mcL^\ast$ is a subgroup of $\bbR^d$ and so its image under the exponential map
$\bbR^d \mapsto \bbT^d$ with $(\theta_1,\ldots, \theta_d) \longmapsto \big( e^{2\pi i \theta_1}, \ldots, e^{2\pi i \theta_d} \big)$ is a multiplicative subgroup of $\bbT^d$ which we call the \textit{$\mcL$-th roots of unity}; let $\bnu_\mcL$ denote its Haar measure.  It consists of those $\bz \in \bbT^d$ for which $z_1^{h_1}\cdots z_d^{h_d} = 1$ for all $\bh \in \mcL$.  Letting $\bnu_\bh := \bnu_{\bh \Zd}$
\begin{equation}\label{lattice measure}
\widehat{\bnu_\mcL}(\bk) = \scalebox{0.8}{$\begin{cases} \vspace{8px} \end{cases} \hspace{-12px}\begin{matrix} 1 & \text{if } \bk \in \mcL \\ 0 & \text{if } \bk \notin \mcL \end{matrix}$} \tbump{0.5in}{and} \widehat{\bnu_\bh}(\bk) =  \scalebox{0.8}{$\begin{cases}\vspace{8px} \end{cases} \hspace{-12px} \begin{matrix} 1 & \text{ if } \qr{\bk}_n = \bZero\\ 0 & \text{ if } \qr{\bk}_n \neq \bZero \end{matrix}$}
\end{equation}
\noindent For $\bq \geq 1,$ define the \textit{$\bq$-adic support measure} $\bomega_\bq$ for the $\bq$-adic roots of unity
\begin{equation}\label{omegaq}
\textstyle \bomega_\bq := \sum_{n \geq 1} 2^{-n} \nu_{\bq^n}
\end{equation}
\noindent corresponding to the $\bq$-adic rationals in $\bbR^d/\Zd$, or those rationals with denominator $\bq^n$ for $n \geq 0.$  

The following describes the discrete spectrum of an aperiodic $\bq$-substitution in terms of its expansion $\bq$ and height $\mcL$, see [\ref{queffelec}, Thm 6.1, 6.2] for $d=1$, and [\ref{frank zd}, \S 3.1] for the general case.  
\begin{thm}[Dekking]
\label{dekking thm}
	If $\mcS$ is a primitive and aperiodic $\bq$-substitution with height lattice $\mcL,$ then the discrete component of $\sigmaX$ is equivalent to $\bomega_\bq \ast \bnu_\mcL.$   If  $\mcL = \Zd,$ the spectrum is pure discrete if and only if a generalized instruction of $\mcS$ is constant.
\end{thm}
The criteria for the pure discrete case are known as \textit{trivial height} ($\mcL = \Zd$) and the \textit{coincidence condition} (a constant generalized instruction), respectively.  In the primitive case the coincidence condition is equivalent to the bisubstitution possessing only one ergodic class, as was observed by Queff\'elec.  More generally, the discrete spectrum of an aperiodic $\bq$-substitution is given by the sum of $\bomega_\bq \ast \bnu_\mcL$ as $\mcL$ ranges over the height lattices of its primitive components.   We comment here that although methods exist to compute the height of a given substitution, this is not necessary in order to compute the spectrum, as our algorithm produces the discrete component of the measure and one can even use this to determine the height; see example \ref{height example} and [\ref{bartlett}, Ex 4.4.2].

By theorem \ref{dekking thm}, a primitive $\bq$-substitution of trivial height has pure discrete spectrum if and only if some generalized instruction $\mcR_j^{(n)}$ is constant (\textit{Dekking's coincidence condition}).  It follows that primitive and aperiodic bijective $\bq$-substitutions cannot be pure discrete in the case of trivial height, and so their spectrum will always contain continuous measures.  The correlation measures, which are central to Queff\'elec's analysis, give us the tools necessary to study the entire spectrum.  

An essential ingredient is the \textit{$\bq$-shift}, a topological dynamical system on $\bbT^d$ given by
$$\bS_\bq: \bbT^d \to \bbT^d \tbump{0.5in}{with} \bz \mapsto \bz^\bq \tbump{0.25in}{or} (z_1,\ldots, z_d) \longmapsto  (z_1^{q_1}, \ldots, z_d^{q_d})$$
\noindent and which is topologically conjugate to the times $\bq$ map sending $\bx \mapsto \bq \bx \bump(\text{mod } \bOne)$ on $\bbR^d / \Zd.$ For a measure $\nu \in \mcM(\bbT^d)$ (complex measures on $\bbT^d$) its Fourier coefficients are given by
$$\textstyle \widehat{\nu}(\bk) := \int_{\bbT^d} z_1^{-k_1} \cdots z_d^{-k_d}  \, d\nu(z_1,\ldots z_d) \hspace{0.5in} \text{for } \bk \in \Zd$$
\noindent as in equation (\ref{spectral coefficients}).  By the following proposition, normalized Lebesgue measure ($m$) on $\bbT^d$ is ergodic for the $\bq$-shift.  The proof is standard, see [\ref{queffelec} \S 3.1.1] for the one-dimensional case.
\begin{prop}\label{spectral qshift identities}
	A nonzero measure $\nu \in \mcM(\bbT^d)$ is $\bq$-shift invariant or strong-mixing, respectively, if and only if their Fourier coefficients satisfy, for every $\ba, \bb \in \Zd$, the identities
	$$\widehat{\nu}(\ba\bq) = \widehat{\nu}(\ba) \tbump{0.55in}{ {resp.} } {\lim}_{p \to \infty}\,  \widehat{\nu}(\bb + \ba \bq^p) = \widehat{\nu}(\bb) \widehat{\nu}(\ba) $$
\end{prop}
\noindent Note that ergodicity is equivalent to strong and weak mixing in this case as the $\bq$-shift is a continuous endomorphism of $\bbT^d$ (as a group), see [\ref{walters}, Thm 1.28].   The \textit{invariant contracted map} (see [\ref{queffelec}, def 7.1]) is the map $\pi : \mcM(\bbT^d) \to \mcM(\bbT^d)$ sending $\nu \mapsto\pi(\nu)$ with Fourier coefficients
\begin{equation}\label{invariant contracted} 
\widehat{\pi(\nu)}(\bk) = \scalebox{0.8}{$\begin{cases} \vspace{8px} \end{cases} \hspace{-12px} \begin{matrix} \widehat{\nu}(\qq{\bk}_1) & \text{if } \qr{\bk}_1 = \bZero \\ 0  & \text{if } \qr{\bk}_1 \neq \bZero \end{matrix}$} \tbump{12px}{ and we write } \Pi := {\sum}_{n \geq 1} 2^{-n} \pi^n
\end{equation}
\noindent The following shows that for any $\bq$-shift invariant measure,  the iterates of the invariant contracted map are uniformly weighted averages of translates of that measure along the $\bq^n$-th roots of unity.
\begin{lem}\label{invariant contracted map}
	If $\rho \in \mcM(\bbT^d, \bS_\bq)$ then $\Pi(\rho) = \bomega_\bq \ast \rho,$  the convolution of $\bomega_\bq$ with $\rho$.
\end{lem}
\begin{proof}
	If $\rho$ is $\bq$-shift invariant, then proposition \ref{spectral qshift identities} gives us $\widehat{\rho}(\qq{\bk}_n) = \widehat{\rho}(\qq{\bk}_n \bq^n)$ so that 
$$\widehat{\pi^n\rho}(\bk) = {\small \begin{cases} \widehat{\rho}(\qq{\bk}_n) & \text{if } \qr{\bk}_n = \bZero \\ 0 & \text{otherwise} \end{cases} } = {\small \begin{cases} \widehat{\rho}(\qr{\bk}_n + \qq{\bk}_n \bq^n) & \text{if } \qr{\bk}_n = \bZero \\ \bZero & \text{if } \qr{\bk}_n \neq \bZero \end{cases}} = \widehat{\nu_{\bq^n}}(\bk) \cdot \widehat{\rho}(\bk) = \widehat{\nu_{\bq^n}\ast \rho}(\bk)$$
\noindent and the result follows from Fourier unicity, linearity and the definition of the map $\Pi$.
\end{proof} 
\noindent Note that the height measure $\bnu_\mcL$ of theorem \ref{dekking thm} is $\bq$-shift invariant as the condition $\mcL + \bq \Zd = \bbZ^d$ merely generalizes the notion of \textit{relatively prime} to higher dimensions.  When $d=1$, if $h,q \in \bbZ$ are relatively prime, we have $aq \in h\bbZ$ if and only if $a \in h \bbZ$.  In the $\Zd$ case, when $\mcL = \bh \Zd$ for example,  the condition $\bh\Zd + \bq\Zd = \Zd$ reduces component-wise to the $d=1$ case.  By identities (\ref{lattice measure}) and proposition \ref{spectral qshift identities} above, the discrete component of $\sigmaX$ is $\bq$-shift invariant, and $\Pi(\bnu_\mcL) = \bomega_\bq \ast \bnu_\mcL$.  


\subsection{The Correlation Measures and $\Sigma$}\label{sigma section}

Recall that for $\alpha, \beta \in \mcA,$ the \textit{correlation measure} $\sigma_{\alpha \beta} \in \mcM(\bbT^d)$ is the spectral measure for the pair of indicator functions $\IND_{[\alpha]}$ and $\IND_{[\beta]}.$   Using the spectral theorem (\ref{spectral coefficients}), we compute the Fourier coefficients of the correlation measure $\sigma_{\alpha \beta}$ using translation invariance and obtain for $\bk \in \Zd$ 
\begin{equation}\label{fourier coefficients} \widehat{\sigma_{\alpha \beta}}(\bk) = \int_{X_\mcS} \IND_{[\alpha]} \circ T^{-\bk} \cdot \IND_{[\beta]} d\mu = \mu(T^\bk[\alpha] \cap [\beta]) = \mu( [\alpha] \cap T^{-\bk}[\beta])\end{equation}
\noindent and so $\widehat{\sigma_{\alpha \beta}}(\bk) = \overline{\widehat{\sigma_{\beta \alpha}}(-\bk)} = \widehat{\sigma_{\beta \alpha}}(-\bk)$ as $\mu \in \mcM(X_\mcS,T)$ is real valued.  Moreover, the above relates the Fourier coefficients of the correlation measures $\sigma_{\alpha \beta} \in \mcM(\bbT^d)$ to the $\mu$-frequencies with which two-letter patterns appear in substitution sequences.  This is the primary application of aperiodicity: using corollary \ref{topological structure}.3 to compute the measures of cylinders combined with the above allows us to compute the Fourier coefficients of the correlation measures explicitly.  

As $\sigmaX$ is the smallest type, or equivalence class of measures, satisfying $\sigma_f \ll \sigmaX$ for all $f \in L^2(\mu)$, and the iterated indicator functions $\IND_{T^\bk \mcS^n[\alpha]}$ have dense span, we can use their spectral measures to characterize the spectrum up to type.   The following generalization of Queff\'elec's result reveals a relationship between the correlation measures and the $\bq$-shift (via the map $\Pi$ of the previous section) which gives rise to the spectrum of $\bq$-substitutions; see [\ref{queffelec}, \S 7.1.2].
\begin{thm}\label{queffelec one}
	If $\mcS$ is an aperiodic $\bq$-substitution, then the maximal spectral type of $(X_\mcS,T,\mu)$ 
	$$\textstyle \sigmaX \sim \sum_{\alpha \in \mcA} \sum_{n \geq 1} 2^{-n} \sigma_{\IND_{\mcS^n[\alpha]}} \sim \sum_{\alpha \in \mcA} \Pi(\sigma_{\alpha \alpha})$$
	\noindent where $\sim$ denotes equivalence, so that the spectrum is determined by the correlation measures.
\end{thm}
\begin{proof}
Let $F \in L^2(\mu)$ be a maximal function for the spectral map $\sigma,$ and $g_n$ a sequence of linear combinations of iterated indicators approximating $F$ in $L^2(\mu)$.   As the $n$-th initial indicators $\IND_{T^\bk\mcS^n[\alpha]}$ for $\bk \in \Zd$ and $\alpha \in \mcA$ span the $m$-th initial indicators for all $m \leq n$ by corollary \ref{topological structure}.3, each $g_n$ can be be written as linear combinations of $\IND_{T^\bk \mcS^{m_n}[\alpha]}$ with coefficients $c_{n,\bk,\alpha}$ for $\bk \in \Zd$, $\alpha \in \mcA$ and $m_n$ sufficiently large.  As the spectral map is continuous and bilinear, we let $n \to \infty$
$$\sigmaX \sim \sigma_F = \lim_{n \to \infty} \sigma_{g_n} = \lim_{n \to \infty}  {\sum}_{\alpha, \beta \in \mcA} {\sum}_{\bj, \bk \in \Zd} c_{n, \bk, \alpha} \overline{c_{n, \bj, \beta}} \sigma_{\IND_{T^\bk \mcS^{m_n}[\alpha]} ,\IND_{T^\bj \mcS^{m_n}[\beta]}} $$
Using standard properties of the spectral map [\ref{queffelec}, prop 2.4] and translation invariance of $\mu$ 
$$\sigma_{\IND_{T^\bk \mcS^m[\alpha]} ,\IND_{T^\bj \mcS^m[\beta]}} \ll \sigma_{\IND_{\mcS^m[\alpha]}} \tbump{12px}{ and } \sigma_{\alpha \beta} \ll \sigma_{\alpha \alpha}$$
for $m \geq 0$ and every $\bj,\bk \in \Zd$ and $\alpha, \beta \in \mcA$, so that the above combine to give us
$$\sigmaX \sim \lim_{n \to \infty} \sigma_{g_n} \ll {\sum}_{m \geq 1} {\sum}_{\alpha \in \mcA} 2^{-m} \sigma_{\IND_{\mcS^m[\alpha]}} \ll \sigmaX$$
by definition of the maximal spectral type, and giving the first equivalence.

By definition of the spectral map and identities (\ref{spectral coefficients}), for $\alpha \in \mcA,$ $n \geq 0$, and $\bk \in \Zd$, we have
$$\widehat{\sigma_{\IND_{\mcS^n[\alpha]}}}(\bk) = \int_{\bbT^d} \IND_{\mcS^n[\alpha]} \circ T^{-\bk} \cdot \IND_{\mcS^n[\beta]} \, d\mu = \int_{\bbT^d} \IND_{T^\bk \mcS^n[\alpha]} \cdot \IND_{\mcS^n[\alpha]} d\mu = \mu(\mcS^n[\alpha] \cap T^\bk \mcS^n[\alpha])$$
so that $\mcS^n[\alpha] \cap T^\bk \mcS^n[\alpha] = \emptyset$ whenever $\qr{\bk}_n \neq \bZero$ (whenever $\bk$ not divisible by $\bq^n$) by (\ref{ifs cancellation}) of corollary \ref{topological structure}, and $\mcS^n([\alpha] \cap T^{\qq{\bk}_n}[\alpha] )$ otherwise.  Using the scaling property of invariant measures, $\mu\circ \mcS^n = \nicefrac{1}{Q^n} \mu,$ gives the second equivalence as $\mu([\alpha] \cap T^{\qq{\bk}_n}[\alpha]) = \widehat{\sigma_{\alpha\alpha}}(\qq{\bk}_n) = \widehat{\pi^n\sigma_{\alpha \alpha}}(\bk)$ and the definition of $\Pi$ as we are only identifying the spectrum up to type.
\end{proof}
\noindent We are interested in linear combinations $\lambda$ of the correlation measures which are $\bq$-shift invariant and which are equivalent to $\sum \sigma_{\alpha \alpha}$: by lemma \ref{invariant contracted map} and linearity of $\Pi$
\begin{equation}\label{q invariance of Pi}
	{\textstyle \lambda \in \mcM(\bbT^d, \bS_\bq) \tbump{4px}{ and } \lambda \sim {\sum}_{\alpha \in \mcA} \, \sigma_{\alpha \alpha} \tbump{0.25in}{$\implies$} \bomega_\bq \ast \lambda  \sim \Pi\left({\sum}_{\alpha \in \mcA} \, \sigma_{\alpha \alpha}\right) \sim \sigmaX}
\end{equation}
and so would give rise to the spectrum of $\mcS$ via convolution with $\bomega_\bq$, a pure discrete measure with $\bq$-shift invariant support.  We now detail how the recursive structure provided by aperiodicity is used to compute the correlation measures and ultimately identify measures satisfying the above.
\begin{defn}\label{sigma definition}
	The vector valued measure $\Sigma := \sum_{\alpha \beta \in \mcA^2} \sigma_{\alpha \beta} \be_{\alpha \beta}$ is the \textit{correlation vector} of $\mcS.$
\end{defn}
\noindent Queff\'elec defines $\Sigma$ as an $\mcA \times \mcA$ matrix valued measure, see [\ref{queffelec}, \S 7.1.3], although it is more natural for us as a (column) vector.  The following \textit{Fourier recursion theorem} provides an infinite family of relations amongst the Fourier coefficients of the correlation measures, and allows for the explicit computation of the correlation measures.  It both generalizes and extends [\ref{queffelec} \S 8.3 (8.7)].
\begin{thm}\label{fourier recursion}
	Let $\mcS$ be an aperiodic $\bq$-substitution on $\mcA.$  Then for $p \in \bbN$ and $\bk \in \Zd$
	$$\widehat{\Sigma}(\bk) = \frac{1}{Q^p} {\sum}_{\bj \in [\bZero,\bq^p)} \mcR_\bj^{(p)} \otimes \mcR_{\bj+\bk}^{(p)} \hspace{3px} \widehat{\Sigma}\big(\qq{\bj+\bk}_p \big) =  \lim_{n \to \infty} \frac{1}{Q^n} {\sum}_{\bj \in [\bZero,\bq^n)}  \mcR_\bj^{(n)} \otimes \mcR_{\bj+\bk}^{(n)}  \hspace{3px} \widehat{\Sigma}(\bZero)$$
	\noindent where $\qq{\bj+\bk}_p$ is the quotient of $\bj+\bk$ {\rm mod} $\bq^p$ and $\mcR^{(n)}$ the generalized instructions, see \S \ref{arithmetic}, \ref{configurations}.
\end{thm}
\begin{proof}
	As $\widehat{\Sigma}(\bk) = {\sum}_{\alpha \beta \in \mcA^2} \widehat{\sigma_{\alpha \beta}}(\bk) \be_{\alpha \beta}$, and $\widehat{\sigma_{\alpha \beta}}(\bk) = \mu([\alpha] \cap T^{-\bk}[\beta])$, the proof is essentially just disjoint additivity applied to corollary \ref{topological structure}.3, using the Kronecker products to reindex the sum.  First, fix $\alpha \beta \in \mcA^2$ and $\bk \in \Zd,$ then define $\omega \in \mcA^+$ by writing
	$$\omega : \bZero \mapsto \alpha \tbump{2px}{ and } \bk \mapsto \beta  \tbump{4px}{ so that } [\omega] = [\alpha] \cap T^{-\bk}[\beta] \tbump{12px}{and} \widehat{\sigma_{\alpha \beta}}(\bk) = \mu([\omega])$$
	Now, for each $\bj \in [\bZero,\bq^n)$ and pair $\gamma \delta \in \mcA^2$ attempt to define a $\eta \in \mcA^+$ by writing 
	$$\eta : \bZero \mapsto \gamma \tbump{2px}{ and } \eta: \qq{\bj+\bk}_n \mapsto \delta \tbump{12px}{ so that } \eta \in \mcA^+ \tbump{4px}{whenever}  \qq{\bj+\bk}_n = \bZero \implies \gamma = \delta$$
	ignoring all other cases.  By definition of the $n$-th desubstitute, $(\bj, \eta) \in \mcS^{-n}(\omega)$ if and only if
	$$\alpha = (T^\bj \mcS^n\eta)(\bZero), \hspace{8px} \beta = (T^\bj\mcS^n\eta)(\bk) \tbump{12px}{ and } \eta \in \mcA^+ \text{ is well defined}$$
	or equivalently, by the definition of the generalized instructions and proposition \ref{qsub},
	$$\alpha = \mcR_\bj^{(n)}(\eta(\bZero)), \hspace{8px} \beta = \mcR_{\bj+\bk}^{(n)}(\eta(\qq{\bj+\bk}_n)) \tbump{12px}{ and } \qq{\bj+\bk}_n = \bZero \implies \gamma = \delta$$
	so that $(\bj,\eta) \in \mcS^{-n}(\omega)$ if and only if (viewing cylinders as subsets of $\AZd$ for the moment)
	$$\alpha = \mcR_\bj^{(n)}(\gamma), \hspace{8px} \beta = \mcR_{\bj+\bk}^{(n)}(\delta) \tbump{12px}{ and } [\eta] = [\gamma] \cap T^{-\qq{\bj+\bk}_n}[\delta] \neq \emptyset \subset \AZd$$
	Using this and (\ref{tensor selector}) , one checks that for $\bj \in [\bZero,\bq^n)$ and $\gamma \delta \in \mcA^2$ the block $\eta: \big\{ {\tiny \begin{matrix} \hspace{12px} \bZero \hspace{10px} \mapsto \gamma \\ \qq{\bj+\bk}_n \mapsto \delta \end{matrix}}$ satisfies
	{\small $$\frac{1}{Q^n} \be_{\alpha\beta}^t \mcR_\bj^{(n)} \otimes \mcR_{\bj+\bk}^{(n)} \be_{\gamma \delta} \, \mu( [\eta]) = \scalebox{0.9}{${\small \begin{cases} \mu\left(T^\bj \mcS^n[\eta]\right) & \text{if } (\bj, \eta) \in \mcS^{-n}(\omega) \\ 0 & \text{otherwise} \end{cases}}$}$$}
	and so adding up over all $\bj \in [\bZero,\bq^n)$ and $\gamma \delta \in \mcA^2$ we obtain the identity
	{\small $$\frac{1}{Q^n}{\sum}_{\bj \in [\bZero,\bq^n)} {\sum}_{\gamma \delta \in \mcA^2} \left( \be_{\alpha \beta}^t \mcR_\bj^{(n)} \otimes \mcR_{\bj+\bk}^{(n)} \be_{\gamma \delta} \right) \, \mu\big([\gamma] \cap T^{-\qq{\bj+\bk}_n}[\delta]\big) = {\sum}_{(\bj,\eta) \in \mcS^{-n}(\omega)} \mu(T^\bj \mcS^n[\eta])$$}
	so that corollary \ref{topological structure}.3 and (\ref{fourier coefficients}) show that the Fourier coefficients of $\sigma_{\alpha \beta}$ can be computed
	{\small $$\widehat{\sigma_{\alpha \beta}}(\bk) = \mu[\omega] = \frac{1}{Q^n} {\sum}_{\bj \in [\bZero,\bq^n)} {\sum}_{\gamma \delta \in \mcA^2} \left( \be_{\alpha \beta}^t \mcR_\bj^{(n)} \otimes \mcR_{\bj+\bk}^{(n)} \be_{\gamma \delta} \right) \, \widehat{\sigma_{\gamma \delta}}(\qq{\bj+\bk}_n)$$}
	which gives the first equality as $\widehat{\Sigma}(\qq{\bj+\bk}_n) = \sum_{\gamma \delta \in \mcA^2} \widehat{\sigma_{\gamma \delta}}(\qq{\bj+\bk}_n) \be_{\gamma \delta}$.
	
	To show the limit identity, simply separate the sum over $\bj \in [\bZero,\bq^n)$ into those integers for which $\qq{\bj+\bk}_n = \bZero$ ($\bj \notin \Delta_p(\bk)$) from those for which $\qq{\bj+\bk}_n \neq \bZero$ ($\bj \in \Delta_p(\bk)$), writing
	{\small \begin{align*}
		\widehat{\Sigma}(\bk) &= \frac{1}{Q^n} {\sum}_{\bj \in [\bZero, \bq^n) \smallsetminus \Delta_n(\bk)} \mcR_\bj^{(n)} \otimes \mcR_{\bj+\bk}^{(n)} \widehat{\Sigma}(\bZero) + \frac{1}{Q^n} {\sum}_{\bj \in \Delta_n(\bk)} \mcR_\bj^{(n)} \otimes \mcR_{\bj+\bk}^{(n)} \widehat{\Sigma}(\qq{\bj+\bk}_n) \\
			&= \frac{1}{Q^n} {\sum}_{\bj \in [\bZero,\bq^n)} \mcR_\bj^{(n)} \otimes \mcR_{\bj+\bk}^{(n)} \widehat{\Sigma}(\bZero) + \frac{1}{Q^n} {\sum}_{\bj \in \Delta_n(\bk)} \mcR_\bj^{(n)} \otimes \mcR_{\bj+\bk}^{(n)} \left( \widehat{\Sigma}(\qq{\bj+\bk}_n) - \widehat{\Sigma}(\bZero) \right)
	\end{align*}}
	\noindent so that letting $n \to \infty$ and using the carry estimates of lemma \ref{small carries} gives the desired result, as the $\mcR_\bj^{(n)} \otimes \mcR_{\bj+\bk}^{(n)}$ are column stochastic and (\ref{fourier coefficients}) which shows that $|\widehat{\Sigma}(\bk)| \leq \text{Card}(\mcA)$ for $\bk \in \Zd$.
\end{proof}
\begin{exmp}\label{TM sigma}
	We consider the Thue-Morse example and compute some Fourier coefficients of the correlation vector $\Sigma.$   First, we find Kronecker products of the instruction matrices $\mcR_0$ and $\mcR_1$
	$${\tiny \mcR_0 \otimes \mcR_0 =\scalebox{0.75}{ $ \begin{pmatrix} 1 & & & \\ & 1 & & \\ & & 1 & \\ & & & 1 \end{pmatrix}$} \hspace{0.25in} \mcR_0 \otimes \mcR_1 = \scalebox{0.75}{$\tiny{\begin{pmatrix} & 1 & & \\ 1 & & & \\ & & & 1 \\ & & 1 & \end{pmatrix}}$} \hspace{0.25in} \mcR_1 \otimes \mcR_0 = \scalebox{0.75}{$\tiny{\begin{pmatrix} & & 1 & \\ & & & 1 \\ 1 & & & \\ & 1 & & \end{pmatrix}}$} \hspace{0.25in} \mcR_1 \otimes \mcR_1 = \scalebox{0.75}{$\begin{pmatrix} & & & 1 \\ & & 1 & \\ & 1 & & \\ 1 & & & \end{pmatrix}$} }$$
	Now, applying the Fourier recursion theorem \ref{fourier recursion} for $k = 1$ and $p = 1$ we have
	$$\widehat{\Sigma}(1) = \frac{1}{2} \mcR_0 \otimes \mcR_{0+1} \widehat{\Sigma}(\qq{0 + 1}_1) + \frac{1}{2} \mcR_1 \otimes \mcR_{1 + 1} \widehat{\Sigma}(\qq{1+1}_1) = \frac{1}{2} \mcR_0 \otimes \mcR_1\widehat{\Sigma}(0) + \frac{1}{2} \mcR_1 \otimes \mcR_0 \widehat{\Sigma}(1)$$
	We can then solve for $\widehat{\Sigma}(1)$ above and use $\widehat{\Sigma}(0) = \frac{1}{2} \sum_{\gamma \in \mcA} \be_{\gamma \gamma}$ as the Perron vector is $(\frac{1}{2}, \frac{1}{2})^t$ so
	$$\widehat{\Sigma}(0) = \frac{1}{2}(1, 0, 0, 1)^t \tbump{12px}{ and } \widehat{\Sigma}(1) = (2\bI - \mcR_1 \otimes \mcR_0)^{-1}\mcR_0 \otimes \mcR_1 \widehat{\Sigma}(0) = \frac{1}{6}(1, 2, 2, 1)^t$$
	with the basis ordered lexicographically: $\be_{00}, \be_{01}, \be_{10}, \be_{11}.$  If $k = 5,$ then for $p=3$ (as $2^2 < 5 < 2^3$)
	\begin{align*}
		\widehat{\Sigma}(5) &= \frac{1}{8} {\sum}_{j=0}^{2} \mcR_j^{(3)} \otimes \mcR_{j+5}^{(3)} \, \widehat{\Sigma}(0) + \frac{1}{8}{\sum}_{j=3}^7 \mcR_j^{(3)} \otimes \mcR_{j +5}^{(3)} \, \widehat{\Sigma}(1) \\
			&= \frac{1}{8}( \mcR_0 \otimes \mcR_0 + \mcR_1 \otimes \mcR_0 + \mcR_1 \otimes \mcR_1)\widehat{\Sigma}(0) + \frac{1}{8}(2\mcR_0 \otimes \mcR_0 + 2\mcR_1 \otimes \mcR_1 + \mcR_0 \otimes \mcR_1)\widehat{\Sigma}(1) \\
			&= \frac{1}{16} {\tiny \scalebox{0.75}{ $ \begin{pmatrix} 1 & & 1 & 1 \\ & 1 & 1 & 1 \\ 1 & 1 & 1 & \\ 1 & 1 & & 1 \end{pmatrix}$}} {\tiny \scalebox{0.75}{$\begin{pmatrix} 1 \\ 0 \\ 0 \\ 1\end{pmatrix}$}} + \frac{1}{48} {\tiny \scalebox{0.75}{ $ \begin{pmatrix} 2 & 1 &  & 2 \\ 1 & 2 & 2 &  \\  & 2 & 2 & 1 \\ 2 & & 1 & 2 \end{pmatrix}$}} {\tiny \scalebox{0.75}{$\begin{pmatrix} 1 \\ 2 \\ 2 \\ 1\end{pmatrix}$}} = \frac{1}{4} (1,1,1,1)^t
	\end{align*}
	\noindent as $\mcR_\bj^{(n)}$ is $\mcR_0$ or $\mcR_1$ corresponding to the parity of $1$'s in its binary expansion, see example \ref{TM example}.
\end{exmp}
\noindent In \S \ref{examples}, we use the Fourier recursion as in the above example to compute $\widehat{\Sigma}(\bk)$ (and ultimately $\sigmaX$) for several examples, but first we use it to determine the set of $\bq$-shift invariant probability measures in the span of the correlation measures.

\subsection{The Spectral Hull - $\mcK$}\label{spectral hull}

As we are interested in linear combinations of the correlation measures, for $\bv \in \bbC^{\mcA^2}$ write
\begin{equation}\label{lambda v}
	\textstyle \lambda_\bv := \bv^t \Sigma = \sum_{\alpha \beta \in \mcA^2} v_{\alpha \beta} \sigma_{\alpha \beta}
\end{equation}
\noindent which defines a linear map taking a $\bbC^{\mcA^2}$ vector onto the corresponding linear combination of the correlation measures.  As we are looking for $\bq$-shift invariant measures in the span, we compute the Fourier coefficients $\widehat{\Sigma}(\ba \bq)$ for $\ba \in \Zd$ using theorem \ref{fourier recursion} giving us the identity
\begin{equation}\label{q invariance of sigma}
\textstyle \widehat{\Sigma}(\ba\bq) = \frac{1}{Q} \sum_{\bj \in [\bZero,\bq)} \mcR_\bj^{(1)} \otimes \mcR_{\bj+\ba \bq}^{(1)} \widehat{\Sigma}(\qq{\bj+\ba\bq}_1)	= \frac{1}{Q}C_\mcS \widehat{\Sigma}(\ba)
\end{equation}
\noindent as $\qq{\bj+\ba\bq}_1 = \ba$ and $\qr{\bj+\ba\bq}_1 = \bj$ for $\bj \in [\bZero,\bq),$ and as $C_\mcS = \sum \mcR_\bj \otimes \mcR_\bj$.  
\begin{lem}\label{qshift invariance}
If $\bv$ is a left $Q$-eigenvector of $C_\mcS,$ then $\lambda_\bv = \bv^t \Sigma$ is invariant for the $\bq$-shift.
\end{lem}
\begin{proof}
	Let $\bv$ be a left $Q$-eigenvector of $C_\mcS$ so that $C_\mcS^t \bv  = Q \bv$.  Then by (\ref{q invariance of sigma})
	$$\textstyle  \widehat{\lambda_\bv}(\ba\bq) = \bv^t \widehat{\Sigma}(\ba\bq) = \frac{1}{Q} \bv^t  C_\mcS \widehat{\Sigma}(\ba) = \bv^t \widehat{\Sigma}(\ba) = \widehat{\lambda_\bv}(\ba)$$
	\noindent and so $\lambda_\bv$ is invariant for the $\bq$-shift by proposition \ref{spectral qshift identities}.
\end{proof}
Thus, the left $Q$-eigenspace of the coincidence matrix gives rise to $\bq$-shift invariant linear combinations of correlation measures.   As the correlation measures $\sigma_{\alpha \beta}$ for $\alpha \neq \beta$ are not positive measures, we cannot be sure that a general linear combination will also be positive.  The following gives a useful condition guaranteeing positivity, and is due to Queff\'elec, see [\ref{queffelec}, Prop 10.3].  
\begin{defn}\label{strong positivity}
	For $\bv = (v_{\alpha \beta})_{\alpha \beta \in \mcA^2} \in \bbC^{\mcA^2},$ let $\mathring{\bv} = (v_{\alpha \beta})_{\alpha, \beta \in \mcA} \in \bM_\mcA(\bbC)$ be its \textit{associated matrix}.   We say $\bv$ is \textit{strongly semipositive} if $\mathring{\bv}$ is positive semidefinite: \textit{writing} $\bv \stpos 0$ whenever $\mathring{\bv} \stpos 0$.  
\end{defn}
\noindent One reads the entries of a $\bbC^{\mcA^2}$ vector into the entries of its $\mcA \times \mcA$ associated matrix along each row sequentially in order.  In the $s = 2$ case, the forward and inverse maps are, respectively, 
$$\small{\hspace{0.35in}\begin{pmatrix} a & b & c & d \end{pmatrix}^\text{t} \longmapsto \begin{pmatrix} a & b \\ c & d \end{pmatrix} \hspace{0.75in} \begin{pmatrix} a_{11} & a_{12} \\ a_{21} & a_{22} \end{pmatrix} \longmapsto \begin{pmatrix} a_{11} & a_{12} & a_{21} & a_{22} \end{pmatrix}^\text{t}}$$
\noindent The associated matrix relates the Kronecker and matrix products: for $\bA, \bB \in \bM_\mcA(\bbC)$ and $\bv \in \bbC^{\mcA^2}$
\begin{equation}\label{associated kronecker product}
{( \bA \otimes \bB )^\circ \bv} = \bA \mathring{\bv} \bB^t
\end{equation}
\noindent and which is used in the proof of theorem \ref{abc theorem} to express the spectrum of an aperiodic bijective commutative $\bq$-substitution as a matrix Riesz product, but most importantly, the following:
\begin{lem}\label{schur inner product}
	If $\bv \in \bbC^{\mcA^2}$ is nonzero and strongly semipositive then $\bv^t \Sigma$ is a positive measure.
\end{lem}
\begin{proof}
	By the Schur product theorem, the Hadamard product ($A \tbump{-3px}{$\circ$} B$)  of two positive (semi)definite matrices is positive (semi)definite, see [\ref{HJ}, \S 5].  As $\bv^t \bw = \sum v_i w_i = \bOne^t (\mathring{\bv} \tbump{-3px}{$\circ$} \mathring{\bw}) \bOne$, we have $\bv^t \bw \geq 0$ whenever both $\bv$ and $\bw$ are strongly semipositive.  Using bilinearity of the spectral map $f,g \mapsto \sigma_{f,g}$ and the spectral theorem, one checks that $\mathring{\Sigma}$ is Hermitian positive definite, so that $\lambda_\bv$ is a positive measure whenever $\bv \stpos 0$ as these conditions are determined pointwise for measures.
\end{proof}
\noindent Finally, recall from the Fourier recursion theorem \ref{fourier recursion} that $\widehat{\Sigma}(\bZero) = \sum_{\alpha \in \mcA} \mu([\alpha]) \be_{\alpha \alpha}$ and so
$$\lambda_\bv(\bbT^d) = \widehat{\lambda_\bv}(\bZero) = \bv^t \widehat{\Sigma}(\bZero) = \textstyle{ \sum_{\alpha \in \mcA} v_{\alpha \alpha} \mu([\alpha])}$$
\noindent so that $\bv \stpos 0$ gives rise to a probability measure if and only if $\bv^t \widehat{\Sigma}(\bZero) = \sum v_{\alpha \alpha} \mu([\alpha]) = 1.$ In the primitive case, this is equivalent to $v_{\alpha \alpha} = 1$ for $\alpha \in \mcA,$ so that the probability measures arising from strong semipositivity via $\lambda$ is a closed and bounded convex set.  The above discussion prompts the following definition, see also [\ref{queffelec}, Def 11.1].
\begin{defn}\label{convex set K}
The \textit{spectral hull} $\mcK(\mcS)$ of a $\bq$-substitution $\mcS$ is the collection
$$\mcK(\mcS) := \{ \bv \in \bbC^{\mcA^2} : C_\mcS^t \bv = Q \bv \text{ and } \bv \stpos 0 \text{ and } \bv^t \widehat{\Sigma}(\bZero) = 1 \}$$
\noindent and we let $\mcK^\ast := \text{Ext}(\mcK) \subset \mcK$ denote the extreme points of the spectral hull.
\end{defn}
\noindent An important example is the vector $\bv = \bOne_{\mcA^2} \in \bbC^{\mcA^2}$, which is a $Q$-eigenvector of the $Q$-stochastic coincidence matrix $C_\mcS^t$ and strongly semipositive as $\mathring{\bOne_{\mcA^2}}$ is positive semidefinite.  As $\sum_{\alpha \in \mcA} u_\alpha = 1$ we have $\bOne_{\mcA^2} \in \mcK(\mcS)$ for every aperiodic $\bq$-substitution $\mcS$ so that $\mcK^\ast$ is nonempty.  Observe that
$${\textstyle \lambda_{\bOne_{\mcA^2}} = {\sum}_{\alpha\beta \in \mcA^2} \sigma_{\alpha \beta} =  {\sum}_{\alpha \in \mcA} {\sum}_{\beta \in \mcA} \sigma_{\IND_{[\alpha]}, \IND_{[\beta]}} = \sigma_{\scalebox{0.6}{${\sum}_{\alpha} \IND_{[\alpha]}$}} = \bdelta_\bOne}$$
\noindent the unit Dirac mass at $\bOne \in \bbT^d$, as $\sum_{\alpha \in \mcA} \IND_{[\alpha]}= 1$: note that $\bdelta_\bOne$ is part of the spectrum of \textit{every} $\Zd$-action on a compact metric space, and is ergodic for the $\bq$-shift.  Combining the above gives:
\begin{prop}\label{K invariance}
	For $\bv \in \mcK(\mcS)$, $\lambda_\bv = \bv^t \Sigma \ll \sigmaX$ is a $\bq$-shift invariant probability measure.
\end{prop}
\noindent Thus, the spectral hull is a nonempty closed convex subset of $\bbC^{\mcA^2}$ with finitely many extreme points whose elements give rise (via $\bv \mapsto \lambda_\bv$) to $\bq$-shift invariant probability measures in the span of the correlation measures.   If we knew that $\lambda_\bw \sim \sum \sigma_{\alpha \alpha}$ for some $\bw \in \mcK$, then we could combine lemma \ref{invariant contracted map} and theorem \ref{queffelec one} to show that $\sigmaX \sim \bomega_\bq \ast \lambda_\bw$.  This is accomplished in theorem \ref{queffelec two} which asserts this for all positive linear combinations of the extreme points of the spectral hull.

Before we continue, we discuss an alternate characterization of the spectral hull which is more suitable for computation.   First, telescope the bisubstitution $\mcS \otimes \mcS$ to have index of imprimitivity $1$, ergodic classes $\mcE_1, \ldots, \mcE_J$, and transient part $\mcT$, see proposition \ref{PRF}.  Let $\mcP_\mcT$ denote the orthogonal projection onto the span of the transient coordinate vectors of $\bbC^{\mcA^2}$ so that, reordering the basis of $\mcA^2$ according to the ergodic decomposition, the reduced normal form (\ref{matrix PRF}) for $C_\mcS^t$ gives us
$$ C_\mcS^t \approx {\tiny \scalebox{0.8}{$\begin{bmatrix} C_{1,1} & & &  \\ & \ddots & &  \\ & & C_{J,J} &  \\ C_{\mcT,1} &\cdots & C_{\mcT,J} & C_{\mcT,\mcT} \end{bmatrix}$}} \tbump{12px}{\normalsize{ and we write } } C_\mcT^t = C_\mcS^t \mcP_\mcT \approx  {\tiny \scalebox{0.8}{$\begin{bmatrix} \tbump{6px}{$\bZero$} & & &  \\ & \ddots & &  \\ & & \tbump{6px}{$\bZero$} &  \\ \tbump{6px}{$\bZero$} &\cdots & \tbump{6px}{$\bZero$} & C_{\mcT,\mcT} \end{bmatrix}$}}$$
\noindent The square blocks along the diagonal of $C_\mcS^t$ correspond to the ergodic classes $\mcE_j$ and are primitive and $Q$-stochastic, excepting those in $\mcT,$ the transient pairs.  For $w_1,\ldots, w_J \in \bbC$, write 
\begin{equation}\label{ergodic class representation}\mcV_\mcE := \mcV_\mcE(w_1,\ldots, w_J) := \textstyle  \sum_{j=1}^J \sum_{\alpha \beta \in \mcE_j} \, w_j \, \be_{\alpha \beta} \in \bbC^{\mcA^2}\end{equation}
\noindent so that $\mcV_\mcE$ is componentwise constant on the ergodic classes, and zero on the transient part.   The following allows us to compute the spectral hull explicitly, see also [\ref{queffelec}, Proposition 10.2].
\begin{prop}\label{characterization of K}
	If $\mcS$ is a $\bq$-substitution on $\mcA$ and $\mcE$ is the ergodic decomposition of its bisubstitution, then $\bv \in \mcK(\mcS)$ if and only if $\mathring{\bv}$ is self-adjoint with nonnegative eigenvalues and satisfies 		
	\vspace{-0.15in}	
	$$ \scalebox{0.9}{${\sum}_{\alpha \in \mcA} v_{\alpha \alpha} u_\alpha = 1 \tbump{0.15in}{and} \bv = \mcV_\mcE - (Q\bI - C_\mcT^t)^{-1} (Q \bI - C_\mcS^t) \mcV_\mcE \tbump{0.15in}{for some} w_1,\ldots, w_J \in \bbC$}$$
	\vspace{-0.15in}
	In particular, if $\mcT = \emptyset$ then $\bv = \mcV_\mcE$; if $\mcS$ is primitive then $\bv_{\alpha \alpha} = 1$ for all $\alpha \in \mcA$.
\end{prop}
\begin{proof}
	Strong semipositivity is equivalent to self-adjointness and nonnegative spectrum by usual properties of positive semidefiniteness, and as $\widehat{\Sigma}(\bZero) = \sum u_\alpha \be_{\alpha \alpha}$, we have $\bv^t \widehat{\Sigma}(\bZero) = 1$ if and only if $\sum v_{\alpha \alpha} u_\alpha = 1.$  It remains to show equivalence of $C_\mcS^t \bv = Q \bv$ with the above identity
	
	Let $\bv_\mcT := \mcP_\mcT \bv$.  As the blocks $C_{j,j}$ corresponding to $C_\mcS^t$ on the $\mcE_j$ are primitive and $Q$-stochastic, one can use the above primitive reduced form of $C_\mcS^t$ to see that its $Q$-eigenvectors are constant on ergodic classes, so $C_\mcS^t \bv = Q\bv$ if and only if $\bv - \bv_\mcT = \mcV_\mcE(w_1,\ldots,w_J)$ for some constants $w_1,\ldots, w_J \in \bbC$, or equivalently 
	$$(Q\bI - C_\mcS^t) \bv = \bZero \tbump{12px}{ $\iff$ } (Q \bI - C_\mcS^t) \mcV_\mcE = (Q \bI - C_\mcS^t)\bv - (Q \bI- C_\mcS^t) \bv_\mcT = - (Q \bI - C_\mcT^t) \bv_\mcT$$
	\noindent as $C_\mcS^t \mcP_\mcT = C_\mcT^t$ and $\bv_\mcT = \mcP_\mcT \bv$.  By the discussion preceeding lemma \ref{perron decomposition} we know that the spectral radius of $C_\mcT$ is less than $Q$, and so $Q \bI - C_\mcT^t$ is invertible.  Solving for $\bv_\mcT$ and writing $\bv = \mcV_\mcE + \bv_\mcT$ gives the final equivalence and the result follows from the definition of $\mcK(\mcS)$.  When $\mcT = \emptyset,$ the statement is clear from the above reduced normal form of $C_\mcS^t$, and when $\mcS$ is primitive the letters $\alpha \alpha$ for $\alpha \in \mcA$ form a single ergodic class, and so as $\sum u_\alpha = 1$ and $\bv$ is constant on its only ergodic class, it follows that $v_{\alpha \alpha} = 1,$ completing the proof.
\end{proof}
\begin{exmp}\label{TM spectral hull}
	For the Thue-Morse substitution $\tau,$ we computed the ergodic classes of the bisubstitution in example \ref{TM bisub} obtaining $\mcE_0 = \{00,11\}$ and $\mcE_1 = \{01,10\}$, with no transient part.  As $\tau$ is primitive, $\sum v_{\alpha \alpha} u_\alpha = 1$ is equivalent to requiring $w_1 = 1$, so we compute eigenvalues of $\mathring{\bv}$ 
	$$0 = \text{det}(\mathring{\bv} - \lambda \bI) =  \text{det} {\tiny \begin{pmatrix} 1 - \lambda & w_2 \\ w_2 & 1 - \lambda \end{pmatrix}}  = ( 1 - \lambda + w_2 )(1 - \lambda - w_2) \tbump{12px}{$\iff$} \lambda = 1 \pm w_2$$
	so that $\bv \in \mcK(\tau)$ if and only if $\bv = \be_{00} + \be_{11} + w_2 \be_{01} + w_2 \be_{10}$ for $-1 \leq w_2 \leq 1$.  Therefore
	$$\mcK(\tau) = \{ \be_{00} + \be_{11} + w \be_{01} + w \be_{10}  : -1 \leq w \leq 1 \} \tbump{12px} { so that if } \lambda_w := \sigma_{00} + \sigma_{11} + w \sigma_{10} + w \sigma_{01}$$
	then $\lambda_w$ is a $2$-shift invariant probability measure on $\bbT$ for $-1 \leq w \leq 1$ by proposition \ref{K invariance}.	Note that $\lambda_1 = \sum \sigma_{\alpha \beta} = \bdelta_1$, the point mass at $1 \in \bbT$.
\end{exmp}


\subsection{The Spectral Decomposition}\label{queffelec section}
 
We are now prepared to state the main result of our paper.  Recall that theorem \ref{queffelec one} relates the correlation vector $\Sigma$ to the maximal spectral type, and proposition \ref{K invariance} identifies the spectral hull $\mcK$ as a convex set of coefficients for linear combinations of the correlation measures giving rise to $\bq$-shift invariant probability measures $\bv^t \Sigma$.  The following theorem shows that the extreme points $\mcK^\ast$ correspond to extreme points of $\mcM(\bbT^d,\bS_\bq)$, and that any positive linear combination of these ergodic measures give rise to the maximal spectral type by convolution with $\bomega_\bq$, by (\ref{q invariance of Pi}) and proposition \ref{K invariance}.  Queff\'elec proved a similar result for aperiodic primitive substitutions of constant length and trivial height, compare [\ref{queffelec}, Thms 10.1, 10.2, 11.1], which we extend.
\begin{thm}\label{queffelec two}
	If $\mcS$ is an aperiodic $\bq$-substitution then its spectrum is determined by $\Sigma$ and $\mcK^\ast$:
	$$\textstyle{\sigmaX \sim  \bomega_\bq \ast \sum_{\bw \in \mcK^\ast} \lambda_\bw}$$
	and the measures $\lambda_\bw = \bw^t \Sigma$ for $\bw \in \mcK^\ast$ are $\bq$-shift ergodic probability measures on $\bbT^d$.
\end{thm}
\noindent The proof is delayed to \S \ref{appendix} as it relies on a number of details that are not necessary to understand the main results, and are entirely unrelated to our computation and analysis of the spectrum, which we now detail. As the extremal measures $\lambda_\bw$ for $\bw \in \mcK^\ast$ are $\bq$-shift ergodic probability measures, they are mutually singular and cannot be decomposed further into $\bq$-shift invariant components, so that theorem \ref{queffelec two} separates the spectrum of $\mcS$ into finitely many distinct and $\bq$-shift invariant indecomposable components.  Moreover, theorem \ref{queffelec two} tells us that these ergodic measures arise as linear combinations of correlation measures with coefficients in the spectral hull, both of which are computable via the Fourier recursion theorem \ref{fourier recursion} for $\Sigma$, and proposition \ref{characterization of K} for $\mcK^\ast$.  This is discussed at length in \S \ref{examples} where we also include several examples.  Before that, we discuss some consequences of the above result that we can obtain immediately.
\begin{cor}\label{pure types}
For $\bw \in \mcK^\ast,$ the measures $\bomega_\bq \ast \lambda_\bw$ are either pure discrete, pure singular continuous, or Lebesgue measure on $\bbT^d,$ and describe distinct components of the spectrum of $\mcS.$  
\end{cor}
\begin{proof}
	We know from the above theorem that each $\lambda_\bw$ is an ergodic probability measure for the $\bq$-shift on $\bbT^d.$  Therefore, as Lebesgue measure is ergodic for the $\bq$-shift (proposition \ref{spectral qshift identities}), $\lambda_\bw$ is either Lebesgue measure $m$, or purely singular.  
	
	By Dekking's theorem \ref{dekking thm}, the discrete component of $\sigmaX$ is $\bomega_\bq \ast \bnu_\mcL,$ where $\bnu_\mcL$ defined by (\ref{lattice measure}) is a discrete $\bq$-shift invariant measure corresponding to the height lattice, see \S \ref{height}.   By ergodicity, however, the measures $\lambda_\bv$ for $\bv \in \mcK^\ast$ have no (a.e.)-proper invariant subsets and so either $\lambda_\bv \ll \bnu_\mcL$ and is pure discrete, or $\lambda_\bv \perp \bnu_\mcL$ in which case it is purely continuous.
	
	Finally, the singular continuous case follows as it is a mutually singular type to Lebesgue and discrete measures, by definition, and all those measures are ergodic and pure types.  As convolution with $\bomega_\bq$ does not change the purity of a measure and $\bomega_\bq \ast m = m$, $\lambda$ is pure discrete, purely singular continuous, or Lebesgue measure $m,$ respectively, if and only if $\lambda \ast \bomega_\bq$ is as well.   That they describe distinct components follows by theorem \ref{queffelec two} and mutual singularity.
\end{proof}
\begin{exmp}\label{TM spectrum}
	For the Thue-Morse substitution $\tau,$ we know from example \ref{TM sigma} that
	$$\widehat{\Sigma}(1) = \scalebox{1}{$\frac{1}{6}$}(1,2,2,1)^t \tbump{0.5in}{ and } \widehat{\Sigma}(5) = \scalebox{1}{$\frac{1}{4}$}(1,1,1,1)^t$$
	From example \ref{TM spectral hull}, we know that $\mcK^\ast$ consists of two vectors corresponding to $w = \pm 1$ or
	$$\bv_1 = \be_{00} + \be_{11} + \be_{01} + \be_{10} = (1,1,1,1)^t \tbump{8px}{ and } \bv_2 = \be_{00} + \be_{11} - \be_{01} - \be_{10} = (1,-1,-1,1)^t$$
	Thus, $\lambda_{\bv_1} = \sum_{\alpha \beta \in \mcA^2} \sigma_{\alpha \beta} = \bdelta_1$ the Dirac-Delta mass at $1 \in \bbT$ (as $\widehat{\bdelta_1}(k) = 1$ for all $k$), and 
	$$\widehat{\lambda_{\bv_2}}(1) = \bv_2^t \widehat{\Sigma}(1) = -\scalebox{1}{$\frac{1}{3}$} \neq 0 \tbump{0.5in}{ and } \widehat{\lambda_{\bv_2}}(5) = \bv_2^t \widehat{\Sigma}(5) = 0$$
	\noindent so that $\lambda_{\bv_2}$ is not $m$ as $\widehat{m}(\bk) = \bZero$ for $\bk \neq \bZero$, and Thue-Morse has purely singular spectrum.  Morever, as $\widehat{\lambda_{\bv_2}}(1) \neq \widehat{\lambda_{\bv_2}}(5),$ we know $\lambda_{\bv_2}$ is not discrete as its Fourier coefficients are not $1$- or $2$-periodic, the only possible heights for a substitution on $2$ letters.   Thus, $\lambda_{\bv_1} = \bdelta_1$ gives rise to the discrete component, $\lambda_{\bv_2}$ the singular continuous component, and $\sigmaX \sim \bomega_2 + \bomega_2 \ast \lambda_{\bv_2}$.  Thus, Thue-Morse has singular spectrum with a singular continuous component, as expected.
\end{exmp}

Recall that two $\bq$-substitutions are configuration equivalent if they have the same collection of instructions, counted with multiplicity.  A property of a substitution is a \textit{configuration invariant} if all configuration equivalent substitutions share that property or, equivalently, if it does not depend on how the configuration arranges the instructions.  The following is immediate:
\begin{prop}\label{configurations prop}
	For a $\bq$-substitution $\mcS$ on $\mcA,$ the matrices $M_\mcS$ and $C_\mcS,$ the Perron vectors and ergodic decompositions of $\mcS$ and $\mcS \otimes \mcS,$ as well as $\widehat{\Sigma}(\bZero)$ and $\mcK(\mcS),$ are configuration invariants.
\end{prop}
As theorem \ref{queffelec two} relates the spectrum of $\mcS$ to the measures $\lambda_\bv = \bv^t \Sigma$ for $\bv \in \mcK(\mcS),$ the spectral theory of aperiodic $\bq$-substitutions can be separated into the study of their correlation vectors and extremal properties of their spectral hulls.  As $\mcK^\ast$ is a configuration invariant, however, this shows us that any property of the spectrum which depends on the configuration of $\mcS$ is determined by the correlation vector $\Sigma.$     It is immediate that the spectrum of a substitution is \textit{not} invariant with respect to configuration equivalence: not only can the spectrum exist on different dimensional tori, but the height (see \S \ref{height}) of a substitution depends heavily on its configuration.  One can, however, use the Fourier recursion theorem to study the effect changes in configuration have on a given substitution, and it is evident from identities such as (\ref{sigma c}) that the structure of the configuration relative to the carry sets $\Delta_p(\bk)$ accounts for much of this difference.   We now discuss a result on singularity of spectrum for a large collection of $\bq$-substitutions before moving on to discuss the algorithm for computing the measures $\lambda_\bv$ for $\bv \in \mcK^\ast,$ as well as a number of examples.


\subsection{Aperiodic Bijective Commutative $\bq$-Substitutions}\label{abc section}

In the $\bbZ$ case, one can combine [\ref{queffelec}, Prop 3.19 and Thm 8.2] to show that all aperiodic bijective commutative substitutions on $\mcA^\bbZ$ have pure singular spectrum (not explicitly stated by Queff\'elec) and this generalizes to $\bq$-substitutions as well. This generalizes a result of Baake and Grimm in [\ref{baake and grimm}] for $\bq$-substitutions on two letters, noting that all bijective substitutions on two letters are necessarily commutative: there are only two bijective instructions, one of which is the identity.  
\begin{thm}\label{abc theorem}
	Aperiodic bijective commutative $\bq$-substitutions have purely singular spectrum.
\end{thm}
\begin{proof}
The argument is in two steps: first we show that $\sum \sigma_{\alpha \alpha}$ can be expressed as a sum of Riesz products in the commutative case, and second that all such measures are singular to Lebesgue measure.  The result then follows from theorem \ref{queffelec one} as $\sigmaX \sim \Pi\left(\sum_{\alpha \in \mcA} \sigma_{\alpha \alpha}\right)$, which is singular.

To realize the autocorrelations as Riesz products, it is convenient to express $\Sigma$ by its associated matrix $\mathring{\Sigma}$ and use the relation (\ref{associated kronecker product}) to express the Fourier recursion theorem \ref{fourier recursion} as
$$\textstyle \widehat{\mathring{\Sigma}}(\bk) = \lim_{n \to \infty} \frac{1}{Q^n} \sum_{\bj \in [\bZero,\bq^n)} \mcR_\bj^{(n)} \widehat{\mathring{\Sigma}}(\bZero) (\mcR_{\bj+\bk}^{(n)})^\ast$$
\noindent Using the instructions of $\mcS$, define matrix polynomials  $R(\bz)$ and $\mathfrak{R}_n(\bz)$ for $n \geq 0$ and  $\bz \in \bbT^d$ by
$$\textstyle R(\bz) = R(z_1,\ldots, z_d) := \sum_{\bj \in [\bZero,\bq)} \mcR_\bj z_1^{j_1}\cdots z_d^{j_d} \tbump{0.25in}{and} \mathfrak{R}_n(\bz) := R(\bS_\bq^{n-1}(\bz)) \cdots R(\bS_\bq (\bz))R(\bz)$$
\noindent where $\bS_\bq$ is the $\bq$-shift $\bz \mapsto \bz^\bq$ on $\bbT^d.$  Applying lemma \ref{small carries} as in the proof of theorem \ref{fourier recursion} one can express $\mathring{\Sigma}$ as a \textit{matrix} Riesz product (see [\ref{queffelec}, \S 8.1]) using normalized Lebesgue measure $d\bz$ on $\bbT^d$
$$\textstyle \mathring{\Sigma} = \text{w}^\ast\text{-}\lim_{n \to \infty} \frac{1}{Q^n} \mathfrak{R}_n^\ast  \mathfrak{R}_n d\bz$$
\noindent  As the instructions are bijective and commute, the matrices $\mcR_\bj$ are commuting permutation matrices and thus simultaneously unitarily diagonalizable - let $\bP \in \bM_\mcA(\bbC)$ be a unitary matrix diagonalizing each instruction $\mcR_\bj$ of $\mcS.$  Using the Riesz product description of $\mathring{\Sigma}$ above, we write
$$\textstyle \bP \mathring{\Sigma} \bP^\ast =  \text{w}^\ast\text{-}\lim_{n \to \infty}\frac{1}{Q^n} \Lambda(\bz)^\ast \Lambda(\bz^\bq)^\ast \cdots \Lambda(\bz^{\bq^{n-1}})^\ast \Lambda(\bz^{\bq^{n-1}}) \cdots \Lambda(\bz^\bq) \Lambda(\bz)$$
\noindent where $\Lambda(\bz) = \bP R(\bz) \bP^\ast$ is a diagonal matrix polynomial.  Thus, \textit{eigenmeasures} $\lambda$ on the diagonal of $\bP\bS\bP^\ast$ are determined by \textit{eigenpolynomials} $r(\bz)$ on the diagonal of $\bP R(\bz)\bP^\ast$ by $\text{weak}^\ast$ limits
$$\lambda = \textstyle \text{w}^\ast\text{-}\lim_{n \to \infty}\frac{1}{Q^n} \prod_{i = 0}^{n-1} |r(\bS_\bq^i (\bz))|^2 d\bz = \text{w}^\ast\text{-}\lim_{n \to \infty}\frac{1}{Q^n} \prod_{i = 0}^{n-1} |r(\bz^{\bq^i})|^2 d\bz $$
\noindent As $\sum \sigma_{\alpha \alpha} = \text{tr} \, \mathring{\Sigma}$ is the sum of the diagonal measures of $\bP\mathring{\Sigma}\bP^\ast$ we can apply theorem \ref{queffelec one} to express the spectrum in terms of the above measures, which are a type of generalized Riesz product (see [\ref{queffelec}, \S 1.3]).   We now show such measures (weak-star limits as above) are singular to Lebesgue measure, justifying the second claim at the start.   Consider the \textit{principal $\bq$-th root} map 
$$\bz = (e^{i\theta_1}, \ldots, e^{i \theta_d}) \tbump{8px}{$\longmapsto$} \bz^{\nicefrac{\bOne}{\bq}} = ( e^{i\nicefrac{\theta_1}{\bq_1}}, \ldots, e^{i\nicefrac{\theta_d}{q_d}})\tbump{12px}{ for } 0 \leq \theta_i < 2\pi$$
 denoting the push-forward of a measure $\rho$ on $\bbT^d$ under this map by $\rho(\bz^{\nicefrac{\bOne}{\bq}})$.  If $\lambda$ is a weak-star limit as above, then as $\bS_\bq(\bz^{\nicefrac{\bOne}{\bq}}) = \bz$ we can use the definitions to write
	$$\scalebox{0.9}{$\lambda(\bz^{\nicefrac{\bOne}{\bq}}) = \text{w}^\ast\text{-}\lim_{n \to \infty} \frac{1}{Q^n}\prod_{j < n} |r\big((\bz^{\nicefrac{\bOne}{\bq}})^{\bq^j} \big)|^2d(\bz^{\nicefrac{\bOne}{\bq}}) =  \text{w}^\ast\text{-}\lim_{n \to \infty} \frac{1}{Q^n} |r(\bz^{\nicefrac{\bOne}{\bq}}) |^2 \prod_{j < n-1} |r(\bz^{\bq^{j}}) |^2d(\bz^{\nicefrac{\bOne}{\bq}}) $}$$
	so that as the Jacobian of $\bz \mapsto \bz^{\nicefrac{\bOne}{\bq}}$ is $\frac{1}{Q} |\bz^{\nicefrac{\bOne}{\bq} - \bOne} | = \frac{1}{Q}$, we have $d(\bz^{\nicefrac{\bOne}{\bq}}) = \frac{1}{Q} d\bz$ which gives us
	\begin{equation}\label{rho inv}
		\textstyle \lambda(\bz^{\nicefrac{\bOne}{\bq}}) = \frac{1}{Q}  |r(\bz^{\nicefrac{\bOne}{\bq}}) |^2 \left(\text{w}^\ast\text{-}\lim_{n \to \infty} \prod_{j < n-1} r(\bz^{\bq^j}) d\bz \right) = \frac{1}{Q} |r(\bz^{\nicefrac{\bOne}{\bq}})|^2 \, \lambda(\bz)
	\end{equation}
	
	\noindent Denoting the $\bq$-th roots of unity by $\bzeta_\bj$ for $\bj \in [\bZero,\bq)$ so that  $\zeta_{\bj,1}^{k_1}\cdots \zeta_{\bj,d}^{k_d} = 1$ for $\bk \in \bq \Zd$ (see \S \ref{height}),
	\begin{equation}\label{r property} 
		\scalebox{0.9}{$\lambda \circ \bS_\bq^{-1} = \displaystyle{{\sum}_{\bj \in [\bZero,\bq)} \lambda \big( \bzeta_\bj \, \bz^{\nicefrac{\bOne}{\bq}} \big) = \frac{1}{Q} {\sum}_{\bj \in [\bZero,\bq)} |r\big( \bzeta_\bj \, \bz^{\nicefrac{\bOne}{\bq}}\big) |^2 \, \lambda(\bzeta_\bj^\bq \, \bz) = \Big( \frac{1}{Q} {\sum}_{\bj \in [\bZero,\bq)} |r\big( \bzeta_\bj \, \bz^{\nicefrac{\bOne}{\bq}}\big) |^2 \Big) \lambda(\bz) }$}
	\end{equation}
	\noindent  as $\zeta_{\bj,1}^{q_1}\cdots \zeta_{\bj,d}^{q_d}  = 1$ for all $\bj \in [\bZero,\bq)$.   We claim the term in the parenthesis is equal to $1$, so that $\lambda \circ \bS_\bq^{-1} = \lambda$ is $\bq$-shift invariant.  First, note that the sum of the $n$-th roots of unity is $0$, as this is the coefficient of $z^{n-1}$ in their minimal polynomial, and so one checks that
	$$\textstyle \sum_{\bj \in [\bZero,\bq)} \zeta_{\bj,1} \cdots \zeta_{\bj,d} =  0 \tbump{0.25in}{$\implies$} \sum_{\bj \in [\bZero,\bq)} \zeta_{\bj,1}^{k_1} \cdots \zeta_{\bj,d}^{k_d} = \Big \{\scalebox{0.75}{$\begin{matrix} 0 & \text{ if } \bk \notin \bq \Zd \\ Q & \text{ if } \bk \in \bq \Zd \end{matrix}$} $$
	as $\bq \mapsto \bq^\bk$ is a covering map of the $\bq$-th roots of unity onto the $\frac{\bq}{(\bk,\bq)}$-th roots of unity.   As the coefficients of $R(\bz)$ are simultaneously unitarily diagonalizable, there is a unit vector $\bv \in \bbC^{\mcA}$ with
	$$|r(\bz)|^2 = \bv^\ast R(\bz) R(\bz)^\ast \bv = \bv^\ast \left( {\sum}_{\bj,\bk \in [\bZero,\bq)} \mcR_\bj \mcR_\bk^\ast z_1^{j_1 - k_1} \cdots z_d^{j_d-k_d}\right) \bv$$
	If we let $\bw$ denote $\bz^{\nicefrac{\bOne}{\bq}}$, the roots of unity property lets us continue (\ref{r property}) obtaining
	\begin{align*}
		{\sum}_{\bi \in [\bZero,\bq)} |r(\bzeta_\bi \, \bw)|^2 &= \bv^\ast \left({\sum}_{\bj,\bk \in [\bZero,\bq)} \mcR_\bj \mcR_\bk^\ast  \left({\sum}_{\bi \in [\bZero,\bq)} \zeta_{\bi,1}^{j_1-k_1} \cdots \zeta_{\bi,d}^{j_d-k_d} \right)w_1^{j_1 - k_1}\cdots w_d^{j_d - k_d} \right) \bv \\
		&= \bv^\ast {\sum}_{\bj \in [\bZero,\bq)} \mcR_\bj \mcR_\bj^\ast \bv = Q |\bv|^2 = Q
	\end{align*} 
	so that $\lambda$ is $\bq$-shift invariant as claimed.   
	
	As $r(\bz)$ vanishes on a set of $0$ Lebesgue measure, and the push forward map preserves absolutely continuous components, the equality in (\ref{rho inv}) passes to $\lambda_\text{ac}$, the absolutely continuous part of $\lambda,$ and so the equality in (\ref{r property}) holds for $\lambda_\text{ac}$ as well, so it is also invariant.  As $m$ is ergodic for the $\bq$-shift, $\lambda_\text{ac} = m$ or is $0.$  As $m$ only satisfies (\ref{rho inv}) for constant $r(\bz)$, the component $\lambda_\text{ac} = \bZero$ when $r(\bz)$ is not constant.   As the monomials $z_1^{j_1} \cdots z_d^{j_d}$ are linearly independent for $\bj \in [\bZero,\bq)$, and $r(\bz)$ is an eigenfunction of $\bR(\bz) = \sum \mcR_\bj z_1^{j_1} \cdots z_d^{j_d}$, it follows that $r(\bz)$ can be constant if and only if there is a nonzero vector in the kernel of $\mcR_\bj$ for $\bj \neq \bZero.$  As the instruction matrices $\mcR_\bj$ are unitary, this is impossible, so that as $\bq \neq \bZero$ it follows that $r(\bz)$ is not a constant, and $\lambda$ is singular to Lebesgue measure.  As $\lambda$ is $\bq$-shift invariant, $\Pi(\lambda) = \lambda \ast \bomega_\bq$ is also singular.
\end{proof}


\section{Computing the Spectrum}\label{examples}

The goal of this section is to describe an algorithm for determining the spectrum of an aperiodic $\bq$-substitution in a manner consistent with theorem \ref{queffelec two}, using it to compute the spectrum of several examples in a manner similar to our treatment of Thue-Morse.  Specifically, we compute the extreme points of the spectral hull and the Fourier coefficients of the correlation measures, the corresponding linear combinations of which correspond to mutually singular measures of easily characterizable pure types giving rise to the spectrum of $\mcS$.  See [\ref{bartlett}, \S 4] for a more detailed description of the algorithm and examples.

Given an arbitrary $\bq$-substitution, one computes its primitive components as in proposition \ref{PRF} by computing orbits of $\mcS$ on letters and partitioning the alphabet into ergodic classes.  Identification of letter orbits for imprimitive substitutions can be simplified by telescoping the substitution with the index of imprimitivity, computed from the spectrum of the substitution matrix, see proposition \ref{PRF} and the following discussion.  One verifies aperiodicity on each component using Pansiot's lemma \ref{pansiot lemma} in the $\bbZ$ case, or via recognizability in general; as we will be making use of theorems \ref{fourier recursion} and \ref{queffelec two}, aperiodicity is a property we must assume of our substitutions, and so we let $\mcS$ be an aperiodic $\bq$-substitution on an alphabet $\mcA$.  

First, compute the instruction matrices $\mcR_\bj \in M_\mcA(\bbC)$ and the substitution matrix $M_\mcS = \sum_{\bj \in [\bZero,\bq)} \mcR_\bj$.  We let $\bu \in \bbC^\mcA$ be a strictly positive linear combination of its Perron vectors, corresponding to the primitive blocks of its primitive reduced form (\ref{matrix PRF}), obtained by reordering $\mcA$ according to the ergodic classes of $\mcS.$  This corresponds, via theorems \ref{michel theorem} and \ref{invariant measures}, to fixing an invariant measure $\mu$ on $X_\mcS$ supported by all the ergodic measures for the subshift, allowing us to compute the spectrum of $\mcS$ for all the primitive components simultaneously.  

Second, we determine $\mcK^\ast$, the extreme points of the spectral hull, by computing the ergodic classes $\mcE$ of the bisubstitution $\mcS \otimes \mcS$ and the coincidence matrix $C_\mcS = \sum_{\bj \in [\bZero,\bq)} \mcR_\bj \otimes \mcR_\bj$.  Partitioning $\mcA^2$ into letter orbits under the bisubstitution gives one the ergodic classes and transient part, and one applies proposition \ref{characterization of K} to determine the hull; note that if the transient part is empty, the spectral hull depends only on the ergodic classes.  Solving for $\bv$ as a function of $w_1,\ldots, w_J$, one imposes strong semipositivity by computing the eigenvalues of $\mathring{\bv}$ and enforcing nonnegativity, as well as the normalization condition, $\sum_{\alpha \in \mcA} v_{\alpha \alpha} u_\alpha = 1.$

Third, we determine the correlation measures by using the Fourier recursion theorem \ref{fourier recursion} so
$${\small \widehat{\Sigma}(\bk) = \scalebox{1}{$\frac{1}{Q^p}$} {\sum}_{\bj \in [\bZero,\bq^p)} \mcR_\bj^{(p)} \otimes \mcR_{\bj+\bk}^{(p)} \hspace{3px} \widehat{\Sigma}(\qq{\bj+\bk}_p) \tbump{0.25in}{for} p \in \bbN \tbump{0.25in}{and} \bk \in \Zd}$$
\noindent For $\bc \in [-\bOne,\bOne] \smallsetminus \bZero$ we can solve for $\widehat{\Sigma}(\bc)$ algebraically using the above: note that for $m > 0$ and $\bj \in [\bZero,\bq^m)$ the $\Zd$ integer $\qq{\bj + \bc}_m$ lies in the smallest rectangle containing $\bZero$ and $\bc$, as $\bc$ cannot shift $\bj$ in any dimension where $\bc$ is $0$.  For example, if $c =\bOne_i$, the $i$-th coordinate vector in $\Zd$, and $\bj \in [\bZero,\bq^p)$ for $p > 0$ then $\qq{\bj+\bOne_i}_p= \bOne_i$, or $\bZero,$ according to whether $\bj$ is in $\Delta_p(\bOne_i)$, or not, so that
$${\small \widehat{\Sigma}(\bOne_i) = \scalebox{1}{$\frac{1}{Q^p}$} {\sum}_{\bj \in \Delta_p(\bOne_i)} \mcR_\bj^{(p)} \otimes \mcR_{\bj+\bOne_i}^{(p)} \, \widehat{\Sigma}(\bOne_i) + \scalebox{1}{$\frac{1}{Q^p}$} {\sum}_{\bj \in [\bZero,\bq^p)\smallsetminus \Delta_p(\bOne_i)} \mcR_\bj^{(p)} \otimes \mcR_{\bj+\bOne_i}^{(p)} \, \widehat{\Sigma}(\bZero)}$$
from which one can solve for $\widehat{\Sigma}(\bOne_i)$ in order to obtain
\begin{equation}\label{sigma c}
	{\small \widehat{\Sigma}(\bOne_i) = \textstyle{\Big( Q^p \bI - \sum_{\bj \in \Delta_p(\bOne_i)} \mcR_\bj^{(p)} \otimes \mcR_{\bj+\bOne_i}^{(p)} \Big)^{-1} \sum_{\bj \in [\bZero, \bq^p) \smallsetminus \Delta_p(\bOne_i)} \mcR_\bj^{(p)} \otimes \mcR_{\bj+\bOne_i}^{(p)} \widehat{\Sigma}(\bZero)}}
\end{equation}
\noindent where the inverse matrix above exists for \textit{some} $p>0$ as the frequency of carries goes to $0$ and thus the spectral radius of the sum over $\Delta_p(\bOne_i)$ as well.  In this way, we compute $\widehat{\Sigma}(\bc)$ for $\bc = \bZero, \pm \bOne_i$ for $1 \leq i \leq d$.  Next, for $\bc$ of the form $\pm \bOne_i \pm \bOne_j,$ the Fourier recursion expresses $\widehat{\Sigma}(\pm \bOne_i \pm \bOne_j)$ in terms of $\widehat{\Sigma}(\bZero), \widehat{\Sigma}(\pm \bOne_i), \widehat{\Sigma}(\pm \bOne_j),$ and $\widehat{\Sigma}(\pm \bOne_i \pm \bOne_j)$ and again we can solve for $\widehat{\Sigma}(\pm \bOne_i \pm \bOne_j)$ for all $i$ and $j$ distinct; here the $\pm$ can be chosen independently between $i$ and $j$, but are otherwise consistent throughout.  Continue similarly for $\pm \bOne_i \pm \bOne_j \pm \bOne_k$ for distinct $i,j,k$, and etc, until all the values of $\widehat{\Sigma}$ on $[-\bOne,\bOne]$ are known.  For any $\bk \in \Zd$, $\widehat{\Sigma}(\bk)$ can be computed directly in terms of $\widehat{\Sigma}(\bc)$ for $\bc \in [-\bOne,\bOne]$ using Fourier recursion with $p \geq \mfp(\bk)$, or the smallest $p$ with $\bk \in (-\bq^p,\bq^p)$.

Now that $\mcK^\ast$ and $\widehat{\Sigma}(\bk)$ for $\bk \in [\bZero,\bOne]$ are known, we apply theorem \ref{queffelec two} to express the spectrum of $\mcS$ in terms of $\lambda_\bw$ for $\bw \in \mcK^\ast$: as pure, $\bq$-shift ergodic types, they fall into one of three categories:
\begin{itemize}
	\item Lebesgue measure $\iff$ for $\bk \neq \bZero$, we have $\widehat{\lambda_\bw}(\bk) = 0$.
	\item pure discrete $\iff$ $\widehat{\lambda_\bw}$ is constant modulo the height lattice $\mcL$, or periodic in the $\bbZ$ case.
	\item pure singular continuous $\iff$ $\widehat{\lambda_\bw}(\bk) \neq 0$ for some $\bk \neq 0$ and isn't constant on cosets of $\mcL$.
\end{itemize}
For each $\bw \in \mcK^\ast$ and as many $\bk \in \Zd$ as are necessary to classify according to the above, compute
$$\textstyle \widehat{\lambda_\bw}(\bk)  = \sum_{\alpha \beta \in \mcA^2} w_{\alpha \beta} \widehat{\sigma_{\alpha \beta}}(\bk) = \bw^t \widehat{\Sigma}(\bk)$$
\noindent  As convolution with $\bomega_\bq$ has no effect on the purity of the above types, such characterizations of $\lambda_\bw$ will pass on to the spectrum of $\mcS.$  Note that Queff\'elec considers the question of spectral multiplicity in [\ref{queffelec}, \S 11.2.1] where it is related to the rank of $\mathring{\Sigma}$, but we have not analyzed these results in our context.  We now work out several examples, computing the spectrum and showing singularity to Lebesgue when possible.

\subsection{Substititutions with Purely Singular Spectrum}

We have already computed the spectrum of Thue-Morse, showing it to be purely singular, in addition to theorem \ref{abc theorem} showing that all aperiodic bijective and commutative $\bq$-substitutions are singular (of which Thue-Morse is an example).  Our next example, due to Queff\'elec, is bijective but not commutative; in [\ref{queffelec}, Examples 9.3, 10.2.2.3, and 11.1.2.3] it was shown to have Lebesgue spectrum, however there were errors in the analysis - as we now show, it is in fact purely singular.
\begin{exmp}[Queff\'elec's $\zeta$]\label{queffelec example}
	Let $\zeta$ be the $3$-substitution on $\mcA = \{0,1,2\}$ given by
	$$\tiny{\zeta:  \begin{cases} \vspace{12px} \end{cases} \hspace{-8px} \begin{matrix} 0 \longmapsto 001	\\ 1 \longmapsto 122	\\ 2 \longmapsto 210	\end{matrix} \tbump{24px}{\normalsize with } M_\zeta = \begin{pmatrix} 1 & & \\ & 1 & \\ & & 1 \end{pmatrix} +  \begin{pmatrix} 1 & & \\ &  & 1 \\ & 1 &  \end{pmatrix} +  \begin{pmatrix}  & & 1 \\ 1 & & \\ & 1 & \end{pmatrix} = \begin{pmatrix} 2 & 0 & 1 \\ 1 & 1 & 1 \\ 0 & 2 & 1 \end{pmatrix}}$$
	\noindent with the instruction matrices $\mcR_0, \mcR_1, \mcR_2$ appearing sequentially above.  As $M_\zeta^2$ is positive, $\zeta$ is primitive.  As $\mcR_0$ is the identity and as $1$ can be followed by $0,1$ and $2$ in $\zeta^2(0),$ Pansiot's lemma applies and shows that $\zeta$ is aperiodic.  
	
	We compute the ergodic decomposition of $\zeta \otimes \zeta$: looking at letter orbits gives ergodic classes $\mcF_1 = \{ 00, 11, 22\}$ and $\mcF_2 = \{01, 10, 12, 21, 02, 20 \}$ and no transient pairs.  By proposition \ref{characterization of K}
	$$\bv \in \mcK \tbump{12px}{$\implies$}\bv = \big( 1, w, w, w, 1, w, w, w, 1 \big)^t \tbump{8px}{ equivalently } \mathring{\bv} = \tiny{\begin{pmatrix} 1 & w & w \\ w & 1 & w \\ w & w & 1 \end{pmatrix}}$$
	\noindent where the basis of $\bbC^{\mcA^2}$ is ordered lexicographically, $00,01,02,10,11,12,20,21,22$ which will be standard for the bialphabet.  Performing simultaneous row and column operations to $\mathring{\bv}$, we have
	$${\small \bv \in \mcK} \tbump{12px}{$\iff$}\scalebox{0.7}{{$\mathring{\bv} \approx {\small \begin{pmatrix} 1 & 0 & 0 \\ 0 & 1 - w^2 & 0 \\ 0 & 0 & \frac{1 +  w - 2 w^2}{1 + w} \end{pmatrix}}\tbump{8px}{ \normalsize{ so that $\mathring{\bv} \stpos 0$ implies }} \begin{cases} (1 + w)(1 - w) \geq 0, \\ (1 - w)(1 + 2w) \geq 0 \end{cases}$}}$$
	\noindent or $-\frac{1}{2} \leq w \leq 1,$ and so $\mcK$ has two extreme points given by the vectors
	$$\small{\bv_1 = \vec{\mcF_1} + \vec{\mcF_2} = \bOne, \tbump{8px}{ and } \bv_2 = \vec{\mcF_1} - \frac{1}{2} \vec{\mcF_2} = \big(1, -\frac{1}{2}, -\frac{1}{2}, -\frac{1}{2}, 1, -\frac{1}{2}, -\frac{1}{2}, -\frac{1}{2}, 1 \big)^t}$$
	\noindent The Perron vector of $M_\zeta$ is $\frac{1}{3} \bOne$ and so $\widehat{\Sigma}(0) = \frac{1}{3} \sum_{\gamma \in \mcA} \be_{\gamma \gamma}.$  As $\Delta_1(1) = \{2\}$ for $q = 3$, (\ref{sigma c}) gives
	\begin{align*}
	\widehat{\Sigma}(1) &= \textstyle{\Big( 3 \bI - \sum_{j \in \Delta_1(1)} \mcR_j^{(1)} \otimes \mcR_{j+1}^{(1)} \Big)^{-1} \sum_{j \notin \Delta_1(1)} \mcR_j^{(1)} \otimes \mcR_{j+1}^{(1)} \widehat{\Sigma}(0)} \\
		&= \textstyle{\big( 3 \bI - \mcR_2 \otimes \mcR_0 \big)^{-1} \big( \mcR_0 \otimes \mcR_1 + \mcR_1 \otimes \mcR_2 \big) \widehat{\Sigma}(0)} = \frac{1}{39}\big(5,6,2,6,2,5,2,5,6\big)^t
	\end{align*}
	\noindent Computing $\widehat{\Sigma}(2)$ using $p=1$ in the Fourier recursion theorem \ref{fourier recursion} gives
	{\small \begin{align*}
		\widehat{\Sigma}(2) &= \scalebox{1}{$\frac{1}{3}$} {\sum}_{j=0}^2 \mcR_j \otimes \mcR_{j+2} \widehat{\Sigma}(\qr{j + 2}_1) = \scalebox{1}{$\frac{1}{3}$}\mcR_0 \otimes \mcR_0 \widehat{\Sigma}(0) + \scalebox{1}{$\frac{1}{3}$}(\mcR_1 \otimes \mcR_0 + \mcR_2 \otimes \mcR_1) \widehat{\Sigma}(1)  \\
			&= \scalebox{1}{$\frac{1}{117}$}\big(7,7,25,25,7,7,7,25,7\big)^t
	\end{align*}}
	Combining the above, one checks that $\lambda_{\bv_1} = \bdelta_1,$  and that $\widehat{\lambda_{\bv_2}}(1) = 0,$ and thus $\widehat{\lambda_{\bv_2}}(3a) = 0$ for all $a$ by $3$-shift invariance.   Using the value of $\widehat{\Sigma}(2)$ above, however, we obtain 
	$$\scalebox{1}{$\widehat{\lambda_{\bv_2}}(2) = \frac{1}{117}\left( 1 \big( 7 \cdot 3 \big) - \frac{1}{2}\big(7 \cdot 3 + 25 \cdot 3\big)\right) \neq 0$}$$
	\noindent so that $\lambda_{\bv_2}$ is \textit{not} Lebesgue measure.  As $\sigmaX(\zeta) = \bomega_3 \ast ( \bdelta_1 + \lambda_{\bv_2} ),$ it follows that the spectrum of $\zeta$ is purely singular to Lebesgue spectrum on the circle $\bbT,$ as both $\bomega_3$ and $\bomega_3 \ast \lambda_{\bv_1} \perp \mathbf{m}.$	
\end{exmp}

Our first example with $d > 1$ comes from a domino tiling system in the plane known as the Table.  In [\ref{robinson}], Robinson described it as a substitution on $4$ symbols in $\bbZ^2$ as follows:
\begin{exmp}[The Table]\label{table example}
Let $\mcT$ be the $(2,2)$-substitution on $\{0,1,2,3\}$ given by
$$\small{\mcT: \hspace{28px} 0 \mapsto \begin{matrix} 3 & \hspace{-0.1in} 0 \\ 1 &  \hspace{-0.1in} 0 \end{matrix} \hspace{28px} 1 \mapsto \begin{matrix} 1 &  \hspace{-0.1in} 1 \\ 0 &  \hspace{-0.1in} 2 \end{matrix}  \hspace{28px} 2 \mapsto \begin{matrix} 2 &  \hspace{-0.1in} 3 \\ 2 &  \hspace{-0.1in} 1 \end{matrix} \hspace{28px}  3 \mapsto \begin{matrix} 0 &  \hspace{-0.1in} 2 \\ 3 &  \hspace{-0.1in} 3 \end{matrix}} $$
\noindent with substitution matrix $M_\mcT = \mcR_{(0,0)} + \mcR_{(0,1)} + \mcR_{(1,0)} + \mcR_{(1,1)}$ summed consecutively below:
$${\tiny M_\mcT =  \scalebox{0.8}{$\begin{pmatrix} & 1 & & \\ 1 & & & \\ & & 1 & \\ & & & 1 \end{pmatrix} + \begin{pmatrix} & & & 1 \\ & 1 & & \\ & & 1 & \\ 1 & & & \end{pmatrix} + \begin{pmatrix} 1 & & & \\ & & 1 & \\ & 1 & & \\ & & & 1 \end{pmatrix} +  \begin{pmatrix} 1 & & & \\ & 1 & & \\ & & & 1 \\ & & 1 & \end{pmatrix} = \begin{pmatrix} 2 & 1 & & 1 \\ 1 & 2 & 1 & \\ & 1 & 2 & 1 \\ 1 & & 1 & 2 \end{pmatrix}$}}$$
\noindent As $M_\mcT^2 > 0$, $\mcT$ is primitive, and aperiodicity follows from recognizability via theorem \ref{aperiodicity implies recognizability}, see [\ref{robinson}].  Although $\mcT$ is a bijective substitution, $R_{(0,0)}$ and $R_{(0,1)}$ do not commute, and so theorem \ref{abc theorem} does not apply.   Computing the spectral hull, one checks from letter orbits that the ergodic classes of the bisubstitution are: 
$$\mcF_1 = \{00,11,22,33\} \tbump{8px}{ and } \mcF_2 = \{01,02,03,10,12,13,20,21,23,30,31,32\}$$
\noindent which partition $\mcA^2$ completely.  Using proposition \ref{characterization of K}, we have $\bv \in \mcK$ if and only if
$$\scalebox{0.9}{$\tiny \mathring{\bv} = \begin{pmatrix} 1 & w & w & w \\ w & 1 & w & w \\ w & w & 1 & w \\ w & w & w & 1 \end{pmatrix} $} \stpos 0 \tbump{8px}{ \normalsize{$\implies$} } \scalebox{0.9}{$\small \begin{cases} (1 + w)(1 - w) \geq 0 \\ (1 - w)^3(1 + 3w) \geq 0 \end{cases}$} $$
\noindent as $\mathring{\bv} \stpos 0$ if and only if its principal minors are positive definite, and so $\mcK^\ast$ consists of two points:
$${\textstyle \bv_1 = \sum_{\alpha \in \mcA} \be_{\alpha \alpha} + \sum_{\gamma \neq \delta \in \mcA} \be_{\gamma \delta} \tbump{8px}{ and } \bv_2 = \sum_{\alpha \in \mcA} \be_{\alpha \alpha} -\frac{1}{3} \sum_{\gamma \neq \delta \in \mcA} \be_{\gamma \delta}}$$
One checks that the Perron vector of $M_\tau$ is $\bu = (\frac{1}{4}, \frac{1}{4}, \frac{1}{4}, \frac{1}{4})^t$ and so $\widehat{\Sigma}(0,0) = \frac{1}{4} \sum_{\gamma \in \mcA} \be_{\gamma \gamma}.$   Writing $\mcA^2$ in the lexicographic order and using (\ref{sigma c}) to solve for $\widehat{\Sigma}(1,0)$ one obtains
$$\widehat{\Sigma}(1,0) = \scalebox{1}{$\frac{1}{20}$}\big(0,2,1,2,0,2,2,1,5,0,0,0,0,1,2,2\big)^t$$  
\noindent so that combining the above we obtain the Fourier coefficients
$$\widehat{\lambda_{\bv_1}}((1,0)) = 1 \tbump{8px}{ and } \widehat{\lambda_{\bv_2}}((1,0)) = \bv_2^t \widehat{\Sigma}((1,0)) = \scalebox{1}{$-\frac{1}{15}$},$$
\noindent so that none of the measures coming from the spectral hull are Lebesgue, and thus the spectrum of the Table is singular to Lebesgue measure on $\bbT^2,$ the two-torus.
\end{exmp}
Note that both $\zeta$ and $\mcT$ are examples of bijective $\bq$-substitutions, as all of their instructions are bijections of $\mcA.$  The correction of Queff\'elec's example is significant, as it represented the only known case of a bijective substitution with Lebesgue component.   In addition to the above examples, computational software has been used to exclude Lebesgue component from the spectrum of all bijective $q$-substitutions for $q \leq 5$ (the $\bbZ$ case)  on alphabets of at most $5$ letters.  Together with theorem \ref{abc theorem}, this suggests that all aperiodic bijective $\bq$-substitutions may have spectrum singular to Lebesgue measure.  We now discuss a collection of (non-bijective) $\bq$-substitutions, due to Frank [\ref{frank lebesgue}], which have Lebesgue component in their spectrum and are based on the famous Rudin-Shapiro substitution.


\subsection{Substitutions with Lebesgue Component in Spectrum}

Although the spectrum of the Rudin-Shapiro substitution is well known, it is nonetheless interesting to see how the details work out allowing for all the terms $\bv^t \widehat{\Sigma}(k)$ for $k \neq 0$ to vanish.  
\begin{exmp}[Rudin-Shapiro]\label{RS example}
	Let $\rho$ be the $2$-substitution on $\{0,1,2,3\}$ 
	$${\tiny \rho:  \begin{cases} \vspace{20px} \end{cases} \hspace{-8px}\begin{matrix} 0 \longmapsto 02 \\ 1 \longmapsto 32 \\ 2 \longmapsto 01 \\ 3 \longmapsto 31 \end{matrix} \tbump{24px}{\normalsize{ with }} M_\rho = \scalebox{0.9}{$\begin{pmatrix} 1 & 0 & 1 & 0 \\ 0 & 0 & 0 & 0 \\ 0 & 0 & 0 & 0 \\ 0 & 1 & 0 & 1 \end{pmatrix} + \begin{pmatrix} 0 & 0 & 0 & 0 \\ 0 & 0 & 1 & 1 \\ 1 & 1 & 0 & 0 \\ 0 & 0 & 0 & 0 \end{pmatrix}$} = \scalebox{0.9}{$\begin{pmatrix} 1 & 0 & 1 & 0 \\ 0 & 0 & 1 & 1 \\ 1 & 1 & 0 & 0 \\ 0 & 1 & 0 & 1 \end{pmatrix}$}}$$
	\noindent As $M_\rho^3 > 0,$ $\rho$ is primitive, and as $\rho^2(0) = 0201$ with $\rho(0)_0 = 0$ the symbol $0$ can be proceeded by both the symbols $1$ and $2,$ it follows from Pansiot's lemma that $\rho$ is aperiodic.  We now compute the ergodic decomposition of $\rho \otimes \rho,$ the bisubstitution.  Here $\mcF_1 = \{ 00, 11, 22, 33\},$ and in this case the only other ergodic class is $\mcF_2 = \{ 03, 12, 21, 30 \},$ so that $\mcT = \{01,02,10,13,20,23,31,32 \}$ is the transient part.  Using proposition \ref{characterization of K} we have $\bv \in \mcK$ if and only if
	$$\mathring{\bv} = \scalebox{0.9}{\tiny $\begin{pmatrix} 1 &  &  &  \\  & 1 &  &  \\  &  & 1 &  \\  &  &  & 1 \end{pmatrix} + w \begin{pmatrix}  & & & 1 \\  &  & 1 &  \\  & 1 &  &  \\ 1 &  &  &  \end{pmatrix} + \frac{1}{2}(1+w) \begin{pmatrix}  & 1 & 1 &  \\ 1 &  &  & 1 \\ 1 &  & &  1 \\  & 1 & 1 &  \end{pmatrix} \stpos$} 0$$
	\noindent We can diagonalize $\mathring{\bv}$ as the above constant matrices are simultaneously diagonalizable, using
	$$\scalebox{0.9}{${\tiny S = \begin{pmatrix} 1 & 0 & -1 & 1 \\ -1 & -1 & 0 & 1 \\ -1 & 1 & 0 & 1 \\ 1 & 0 & 1 & 1 \end{pmatrix}} \tbump{12px}{ \normalsize{we have} } 
	{\tiny S^{-1} \mathring{\bv} S = \begin{pmatrix} 0 & & & \\ & 1 - w & & \\ & & 1 - w & \\ & & & 2 + 2w \end{pmatrix}}$}$$
	\noindent so that $\bv$ is strongly semipositive if and only if $-1 \leq w \leq 1$ and so $\mcK^\ast$ is given by the vectors
	$$\bv_1 = \bOne \tbump{8px}{ and } \bv_2 = (1,0,0,-1,0,1,-1,0,0,-1,1,0,-1,0,0,1)^t$$
	\noindent The Perron vector is $\frac{1}{4}\bOne$ so $\widehat{\Sigma}(0) = \sum_{\gamma \in \mcA} \frac{1}{4} \be_{\gamma \gamma}.$  As $\Delta_1(1) = \{1\}$ when $q = 2$, identity (\ref{sigma c}) gives
	$$\widehat{\Sigma}(1) = (2\bI - \mcR_1 \otimes \mcR_0)^{-1} \mcR_0 \otimes \mcR_1 \widehat{\Sigma}(0) = \scalebox{0.8}{$\frac{1}{8}$}(0,1,1,0,1,0,0,1,1,0,0,1,0,1,1,0)^t$$
	\noindent where the basis of $\bbC^{\mcA^2}$ is ordered lexicographically. Using theorem \ref{fourier recursion} for $k = 2$ and $p=1$ 
	$$\widehat{\Sigma}(2) = \scalebox{0.8}{$\frac{1}{2}$} (\mcR_0 \otimes \mcR_0 + \mcR_1 \otimes \mcR_1) \widehat{\Sigma}(1) = \scalebox{0.8}{$\frac{1}{8}$}(1,0,0,1,0,1,1,0,0,1,1,0,1,0,0,1)^t$$
	\noindent Note that $\widehat{\Sigma}(1) \perp \widehat{\Sigma}(2)$: the former lies in the span of transient vectors $\be_{\alpha \beta}$ for $\alpha \beta \in \mcT$, and the latter in the span of the ergodic vectors $\be_{\alpha \beta}$ for $\alpha \beta \in \mcE_1 \sqcup \mcE_2$.   By explicit computation and the Fourier recursion theorem \ref{fourier recursion} applied with $p=1$, one verifies the following relations for $n \in \bbZ$ 
	$$\scalebox{0.8}{$\begin{cases} 
	\widehat{\Sigma}(2n) =  \frac{1}{2} \left( \mcR_{0} \otimes \mcR_{0} + \mcR_{1} \otimes \mcR_1 \right) \widehat{\Sigma}(n) = \frac{1}{2} C_\rho\widehat{\Sigma}(n)
	\\
	 \frac{1}{2} C_\rho \widehat{\Sigma}(0) = \widehat{\Sigma}(0) 
	 	\\
	 \frac{1}{2}C_\rho \widehat{\Sigma}(1) = \frac{1}{2} C_\rho \widehat{\Sigma}(2) = \widehat{\Sigma}(2) 
 
	 \end{cases} \hspace{-0.3in}\tbump{8px}{and} \begin{cases} 
	 	\widehat{\Sigma}(2n+1)  =  \frac{1}{2} \big( \mcR_{0} \otimes \mcR_1 \widehat{\Sigma}(n) + \mcR_{1} \otimes \mcR_0 \widehat{\Sigma}(n+1)\big)
		\\
	\frac{1}{2} \big( \mcR_{0} \otimes \mcR_1 \widehat{\Sigma}(1) + \mcR_{1} \otimes \mcR_0 \widehat{\Sigma}(2)\big) =  \widehat{\Sigma}(1) 
	 \\
	 \frac{1}{2} \big( \mcR_{0} \otimes \mcR_1 \widehat{\Sigma}(2) + \mcR_{1} \otimes \mcR_0 \widehat{\Sigma}(1)\big)= \widehat{\Sigma}(1) 
\end{cases}$}$$ 
	Using these identities and relations, one checks recursively that the identities $\widehat{\Sigma}(2n+1) = \widehat{\Sigma}(1)$ for $n \in \bbZ$ and $\widehat{\Sigma}(2n) = \widehat{\Sigma}(2)$ for $n \neq 0$ are satisfied.   This can be used to show that $\Sigma$ is a vector linear combination of Lebesgue measure and discrete measures supported at $\pm 1$.
	
	As usual, $\lambda_{\bv_1} = \bdelta_1,$ and using the computed values of $\widehat{\Sigma}(k),$ one checks that $\widehat{\lambda_{\bv_2}}(k) = 0$ for $k \neq 0,$ and so $\lambda_{\bv_2}$ is Lebesgue measure.  As $m$ is $q$-shift invariant for all $q,$ we have $\sigmaX \sim \bomega_2 + m.$
\end{exmp}	

We now briefly discuss Frank's generalizations of Rudin-Shapiro.  A Hadamard matrix $\bH$ is an even dimensional matrix whose coefficients are $\pm 1$ with mutually orthogonal rows and columns.  Given a Hadamard matrix $\bH \in \bM_{2n}(\pm 1)$, consider the signed alphabet $\mcA = \{ \pm k \text{ for } 1 \leq k \leq 2n \}$ and define the \textit{induced instructions} of $\bH$ using the columns of $\bH$ to alternate the sign of letters: 
$$\mcI(\bH) := \big\{ \mcR : \mcA \to \mcA \tbump{4px}{ sending } \alpha \mapsto  \big(\text{sign}(\alpha) \, \bH_{|\alpha|, k} \big)  \, k \tbump{4px}{ for } 1 \leq k \leq 2n \big\}$$
\noindent If $\bq \in \Zd$ satisfies $Q = 2n,$ then any configuration $\mcR$ of $[\bZero,\bq)$ onto $\mcI(\bH)$ gives rise to a $\bq$-substitution on $\mcA$.  For example, in the $n=1$ case with Hadamard matrix $\bH$ and corresponding induced instructions $\mcR_0$ and $\mcR_1$ given by
$$\bH = \scalebox{0.7}{$\begin{pmatrix} 1 & -1 \\ -1 & 1 \end{pmatrix} \tbump{10px}{with} \mcR_0 : \begin{cases}   \hspace{8px} 1 \mapsto   \hspace{8px}  1, &   \hspace{8px} 2 \mapsto \hspace{8px}  1 \\ -1 \mapsto -1, & -2 \mapsto   -1 \end{cases} \tbump{10px}{and} \mcR_1 : \begin{cases}  \hspace{8px} 1 \mapsto -2, &  \hspace{8px} 2 \mapsto  - 2 \\ -1 \mapsto  \hspace{8px}  2, & -2 \mapsto \hspace{8px} 2 \end{cases}$}$$
corresponds to Rudin-Shapiro substitution by sending $1 \mapsto 0, -1 \mapsto 3, 2 \mapsto 1,$ and $-2 \mapsto 2.$
\begin{thm}[Frank]
	If $\bH$ is a $Q \tbump{-4px}{$\times$} Q$-Hadamard matrix with $Q = q_1 \cdots q_d$, then any $\bq$-substitution $\mcS$ with induced instructions $\mcI(\bH)$ has multiplicity $Q$ Lebesgue spectral components.
\end{thm} 
As suggested in [\ref{frank lebesgue}, \S 5.1], any configuration of the instructions $\mcI(\bH)$ give rise to substitutions in $\Zd$ with Lebesgue spectrum.  As convolution with $\bomega_\bq$ has no effect on absolutely continuous spectrum, theorem \ref{queffelec two} tells us that presence of Lebesgue component is determined by the spectral hull and the correlation vector, the first of which is already configuration invariant.  This raises an interesting question: is presence of absolutely continuous component a configuration invariant in general?  If the configuration $\mcR$ gives a $\bq$-substitution with Lebesgue component in its spectrum, then there is a $\bv$ in the spectral hull for which $\widehat{\Sigma}(\bk)$ is orthogonal to $\bv$ for all $\bk \neq \bZero.$  As a change in configuration is a relabeling of the indices of the instructions, one can use theorem \ref{fourier recursion} to compare correlation vectors for substitutions with different configurations on the same instructions.

\subsection{Substitutions with Every Spectral Component}\label{substitution product section}

As the instructions for a substitution product are the Kronecker products of the instructions of its factors, and ergodic classes depend only on the orbits of the generalized instructions, it follows from the mixed product property that the ergodic classes for the bisubstitution partition the pairwise products of the ergodic classes of its factors and gives the following:
\begin{prop}\label{ergodic substitution product}
	Let $\mcS$ and $\tilde{\mcS}$ be $\bq$-substitutions on the alphabets $\mcA$ and $\tilde{\mcA}.$   Let $\mcE$ and $\tilde{\mcE}$ denote the ergodic classes of their respective bisubstitutions, with transient classes $\mcT$ and $\tilde{\mcT}.$  Then the ergodic classes of the bisubstitution of $\mcS \otimes \tilde{\mcS}$ partition the alphabets $\mcE_i \tilde{\mcE}_j$ as $i,j$ range over the indices for the respective bisubstitutions, and its transient part is given by $\mcA \tilde{\mcT} \cup \mcT \tilde{\mcA}.$
\end{prop}
\noindent Our last example is a substitution possessing all three pure types in its spectrum and allows us to illustrate an interesting property of the substitution product, see also [\ref{BGG}, \S 2].
\begin{exmp}\label{substitution product example}
Consider the $2$-substitutions $\tau$ (Thue-Morse) and $\rho$ (Rudin-Shapiro) represented on the alphabets $\mcA_\tau = \{\bar{\hspace{4px}},\underline{\hspace{4px}}\bump\}$ and $\mcA_\rho = \{a,b,c,d\}$ respectively, by
$$\scalebox{0.9}{$\tau: \begin{cases} \vspace{8px} \end{cases} \hspace{-8px} \begin{matrix} \bar{\hspace{4px}} \longmapsto \bar{\hspace{4px}} \bump \underline{\hspace{4px}} \\ \underline{\hspace{4px}} \longmapsto \underline{\hspace{4px}} \bump \bar{\hspace{4px}} \end{matrix} \tbump{8px}{ and } \rho :  \begin{cases} \vspace{8px} \end{cases} \hspace{-8px} \begin{matrix} a \longmapsto ac, & b \longmapsto dc \\ c \longmapsto ab, & d \longmapsto db \end{matrix}$}$$
\noindent and consider the substitution $\mcS$ of constant length $2$ on $\mcA = \{\bar{a},\bar{b},\bar{c},\bar{d}, \underline{a},\underline{b}, \underline{c}, \underline{d}\},$  given by
$$\scalebox{0.9}{$\mcS:  \begin{cases} \vspace{8px} \end{cases} \hspace{-8px} \begin{matrix} \bar{a} \longmapsto \bar{a} \underline{c}, \hspace{16px} \bar{b} \longmapsto \bar{d} \underline{c},  \hspace{16px}   \bar{c} \longmapsto \bar{a} \underline{b},  \hspace{16px}  \bar{d} \longmapsto \bar{d} \underline{b} \\ 
 \underline{a} \longmapsto \underline{a}\bar{c},  \hspace{16px} \underline{b} \longmapsto \underline{d} \bar{c},  \hspace{16px}  \underline{c} \longmapsto \underline{a} \bar{b},  \hspace{16px}  \underline{d} \longmapsto \underline{d} \bar{b} \end{matrix}$} $$
\noindent which is equivalent to the substitution product $\tau \otimes \rho$ via the obvious map $\mcA_\tau \mcA_\rho \to \mcA.$  Note that it is aperiodic being the substitution product of aperiodic substitutions, and primitivity follows as $M_{\mcS}^5$ is positive.  The spectrum is computed using our algorithm in detail in [\ref{bartlett}], where we use the above proposition to show that, in this case, the spectrum of the substitution product is equivalent to the sum of the spectra of its factors, or $\sigma_{\tau \otimes \rho} \sim \sigma_\tau + \sigma_\rho,$ so that $\tau \otimes \rho$ has discrete, singular continuous, and Lebesgue components in its spectrum.  These claims are confirmed by work of Baake, Gahler, and Grimm in [\ref{BGG}], where they consider an identical substitution (although there it is not formulated as a substitution product).  In general, it would be interesting to know the relationship between the spectrum of a substitution product and its factors and raises an interesting question: are there any substitutions with Lebesgue spectrum which do not arise from substitution product with substitutions of Frank type or their factors?
\end{exmp}


\subsection{Substitutions with Nontrivial Height}\label{height example section}

We now describe a class of aperiodic bijective commutative $\bq$-substitutions which attain any height lattice of the form $\bh \Zd.$  Fix $\bh \geq \bOne$ in $\Zd$ and let $\mcA = \mcA_\bh := \Zd / 2\bh \Zd$ be the quotient ring of $\Zd$ integers modulo $2\bh,$ using the residue class $[\bZero, 2\bh)$ to represent the letters.  Let $\bq := \bh + \bOne,$ so that $[\bZero,\bq) = [\bZero,\bh]$ and consider the $\bq$-substitution $\mcH := \mcH_\bh$ on $\mcA_\bh$ determined by the instructions
$$\mcR_\bk : \mcA_\bh \to \mcA_\bh \tbump{12px}{ with } \mcR_\bk : \alpha \longmapsto \alpha + \bk \bump (\text{mod } 2\bh) \tbump{12px}{ for } \bk \in [\bZero,\bh+\bOne) \text{ and } \bZero \leq \alpha \leq 2\bh $$
As cyclic permutations, the instructions of $\mcH_\bh$ are bijective and commute.  Moreover, these substitutions have height $\bh \Zd$ as discussed in [\ref{bartlett}, \S 4.4] where they are also shown to be recognizable and therefore aperiodic.  Although theorem \ref{abc theorem} tells us that such substitutions are purely singular, they are interesting as they are nontrivial height and are the only example we know of where $\mcK^\ast$ contains complex vectors.  Note that the Thue-Morse substitution $\tau = \mcH_1$, and is the only substitution in the above class with trivial height.  The following example has height $3$ and is discussed in [\ref{bartlett}, Ex 4.4.2] in detail; we summarize the results here.

\begin{exmp}\label{height example}
Let $\mcH = \mcH_3$ be the $4$-substitution of height $3$ described above, and given by
$${\tiny \mcH : \begin{cases} \vspace{0.25in} \end{cases} \hspace{-8px} \begin{matrix} 0 \longmapsto 0123 & & 3 \longmapsto 3450 \\ 1 \longmapsto 1234 & & 4 \longmapsto 4501  \\ 2 \longmapsto  2345 & & 5 \longmapsto 5012\end{matrix} } \tbump{12px}{ with instructions } \mcR_j : \nicefrac{\bbZ}{6 \bbZ} \to \nicefrac{\bbZ}{6 \bbZ} \tbump{12px}{ taking } \alpha \mapsto \alpha + j$$
\noindent As shown in [\ref{bartlett}], the ergodic classes of the bisubstitution of $\mcH_\bh$ correspond to the (permutation) matrices $\mcR_\bj$, so that a left $Q$-eigenvector $\bv$ of $C_{\mcH}$ has an associated matrix of  the form 
$${\small \mathring{\bv} = \sum_{j = 0}^5 w_j \mcR_j =  \scalebox{0.7}{$\begin{bmatrix} w_0 & w_1 & w_2 & w_3 & w_4 & w_5 \\  w_5 & w_0 & w_1 & w_2 & w_3 & w_4 \\ w_4 & w_5 & w_0 & w_1 & w_2 & w_3 \\  w_3 & w_4 & w_5 & w_0 & w_1 & w_2 \\  w_2 & w_3 & w_4 & w_5 & w_0 & w_1 \\ w_1 & w_2 & w_3 & w_4 & w_5 & w_0 \end{bmatrix}$}
 \tbump{12px}{ so that }
\mcK^\ast = \scalebox{0.6}{$\begin{cases}  \bv_1 = (1,1,1,1,1,1) \\ 
	 			\bv_2 = (1,-\frac{1}{2} - \frac{\sqrt{3}}{2}i, -\frac{1}{2} +\frac{\sqrt{3}}{2}i, 1, -\frac{1}{2} - \frac{\sqrt{3}}{2}i, -\frac{1}{2} + \frac{\sqrt{3}}{2}i ) \\  
				\bv_3 = (1, -\frac{1}{2} + \frac{\sqrt{3}}{2}i, -\frac{1}{2} - \frac{\sqrt{3}}{2}i, 1, -\frac{1}{2} + \frac{\sqrt{3}}{2}i, -\frac{1}{2} - \frac{\sqrt{3}}{2}i)  \\ 
				\bv_4 = (1, -1, 1, -1, 1, -1)  \\ 
				\bv_5 = (1,\frac{1}{2} -\frac{\sqrt{3}}{2}i,-\frac{1}{2} - \frac{\sqrt{3}}{2}i, -1, -\frac{1}{2}+\frac{\sqrt{3}}{2}i, \frac{1}{2} + \frac{\sqrt{3}}{2}i) \\ 
				\bv_6  = (1, \frac{1}{2} + \frac{\sqrt{3}}{2}i, -\frac{1}{2} + \frac{\sqrt{3}}{2}i, -1,-\frac{1}{2} - \frac{\sqrt{3}}{2}i, \frac{1}{2} - \frac{\sqrt{3}}{2}i) \end{cases}$}}$$
\noindent Now, as $\mcH$ is bijective, $\frac{1}{6}M_\mcH$ is both row and column stochastic, so its Perron vector is $\frac{1}{6} \bOne,$ and so $\widehat{\Sigma}(0) = \sum_{\gamma \in \mcA} \frac{1}{6} \be_{\gamma \gamma}.$  Using (\ref{sigma c}) and theorem \ref{fourier recursion} we compute $\widehat{\Sigma}(\bk)$ (in matrix form for space)
$$\scalebox{0.9}{$\widehat{\Sigma}^\circ(0) = \frac{1}{6} \mcR_0, \hspace{12px} \widehat{\Sigma}^\circ(1) = \frac{1}{30} \mcR_2 + \frac{4}{30} \mcR_5, \hspace{12px} \widehat{\Sigma}^\circ(2) = \frac{2}{30}\mcR_1 + \frac{3}{30}\mcR_4, \hspace{12px} \widehat{\Sigma}^\circ(3) = \frac{3}{30} \mcR_0 + \frac{2}{30} \mcR_3$}$$
\noindent Letting $\lambda_j := \lambda_{\bv_j} = \bv_j^t \Sigma,$ we compute the Fourier coefficients and obtain
\vspace{-0.06in}
\begin{multicols}{6}
{\tiny
\hspace{-0.25in}$\widehat{\lambda_1}(1) = 1$  
\vspace{6px}

\hspace{-0.25in}$\widehat{\lambda_1}(2) = 1$ 
\vspace{6px}

\hspace{-0.25in}$\widehat{\lambda_1}(3) = 1$ 
\vspace{6px}

\hspace{-0.55in}$\widehat{\lambda_2}(1) = -\frac{1}{2} - \frac{\sqrt{3}}{2}i$ 
\vspace{6px}

\hspace{-0.55in}$\widehat{\lambda_2}(2) = -\frac{1}{2} - \frac{\sqrt{3}}{2}i$ 
\vspace{6px}

\hspace{-0.55in}$\widehat{\lambda_2}(3) = 1$ 
\vspace{6px}

\hspace{-0.4in}$\widehat{\lambda_3}(1) = -\frac{1}{2} + \frac{\sqrt{3}}{2}i$  
\vspace{6px}

\hspace{-0.4in}$\widehat{\lambda_3}(2) = -\frac{1}{2} + \frac{\sqrt{3}}{2}i$  
\vspace{6px}

\hspace{-0.4in}$\widehat{\lambda_3}(3) = 1$ 
\vspace{6px}

\hspace{-0.25in}$\widehat{\lambda_4}(1) = -\frac{3}{5}$ 
\vspace{6px}

\hspace{-0.25in}$\widehat{\lambda_4}(2) = \frac{1}{5}$ 
\vspace{6px}

\hspace{-0.25in}$\widehat{\lambda_4}(3) = \frac{1}{5}$ 
\vspace{6px}

\hspace{-0.45in}$\widehat{\lambda_5}(1) = \frac{3}{10} + \frac{3\sqrt{3}}{10}i$ 
\vspace{6px}

\hspace{-0.45in}$\widehat{\lambda_5}(2) = -\frac{1}{10} + \frac{\sqrt{3}}{10}i$ 
\vspace{6px}

\hspace{-0.45in}$\widehat{\lambda_5}(3) =  \frac{1}{5}$ 
\vspace{6px}

\hspace{-0.25in}$\widehat{\lambda_6}(1) = \frac{3}{10} - \frac{3\sqrt{3}}{10}i$  
\vspace{6px}

\hspace{-0.25in}$\widehat{\lambda_6}(2) = -\frac{1}{10} - \frac{\sqrt{3}}{10}i$ 
\vspace{6px}

\hspace{-0.25in}$\widehat{\lambda_6}(3) = \frac{1}{5} $ 
\vspace{6px}
}
\end{multicols}

\vspace{-0.26in}
\begin{multicols}{2}
{\tiny
\hspace{-0.25in}$\widehat{\lambda_i}(k) = \widehat{\lambda_i}(\qr{k}_3)$ for $i = 1,2,3$ and $k \in \bbZ$
\vspace{6px}

\hspace{-0.25in}$\widehat{\lambda_4}(k) + \widehat{\lambda_5}(k) + \widehat{\lambda_6}(k) = \widehat{\lambda_4}(\qr{k}_3) + \widehat{\lambda_4}(\qr{k}_3) + \widehat{\lambda_6}(\qr{k}_3)$ for $k \in \bbZ$
\vspace{6px}
}
\end{multicols}

\vspace{-0.06in}
\noindent and one checks that $\lambda_1 + \lambda_2 + \lambda_3 = 3\bnu_3$ by (\ref{lattice measure}), which represents the entire discrete spectrum by Dekking's theorem, and one checks that $\lambda_4 + \lambda_5 + \lambda_6$ is equal to $\bnu_3 \ast \lambda$ for some singular continuous measure $\lambda$ on $\bbT$, and $\sigmaX \sim \bomega_4 \ast \bnu_3 \ast ( \bdelta + \lambda )$.  
\end{exmp}


\section{Appendix}\label{appendix}

The purpose of the appendix is to prove theorem \ref{queffelec two}, based on Queff\'elec's work for one-sided substitutions of constant length in one-dimension.  First, however, we need to extend  [\ref{queffelec}, Corollary 7.1] of Queff\'elec allowing for the diagonalization of an ergodic matrix of measures on the unit circle group.  Our approach, though similar, is more direct and allows for the diagonalization of a matrix of measures ergodic for a continuous action on a compact metric space.   Both rely on a local representation of the dual space to measures due to \v{S}reider.

\subsection{Generalized Functionals on $\mcM(X)$}\label{measures appendix}

Let $X$ be a metrizable compact space.  By a \textit{measure (on $X$)} we mean a complex Borel measure of finite total variation, and denote by $\mcM := \mcM(X)$ the Banach space of measures on $X$ under the total variation norm $\| \mu \| := |\mu|(X)$, where $|\mu| = \mu^+ + \mu^-$ is the total variation measure of $\mu \in \mcM$.  Let $\mcM^\ast$ denote the Banach space dual of $\mcM$, consisting of continuous linear functionals $\mcM \to \bbC.$  We extend the notions of absolute continuity, mutual singularity, equivalence of measures, almost-everywhere, and null-sets to complex measures using their total variation measures.  We will often view $\mcM$ via Riesz representation as continuous linear functionals on $C(X)$, expressing a measure as $\mu$ or $d\mu$ depending on the context.

For $\mu \in \mcM$, let $\mcL(\mu) := \{ \nu \in \mcM : \nu \ll \mu \}$ be the \textit{$\mcL$-space for $\mu,$} consisting of all measures absolutely continuous with respect to $\mu$, and $\mcL(\mu)^\perp$ those measures which are mutually singular to $\mu.$  Lebesgue decomposition gives $\mcM = \mcL(\mu) \oplus \mcL(\mu)^\perp,$ and implies that $\mcL(\mu)$ is a closed subspace.
 
\begin{thm}\label{RN}
	For each $\mu \in \mcM,$ there exist isometric isomorphisms $\partial_\mu, \partial_\mu^\ast$, such that for $\nu \ll \mu,$
	{\small
	\begin{align*}
		\textstyle \partial_\mu : \mcL(\mu) \xrightarrow{\sim} L^1(\mu) && \text{ with } &&  d\nu = \partial_\mu \nu \bump d\mu \hspace{12px} && \text{ and } && \partial_\mu \vert_{\mcL(\nu)} = \partial_\mu \nu \cdot \partial_\nu
		\\
		\textstyle \partial_\mu^\ast : \mcL(\mu)^\ast \xrightarrow{\sim} L^\infty(\mu) && \text{ with } && F(\nu) = {\textstyle \int_X \partial_\mu^\ast F d\nu}  &&\text{ and } && \partial_\nu^\ast \vert_{\mcL(\mu)^\ast}= \partial_\mu^\ast \hspace{8px}
	\end{align*}}
\end{thm}
\noindent The statements about $\partial_\mu$ follow immediately from the Radon-Nikodym theorem for complex measures and those about $\partial_\mu^\ast$ follow as $L^1(\mu)^\ast$ is isometrically isomorphic to $L^\infty(\mu)$ when $\mu$ has finite total variation.   Here, the identity $d\nu = \partial_\mu \nu \bump d\mu$ holds in the sense of the Riesz Representation theorem.  Using density arguments, this extends to integration of $L^1(\nu)$ functions, so that $\nu(A) = \int_A \partial_\mu \nu d\mu$ for $A$ Borel, and they agree pointwise as measures.  In this sense, we can think of $\mcL(\mu)$ as $L^1(\mu)\bump d\mu,$ and one can use this to show that multiplication by $\partial_\mu \nu$ is a map from $\mcL(\nu) \to \mcL(\mu).$  We now use the maps $\partial_\mu^\ast$ to describe a similar localization of $\mcM^\ast$ due to \v{S}reider [\ref{sreider}], giving rise to an action on $\mcM$ used by Queff\'elec.
\begin{defn}\label{generalized functional}
A \textit{generalized functional} $\varphi$ on $\mcM$ is an association of $\mcM$ with essentially bounded functions on $X$ sending $\mu \in \mcM \mapsto \varphi_\mu \in L^\infty(\mu)$, such that for all $\mu, \nu \in \mcM$
$$\textstyle \nu \ll \mu \tbump{2px}{ $\implies$ } \varphi_\mu = \varphi_\nu \hspace{3px} \nu\text{-ae} \tbump{24px}{ and } \|\varphi\| := \sup_{\|\mu\| = 1} \|\varphi_\mu\|_{L^\infty(\mu)} < \infty$$
\end{defn}	
\noindent By the above theorem, it is clear that $\mu \mapsto \partial_\mu^\ast F$ takes every $F \in \mcM^\ast$ to a generalized functional.  Moreover, each generalized functional $\varphi$ determines a functional $F$ on $\mcM$ given by $F(\nu) := \int_X \varphi_\nu d\nu$ with $\|F\| = \|\varphi\| ,$ which is the content of the following result of \v{S}reider, see [\ref{sreider}, Thm 1].
\begin{thm}[\v{S}reider]\label{sreider thm}
	The Banach dual $\mcM^\ast$ coincides with generalized functionals on $\mcM.$
\end{thm}
As our measures have finite total variation, $L^\infty(\mu) \subset L^1(\mu)$ for all $\mu \in \mcM$.  Thus, we can compose $\partial_\mu^{-1} \circ \partial_\mu^\ast : \mcL(\mu)^\ast \hookrightarrow \mcL(\mu),$ sending $F \mapsto \partial_\mu^\ast F d\mu$, giving rise to an action of $\mcM^\ast$ on $\mcM$
$$\mcM^\ast \times \mcM \to \mcM \tbump{22px}{ sending } F,\mu \longmapsto F \adot \mu \tbump{22px}{ with } d(F \adot \mu) = \partial_\mu^\ast F d\mu$$
\noindent A map $\psi: \mcM \to \mcM$ is \textit{absolutely continuous} if $\psi : \mcL(\mu) \to \mcL(\mu)$ for all $\mu \in \mcM.$
\begin{prop}\label{properties of action}
	The action of $\mcM^\ast$ on $\mcM$ is by absolutely continuous commuting operators.
\end{prop}
\begin{proof}
	That the action is absolutely continuous follows as $\partial_\mu^{-1} \circ \partial_\mu^\ast$ is a map into $\mcL(\mu)$ and so $\nu \ll \mu$ implies $F\adot \nu \ll \nu \ll \mu.$  For $F, G \in \mcM^\ast,$ theorem \ref{RN} tells us that $\partial_{G \adot \mu}^\ast F = \partial_\mu^\ast F \in L^1(G\adot \mu)$ so
$$d(F \adot G \adot \mu) = \partial_{G \adot \mu}^\ast F d (G \adot \mu) = \partial_\mu^\ast F d(G \adot \mu) = \partial_\mu^\ast F \partial_\mu^\ast G d \mu = \partial_{F \adot \mu}^\ast G d(F \adot \mu) = d(G \adot F \adot \mu)$$
\noindent which shows that the action is commutative.
\end{proof}
We say $F \in \mcM^\ast$ \textit{acts invariantly} on $\mu \in \mcM$ if there exists $F_\mu \in \bbC$ with $F \adot \mu = F_\mu \mu,$ a constant multiple of $\mu.$   We refer to $F_\mu$ as the \textit{eigenvalue for $F$ on $\mu.$}  If $\nu \in \mcL(\mu)$ then $F_\mu = F_\nu$ $\nu$-ae by theorem \ref{RN} so that if $F$ acts invariantly on $\mu,$ it acts invariantly on $\mcL(\mu)$ with the same eigenvalue.
\begin{prop}
	If $F \in \mcM^\ast$ acts invariantly on $\mu, \nu \in \mcM^\ast$ and $F_\mu \neq F_\nu,$ then $\mu \perp \nu.$
\end{prop}
\begin{proof}
	If $\rho \in \mcL(\mu) \cap \mcL(\nu)$ is not $\bZero$, then $F$ will also act invariantly on $\rho$ and the eigenvalues will satisfy $F_\mu=F_\rho=F_\nu,$ giving a contradiction.  This implies the measures are mutually singular.
\end{proof}
We now discuss some preliminaries allowing us to extend the action of $\mcM^\ast$ to matrices with entries in $\mcM.$  For $n \geq 1,$ let $\bM_n(\bbC)$ denote the space of $n \times n$ complex matrices and $\mcM_n :=\mcM_n(X)$ the collection of $\bM_n(\bbC)$-valued Borel measures of finite total variation.   We abuse notation by using the symbol $\bZero$ to denote both the zero matrix and zero matrix of measures, for any dimension.  For $\mcW = (\omega_{ij})_{1 \leq i,j \leq n} \in \mcM_n(X)$ we write $|\mcW| := \sum_{i,j} |\omega_{ij}|$ for the \textit{total variation measure of $\mcW$}, so that $A$ is $|\mcW|$-null if and only if $\mcW(B) = \bZero$ for every $B \subset A.$  For readability, we will often denote $\bomega := |\mcW|.$ The following shows that the total variation's type is a similarity invariant (over $\bbC.$)
\begin{prop}\label{nullity invariance}
	If $S \in \bM_n(\bbC)$ is invertible and $\mcW \in \mcM_n(X),$ then $|\mcW| \sim |S \mcW S^{-1}|.$
\end{prop}
\begin{proof}
	If $A$ is $|\mcW|$-null then $A$ is $\omega_{ij}$-null for $1 \leq i,j \leq n,$ so that $\mcW(B) = \bZero$ for every $B \subset A$ and $S \mcW(B) S^{-1} = \bZero$ for all $B \subset A.$  Thus $|S \mcW S^{-1}| \ll |\mcW|$ and conversely as $\mcW = S^{-1} (S \mcW S^{-1}) S.$ 
\end{proof}
Integration with respect to $\mcW$ is done as follows: for $f \in L^1(\bomega)$ the integral $\int_X f d\mcW \in \bM_n(\bbC)$ is the matrix with the integrals $\int_X f d\omega_{ij}$ as its components.  All of the above notions can be extended to rectangular matrices as well as vectors, although the square case is of primary interest to us and we will not need the others beyond formalities.    A matrix of measures is \textit{positive definite} if it is \textit{pointwise positive semidefinite} when evaluated on sets. 

Extend the action of $\mcM^\ast$ to $\mcM_n$ componentwise (as $\partial_{\bomega}^\ast F$ is a scalar function for $F \in \mcM^\ast$)
$$\mcM^\ast \times \mcM_n \longrightarrow \mcM_n \tbump{12px}{sending} F, \mcW \longmapsto F \adot \mcW  \tbump{12px}{with} d(F \adot \mcW) = \partial_{\bomega}^\ast F d\mcW$$
\noindent We say that $F$ \textit{acts invariantly} on $\mcW \in \mcM_n$ if there exists a $F_\mcW \in \bM_n(\bbC)$ such that $F \adot \mcW = \mcW F_\mcW$ (or $= F_\mcW \mcW$) and refer to $F_\mcW$ as a right (or left) \textit{eigenmatrix} for $F$ on $\mcW.$  As $(F \adot \mcW)^\ast = \overline{F} \adot \mcW^\ast$	
$$F \adot \mcW = F_\mcW \mcW \tbump{12px}{ $\iff$ } \overline{F} \adot \mcW^\ast = \mcW^\ast F_\mcW^\ast$$
\noindent we may restrict our attention to pairs $F,\mcW$ for which the eigenmatrices act from the right without loss of generality.  Although in general eigenmatrices for a given functional $F$ need not be unique, they must be whenever $\mcW(A)$ is invertible for some $A \subset X,$ as
$$\textstyle F \adot \mcW = \mcW F_\mcW \tbump{12px}{ $\implies$ } F_\mcW = \mcW(A)^{-1} \int_A \partial_{\bomega}^\ast F d\mcW$$
\noindent Let $\mcM_n^\circ$ denote the collection of \textit{nondegenerate matrix measures}, or
$$\mcM_n^\circ := \{ \mcW \in \mcM_n : \text{ there exists } A \subset X \text{ Borel with } \mcW(A) \text{ invertible} \}$$
\noindent Thus when $F$ acts invariantly on $\mcW \in \mcM_n^\circ$ we can speak of \textit{the eigenmatrix} $F_\mcW$ for the action of $F$ on $\mcW.$  For the action of $\mcM^\ast$ on a (general) matrix measure $\mcW,$ the (right) eigenmatrices are
$$\text{Eig}_{\text{R}}(\mcW) := \{ \bB \in \bM_n(\bbC) : F \adot \mcW = \mcW \bB \text{ for some } F \in \mcM^\ast \}$$
\begin{prop}\label{simultaneous diagonalization}
	For $\mcV \in \mcM_n^\circ,$ the matrices in $\text{Eig}_{\text{R}}(\mcV)$ are simultaneously diagonalizable.
\end{prop}
\begin{proof}
	If $F, G \in \mcM^\ast$ act invariantly on $\mcV \in \mcM_n^\circ$ with $\bnu := |\mcV|$, then as the action of $\mcM^\ast$ on $\mcM_n$ is componentwise, we have $F \adot G \adot \mcV = G \adot F \adot \mcV$ by proposition \ref{properties of action} and so
	$$\mcV F_\mcV G_\mcV = G \adot F \adot \mcV = F \adot G \adot \mcV =  \mcV G_\mcV F_\mcV$$
	\noindent so that, as $\mcV(A)$ is invertible for some $A$ Borel, we can cancel and the eigenmatrices commute.  Now, suppose $F \in \mcM^\ast$ acts invariantly on $\mcV \in \mcM_n^\circ$ we show that $F_\mcV$ is diagonalizable.  Let $\lambda \in \bbC$ be an eigenvalue of $F_\mcV$ and $\bv \in \bbC^n$ a generalized eigenvector of degree $k > 0,$ so that $(F_\mcV - \lambda I)^k \bv = 0.$ The claim will follow by showing that $k = 1,$ as this implies the eigenvalues of $F_\mcV$ are simple, and thus diagonalized by putting into Jordan form.  As $F \adot \mcV = \mcV F_\mcV,$ we have
	$$\mcV (F_\mcV - \lambda I)^k \bv = (F - \lambda)^k \adot \mcV \bv = \bZero$$
	\noindent where $(F-\lambda)^k$ is the action of the functional $\mu \mapsto F(\mu) - \lambda\mu(X)$ on $\mcV,$ $k$ consecutive times; as $\partial_{\bnu}^\ast F - \lambda \in L^\infty(\bnu),$ this defines a generalized functional on $\mcL(\bnu).$  The above then gives us
	$$( \partial_{\bnu}^\ast F - \lambda)^k = 0 \hspace{8px} \mcV\bv\text{-ae} \tbump{8px}{ $\implies$ } \partial_{\bnu}^\ast F = \lambda \hspace{8px} \mcV \bv\text{-ae}$$
	\noindent  In particular, this shows $\mcV F_\mcV \bv = F \adot \mcV \bv = \lambda \mcV \bv$ and so $F_\mcV \bv = \lambda \bv$ as $\mcV \in \mcM_n^\circ.$ Thus, $\bv$ is an eigenvector for $F_\mcV$ which is diagonalizable as it has no generalized eigenvectors.

	Thus $\text{Eig}_{\text{R}}(\mcV)$ is a commuting family of diagonalizable matrices, and therefore can be simultaneously diagonalized, as commuting matrices share invariant subspaces and diagonalizable matrices have one-dimensional invariant subspaces spanning their domain.
\end{proof}


\subsubsection{Diagonalization of $\mcM_n(X,S)$}\label{matrix diagonalization}

For a continuous $\Zd$ action $S$ on $X,$ let $\mcM(X,S)$ denote the collection of $S$-invariant complex Borel measures on $X$; note the distinction from the prequel, where it represents only the positive probability measures.  Restricting to $S$-invariant probability measures gives a compact convex set, the extreme points of which are ergodic and mutually singular.  By ergodic decomposition, $\mcM(X,S)$ is the $\bbC$-span of $\mcE = \mcE(X,S),$ the ergodic $S$-invariant probability measures, see [\ref{walters}].

For each $\mu \in \mcM(X,S)$ we have the $\sigma$-algebra (and where $\Delta$ denotes the symmetric difference)
$$\mcB_\mu = \{ A \subset X \text{ Borel such that } A \Delta S^{\bk}A \text{ is } \mu\text{-null } \forall \, \bk \in \Zd \}$$
\noindent consisting of the \textit{$\mu$-ae $S$-invariant Borel subsets} of $X.$  The simple functions over this $\sigma$-algebra generate a closed subspace of $L^2(\mu)$ which we denote $L^2(\mu,S)$ and the orthogonal projection $\bbE_\mu : L^2(\mu) \to L^2(\mu,S)$ sends $f$ to $\bbE_\mu(f) = \bbE(f \vert \mcB_\mu),$ the conditional expectation of $f$ given the $\sigma$-algebra of $\mu$-ae $S$-invariant Borel sets.  Thus, by the mean ergodic theorem, the ergodic averages
$$\textstyle{A_m(f) := \frac{1}{m^d} \sum_{\bk \in [\bZero, m\bOne]} f \circ S^{-\bk}} \tbump{6px}{$\xrightarrow{ \text{ as } n \to \infty}$} \bbE_\mu(f) \tbump{6px}{$\in$}  L^2(\mu)$$

Using this, we construct a collection of functionals on $\mcM$ determined by a bounded Borel function on $X.$  For such a function, $f,$ the ergodic averages $\mcA_m(f)$ are all bounded by the supremum norm of $f,$ so that $\|\bbE_\mu(f)\|_{\infty} \leq \|f\|_{\infty}$ and $\bbE_\mu(f) \in L^\infty(\mu).$  As all our measures have finite total variation, $L^2$ convergence implies convergence in $L^1,$ and so subsequences of ergodic averages converge pointwise almost everywhere with respect to an invariant measure.   Thus, if $\nu \ll \mu$ are both invariant, we can pass to the almost everywhere pointwise subsequential limits of $\mcA_m(f)$ and show that $\bbE_\mu(f) = \bbE_\nu(f)$ $\nu$-ae.  For each bounded $f,$ \v{S}reider's theorem \ref{sreider thm} shows that $\mu \mapsto \bbE_\mu(f)$ gives rise to a functional on $\mcM(X,S),$  a closed subspace of $\mcM,$ which can be extended to $\mcM$ via the Hahn-Banach theorem; let $[f]$ denote the collection of all such extensions.  Thus, we obtain:
\begin{prop}\label{induced functional}
	For each bounded function $f$ on $X$, there exists a nonempty collection of functionals $[f] \subset \mcM^\ast$ such that $F \in [f]$ and $\mu \in \mcM(X,S)$ give $\partial_\mu^\ast F = \bbE_\mu(f)$ $ \mu\text{-ae}$
\end{prop}
Note that for every $F \in [f]$ and $\mu \in \mcM(X,S)$ we have $F(\mu) = \int_X \bbE_\mu(f) d\mu = \int_X f d\mu$ as $X$ is always an invariant set, and $\bbE_\mu(f)$ is the conditional expectation.  We use the above to prove a diagonalization result for ergodic matrices of measures relative to $S: X \to X$ continuous.  Let $\mcM_n(X,S)$ denote those measures in $\mcM_n$ all of whose components are in $\mcM(X,S).$  A matrix of measures $\mcW \in \mcM_n(X,S)$ with total variation $\bomega$ is ergodic provided for all $f, g \in L^2(\bomega)$	
\begin{equation}\label{ergodic MoM} \textstyle{\frac{1}{m^d} \sum_{\bk \in [\bZero, m\bOne]} \int_X f \circ S^{-\bk} \cdot \, g \, d\mcW \longrightarrow \int_X f \, d\mcW \, \int_X g \, d \mcW}\end{equation}
\noindent as $m \to \infty.$ Characterizations for mixing are similar, and just as in the $n=1$ case, mixing implies ergodic by looking at component measures, and so the familiar Fourier characterizations for mixing hold in this setting as well.  Note that in the above, we do not include the case where the limits converge to $\left( \int g d\mcW\right) \left( \int f d\mcW \right)$ as possibilities, as the following theorem shows that these matrices will always commute as long as the matrix of measures is ergodic.

We say a matrix of measures $\mcW$ is \textit{diagonalizable} if there exists a $\bQ \in \bM_n(\bbC)$ invertible with $\bQ \mcW \bQ^{-1}$ a diagonal matrix of measures; the measures appearing on the diagonal are called \textit{eigenmeasures} of $\mcW.$  Recall that the nondegenerate matrices of measures are the collection $\mcM_n^\circ$ of matrices of measures which are invertible when evaluated on some Borel subset of $X$.
\begin{thm}\label{ergodic diagonalization}
If $\mcW \in \mcM_n(X,S) \cap \mcM_n^\circ$ is ergodic, it is diagonalizable with ergodic eigenmeasures.
\end{thm}
\begin{proof}
	The reverse direction is immediate, the forward follows from propositions \ref{simultaneous diagonalization} and \ref{induced functional}: for $A \subset X$ Borel and $F \in [\IND_A],$ consider the measure $F \cdot \mcW \in \mcM_n(X)$.  As $\mcW$ is ergodic, for $g \in C(X)$
	$${\small  \int_X g \, d(F \cdot \mcW) = \int_X \partial_{\bomega}^\ast F \, g \, d\mcW = \int_X \bbE_{\bomega}(\IND_A) \, g \, d\mcW = \lim_{m \to \infty} \int_X \mcA_m(\IND_A) \, g \, d\mcW =  \mcW(A) \int_X g \, d\mcW}$$
	\noindent by the Lebesgue dominated convergence theorem and the mean ergodic theorem.  As measures in $\mcM(X)$ are characterized by integration against continuous functions, this implies that $F \cdot \mcW = \mcW(A) \mcW,$ and so  $\mcW(A) \in \text{Eig}_{\text{R}}(\mcW)$ for all $A \subset X$ Borel.  By proposition \ref{simultaneous diagonalization}, all the matrices $\mcW(A)$ for $A$ Borel are simultaneously diagonalizable over $\bbC,$ and there exists a matrix $\bQ \in \bM_n(\bbC)$ such that $\bQ \mcW(A) \bQ^{-1}$ is diagonal, for all $A \subset X$ Borel.  This implies that $\bQ \mcW \bQ^{-1}$ is a diagonal matrix of measures, and so $\mcW$ is diagonalizable over $\bbC.$  The claim of ergodicity and invariance follows as
	$$\textstyle \bQ \left( \int f d\mcW \right) \bQ^{-1} = \int f \,d(\bQ \mcW \bQ^{-1})$$
	\noindent and using the descriptions of invariance and ergodicity given earlier.
\end{proof}


\subsection{Proof of Theorem \ref{queffelec two}}\label{proof of queffelec}

We now proceed to the proof of our main result, based on Queff\'elec's work.  We already know that $\sigmaX$ is related to linear combinations of correlation measures (theorem \ref{queffelec one}), and we know that the spectral hull gives rise to $\bq$-shift invariant measures (proposition \ref{K invariance}) in the span of the correlation measures.  To finish, we need to show that the spectral hull actually generates the entire spectrum, and the extreme points give rise to measures with the specified ergodic properties. This is accomplished via Queff\'elec's \textit{bicorrelation measure} $\mcZ$, a matrix of measures related to the correlation measures of the bisubstitution $\mcS \otimes \mcS.$    Recall that all our $\bq$-substitutions $\mcS$ on $\mcA$ have been telescoped so that $\mcS$ and $\mcS \otimes \mcS$ have index of imprimitivity $1$, see proposition \ref{PRF} so that
\begin{equation}\label{projection matrix}\textstyle \bP := \lim_{n \to \infty} \frac{1}{Q^{n}} M_{\mcS}^n \tbump{8px}{ and } \mcP := \lim_{n \to \infty} \frac{1}{Q^n} M_{\mcS \otimes \mcS}^n = \lim_{n \to \infty} \frac{1}{Q^n} \sum_{\bj \in [\bZero,\bq^n)} \mcR_\bj^{(n)} \otimes \mcR_\bj^{(n)} \end{equation}
\noindent exist by the primitive reduced form of proposition \ref{PRF} and results in [\ref{gantmacher}, Chap III \S 7], converging to nonorthogonal projections onto the $Q$-eigenspaces of the substitution and coincidence matrices.  
\begin{prop}\label{coefficients exist} 
	Let $\mcS$ be an aperiodic $\bq$-substitution on $\mcA.$  If, for $\bk \in \Zd$ and $n \geq 0$, we write 
	$$\bC_\bk^{(n)} := \textstyle \frac{1}{Q^n} \sum_{\bj \in [\bZero,\bq^n)} \mcR_\bj^{(n)} \otimes \mcR_{\bj+\bk}^{(n)} \in \bM_{\mcA^2}(\bbC) \tbump{12px}{ then } \bC_\bk := \lim_{n \to \infty} \bC_\bk^{(n)} \in \bM_{\mcA^2}(\bbC)$$
	\noindent exists for every $\bk \in \Zd,$ and form the Fourier coefficients for a matrix of measures $\mcZ \in \mcM_{\mcA^2}(\bbT^d).$
\end{prop}
\begin{proof}
	Note that $C_\bZero = \mcP$ by definition.  For $\omega \in \mcA^+$ and $n \geq 0,$ let $\bh_\omega^{(n)}$ be the vector in $\bbC^\mcA$ recording the frequencies with which $\omega$ appears in superblocks $\mcS^n(\gamma)$, so that for each $\gamma \in \mcA,$
$$\be_\gamma^t \bh_\omega^{(n)} := \scalebox{1.2}{$\nicefrac{1}{Q^n}$} \,\text{Card}\big\{ \bj \in \Zd \tbump{2px}{$:$} \omega \leq T^\bj \mcS^n(\gamma) \}$$
	\noindent Although limits of such frequencies are known to exist in general, we show it here for completeness.
\begin{lem}\label{block frequency}
	If $\mcS$ is an aperiodic $\bq$-substitution on $\mcA$, then $\lim_{n \to \infty} \bh_\omega^{(n)}$ exists for every $\omega \in \mcA^+$
\end{lem}
\begin{proof}
	See also [\ref{queffelec}, Proof of Prop 10.4].  If $\omega$ is a subblock of $\mcS^{n+p}(\gamma)$ for $n, p > 0,$ then either:  $\omega$ appears as a subblock of $\mcS^n(\alpha)$ for some $\alpha$ appearing in $\mcS^p(\gamma),$ or $\omega$ overlaps two or more such blocks.
	First, the number of ways $\alpha$ can appear in $\mcS^p(\gamma)$ is $\be_\alpha^t M_\mcS^p \be_\gamma.$  Second, if $\supp(\omega)$ is contained in $[\bj,\bj + \bk)$ for $\bj, \bk \in \Zd,$ then the number of ways $\omega$ can overlap multiple such blocks is bounded by  $Q^p \text{Card}(\Delta_p(\bk)),$ as $\bj$ will be translated into one of the $Q^p$ subblocks of $[\bZero,\bq^{n+p})$ in the $\bq^n \Zd$ lattice. Lastly, the number of ways $\omega$ appears in $\mcS^n(\alpha)$ is $\big(\bh_\omega^{(n)}\big)_\alpha$ so that adding over $\alpha$
	{\small $${\sum}_{\alpha \in \mcA} \big(\bh_\omega^{(n)}\big)_\alpha \big( M_\mcS^p \big)_{\alpha,\gamma} \leq \big(\bh_\omega^{(n+p)}\big)_\gamma\leq   Q^p \text{Card}(\Delta_n(\bk)) + {\sum}_{\alpha \in \mcA} \big(\bh_\omega^{(n)}\big)_\alpha \big( M_\mcS^p \big)_{\alpha,\gamma}$$}
	\noindent Dividing through by $Q^{n+p}$ and taking the transpose gives componentwise inequalities
	\begin{equation}\label{block frequency inequality 1}Q^{-p}  (M_\mcS^p)^t \bh_\omega^{(n)} \leq \bh_\omega^{(n+p)} \leq Q^{-n}\text{Card}(\Delta_n(\bk)) \bOne + Q^{-p} (M_\mcS^p)^t \bh_\omega^{(n)}\end{equation}
	\noindent where $\bOne$ is again the vector of $1$'s.  Letting $p \to \infty$, we have
	$$ \bP^t \bh_\omega^{(n)}  \leq {\liminf}_{p \to \infty} \bh_\omega^{(n+p)} \leq {\limsup}_{p \to \infty} \bh_\omega^{(n+p)} \leq Q^{-n}\text{Card}(\Delta_n(\bk))\bOne + \bP^t \bh_\omega^{(n)}$$
	\noindent The first inequality in (\ref{block frequency inequality 1}) together with the identity $M_\mcS \bP = Q \bP$ gives $\bP^t \bh_\omega^{(n)} \leq \bh_\omega^{(n+p)} \bP$ so that, as $\bZero \leq \bh_\omega^{(n)} \leq \bOne$ and $\bP$ is nonnegative, the limit $\bh_\omega^{\infty} := \lim_{n \to \infty} \bP^t \bh_\omega^{(n)}$ exists.  Letting $n \to \infty$ in the above inequalities and using lemma \ref{small carries} shows that the limit of $\bh_\omega^{(n)}$ exists, as $\bh_\omega^{\infty} \leq \liminf_{p \to \infty} \bh_\omega^{(p)} \leq \limsup_{p \to \infty} \bh_\omega^{(p)} \leq \bh_\omega^{\infty}$, completing the proof.
\end{proof}

	We now proceed to prove the above proposition, see also [\ref{queffelec}, Prop 10.4].  Recall from \S \ref{aperiodicity} that the $n$-th desubstitute of $\omega \in \mcA^+$ corresponds to the smallest (subset ordering on support) blocks $\eta \in \mcA^+$ extending $\omega$ after $n$-steps of the iterated function system $T^\bk\mcS$ for $\bk \in [\bZero,\bq)$, or 
	$$\mcS^{-n}(\omega) = \{ (\bj,\eta) \in [\bZero,\bq^n) \times \mcA^+ \tbump{0px}{$:$} \supp(\eta) = \qq{\bj + \supp(\omega)}_n \text{ and } \omega \leq T^\bj \mcS^n(\eta) \}$$
	\noindent By corollary \ref{topological structure}.3, the cylinders of $X_\mcS$ can be decomposed via the iterated function system
	$$[\omega] = \scalebox{1.3}{${\bigsqcup}$}_{(\bj,\eta) \in \mcS^{-n}(\omega)} T^\bj \mcS^n[\eta]$$
	\noindent so that $\bh_\omega^{(n)}$ is counting those pairs $(\bj,\eta)$ for which $\supp(\eta) = \bZero,$ i.e. when $\bj \notin \Delta_n(\supp(\omega)).$  In the case where $\omega$ sends $\bZero \mapsto \alpha$ and $\bk \mapsto \beta,$ this means $\bh_\omega^{(n)}$ is \textit{undercounting}, while the proof of theorem \ref{fourier recursion} shows $\be_{\alpha \beta}^t \sum_{\bj \in [\bZero,\bq^n)} \mcR_\bj^{(n)} \otimes \mcR_{\bj+\bk}^{(n)}$ is \textit{overcounting}, $\mcS^{-n}(\omega)$ each by a bounded factor of $\text{Card}(\Delta_n(\bk)).$  Thus $\big| \be_{\alpha \beta}^t C_\bk^{(n)} - (\bh_\omega^{(n)})^t \big| \to 0$ as $n \to \infty$ and $C_\bk$ exists for $\bk \in \Zd$.
	We now show $C_\bk$ are Fourier coefficients for a matrix of measures $\mcZ \in \mcM_{\mcA^2}(\bbT^d)$.  For $\bw \in \bbT^d$ and $\bk \in \Zd$ write
	$$\textstyle \mcZ_n(\bw) := \frac{1}{Q^n} \sum_{\bi, \bj \in [\bZero,\bq^n)} \mcR_\bj^{(n)} \otimes \mcR_\bi^{(n)} \bw^{\bi - \bj} \tbump{12px}{ where } \bw^\bk = (w_1^{k_1}, \ldots, w_d^{k_d})$$
	\noindent Viewing $\mcZ_n(\bw) d\bw \in \mcM_{\mcA^2}(\bbT^d)$, for each $\alpha \beta, \gamma \delta \in \mcA^2$ the total variation norm of $\be_{\alpha \beta}^t \mcZ_n \be_{\gamma \delta}$ is
	$$\small{ \left\| \be_{\alpha \beta}^t \mcZ_n \be_{\gamma \delta} \right\| = \int_{\bbT^d} \left| \be_{\alpha\beta}^t \mcZ_n(\bw) \be_{\gamma \delta} \right| d\bw  = \frac{1}{Q^n} \int_{\bbT^d} \Big| \sum_{\bi, \bj \in [\bZero,\bq^n)} ( \be_\alpha^t \mcR_\bj^{(n)} \be_\gamma \bw^{-\bj}) (\be_\beta^t \mcR_\bi^{(n)} \be_\delta \bw^\bi) \Big| d\bw}$$
	\noindent Using the Cauchy-Schwartz inequality, this gives (where $s = \text{Card} \mcA$)
	{\small \begin{align*}
		\left\| \be_{\alpha \beta}^t \mcZ_n \be_{\gamma \delta} \right\|  &\leq \frac{1}{Q^n} \Big( \int \big| {\sum}_{\bj \in [\bZero,\bq^n)} \be_\alpha^t \mcR_\bj^{(n)} \be_\gamma \bw^{-\bj} \big|^2 \, d\bw \Big)^{1/2} \Big( \int \big| {\sum}_{\bi \in [\bZero,\bq^n)} \be_\beta^t \mcR_\bi^{(n)} \be_\delta \bw^\bi \big|^2 \, d\bw\Big)^{1/2} \\
			& = \frac{1}{Q^n} \Big( \sum_{\bj \in [\bZero,\bq^n)} (\be_\alpha^t \mcR_\bj^{(n)} \be_\gamma)^2 \Big)^{1/2} \Big( \sum_{\bi \in [\bZero,\bq^n)} (\be_\beta^t \mcR_\bi^{(n)} \be_\delta)^2 \Big)^{1/2}  = \frac{1}{Q^n} (\be_\alpha^t \mcM_\mcS^n \be_\gamma)^{\frac{1}{2}} (\be_\beta^t \mcM_\mcS^n \be_\delta)^{\frac{1}{2}} \leq 1
	\end{align*}}
	\noindent as $\be_\iota^t \mcR_\bk^{(n)} \be_\kappa = 0$ or $1$, and $M_\mcS^n = \sum_{\bj\in [\bZero,\bq^n)} \mcR_\bj^{(n)} $. As $\| \mcZ_n \| \leq s^4,$ the measures $\mcZ_n$ are uniformly bounded and a subsequence converges in the weak-star topology to some measure $\mcZ.$  However, as
	$$\textstyle \widehat{\mcZ_n}(\bk) = \frac{1}{Q^n} \sum_{\bj \in [\bZero,\bq^n)} \mcR_\bj^{(n)} \otimes \mcR_{\bj+\bk}^{(n)}$$
	\noindent the Fourier coefficients all converge to the same limit, so $\mcZ$ is the unique limit point and
	$$\textstyle \widehat{\mcZ}(\bk) = \lim_{n \to \infty} \widehat{\mcZ_n}(\bk) \lim_{n \to \infty} C_\bk^{(n)} =  \lim_{n \to \infty} \frac{1}{Q^n} \sum_{\bj \in [\bZero,\bq^n)} \mcR_\bj^{(n)} \otimes \mcR_{\bj+\bk}^{(n)}$$
	\noindent completing the proof.
\end{proof}
The matrix $\mcZ = (\sigma_{\alpha \beta}^{\gamma \delta})_{\alpha \beta, \gamma \delta \in \mcA^2}$ is the \textit{bicorrelation matrix,} after Queff\'elec, as its components $\sigma_{\alpha \beta}^{\gamma \delta}$ are the correlation measures of the bisubstitution.  Note that in the definition of $\mcZ$, the limit above can be split at the $p$-th scale $\bq$-adicly letting us write
$${\textstyle\widehat{\mcZ}(\bk) = C_\bk = \left(\frac{1}{Q^p} {\sum}_{\qr{\bj}_p \in [\bZero,\bq^p)} \mcR_\bj^{(p)} \otimes \mcR_{\bj+\bk}^{(p)} \right) \lim_{n \to \infty} \left( \frac{1}{Q^{n-p}} {\sum}_{\qq{\bj}_p \in [\bZero,\bq^{n-p})} \mcR_{\qq{\bj}_p}^{(n-p)} \otimes \mcR_{\qq{\bj+\bk}_p}^{(n-p)} \right)} $$
\noindent which gives a Fourier recursion for $\widehat{\mcZ}(\bk)$ similar to theorem \ref{fourier recursion} for $\Sigma$ and letting us write
\begin{equation}\label{Z coefficients}\widehat{\mcZ}(\bk) = \frac{1}{Q^p} {\sum}_{\bj \in [\bZero,\bq^p)} \mcR_\bj^{(p)} \otimes \mcR_{\bj+\bk}^{(p)} \, \widehat{\mcZ}(\qq{\bj+\bk}_p) = \lim_{p \to \infty} \frac{1}{Q^p} {\sum}_{j \in [\bZero,\bq^p)} \mcR_\bj^{(p)} \otimes \mcR_{\bj+\bk}^{(p)} \widehat{\mcZ}(\bZero)\end{equation}
\noindent using lemma \ref{small carries} as in theorem \ref{fourier recursion}.  Thus $\widehat{\mcZ}(\bk) = \widehat{\mcZ}(\bk) \widehat{\mcZ}(\bZero)$ for all $\bk$.  As $\widehat{\mcZ}(\bZero) = \mcP$ by definition,
\begin{equation}\label{P invariance of Z}\mcZ = \mcZ \widehat{\mcZ}(\bZero) = \mcZ \mcP \end{equation}	
\begin{prop}\label{diagonalization of PZ}
	 $\mcP \mcZ$ is \scalebox{0.8}{$\bM_{\mcA^2}(\bbC)$}-diagonalizable with eigenmeasures ergodic for the $\bq$-shift.
\end{prop}
\begin{proof}
	First, we show that $\mcP\mcZ$ has the required ergodic properties. For $\ba \in \Zd$ use (\ref{Z coefficients}) to write	
\begin{equation}\label{PZ invariant}\widehat{\mcP \mcZ}(\ba \bq) = \frac{1}{Q}\mcP {\sum}_{\bj \in [\bZero,\bq)} \mcR_\bj \otimes \mcR_{\bj + \ba \bq} \mcZ(\qq{\bj+\ba\bq}_1) =  \frac{1}{Q}\mcP {\sum}_{\bj \in [\bZero,\bq)} \mcR_\bj \otimes \mcR_\bj \widehat{\mcZ}(\ba) =  \widehat{\mcP \mcZ}(\ba) \end{equation}
\noindent by the same operations used in \S \ref{spectral theory} to show (\ref{q invariance of sigma}) and the definition (\ref{projection matrix}) of $\mcP$, so that $\mcP \mcZ$ is $\bq$-shift invariant.  Fixing $\ba, \bb \in \Zd$ and writing $\bk \in [\bZero, \bq^{n+p})$ as $\bi + \bj \bq^p$ for $\bZero \leq \bi < \bq^p$ and $\bZero \leq \bj < \bq^n,$
{\small \begin{align*}
	\widehat{\mcZ}(\bb + \ba \bq^p) &= \lim_{n \to \infty} \frac{1}{Q^{n+p}} {\sum}_{\scalebox{0.7}{\small{$\bi \in [\bZero,\bq^p), \, \bj \in [\bZero,\bq^n)$}}} \mcR_{\bi + \bj \bq^p}^{(n+p)} \otimes \mcR_{\bi + \bj \bq^p + \bb + \ba \bq^p}^{(n+p)} \\
		&= \lim_{n \to \infty} \frac{1}{Q^{n+p}} {\sum}_{\scalebox{0.7}{\small{$\bi \in [\bZero,\bq^p), \, \bj \in [\bZero,\bq^n)$}}} \mcR_\bi^{(p)} \mcR_\bj^{(n)} \otimes \mcR_{\bi + \bb}^{(p)} \mcR_{\bj + \ba + \qq{\bi+\bb}_p}^{(n)}
\end{align*}}\noindent using the identities of proposition \ref{qsub}.  Using the mixed product property of the Kronecker product
$$\scalebox{0.9}{$\small{\lim_{n \to \infty}\frac{1}{Q^{n+p}} \sum_{\bi \in [\bZero,\bq^p)} \mcR_\bi^{(p)} \otimes \mcR_{\bi + \bb}^{(p)} \sum_{\bj \in [\bZero,\bq^n)} \mcR_\bj^{(n)} \otimes \mcR_{\bj + \ba + \qq{\bi+\bb}_p}^{(n)} =  \frac{1}{Q^p} \sum_{\bi \in [\bZero,\bq^p)} \mcR_\bi^{(p)} \otimes \mcR_{\bi + \bb}^{(p)} \widehat{\mcZ}(\ba + \qq{\bi+\bb}_p)}$}$$
As $\mcZ = \mcZ \mcP,$ we let $p \to \infty$ and use lemma \ref{small carries} again as in the proof of theorem \ref{fourier recursion} and obtain
\begin{equation}\label{PZ mixing}
	\textstyle \lim_{p \to \infty} \widehat{\mcP \mcZ}(\bb + \ba \bq^p) = \widehat{\mcP \mcZ}(\bb) \widehat{\mcP \mcZ}(\ba)
\end{equation}
\noindent so that $\mcP\mcZ$ defines a matrix of measures which is ergodic for the $\bq$-shift on $\bbT^d$, see (\ref{ergodic MoM}) in \S \ref{matrix diagonalization}.

To finish, we need to extend theorem \ref{ergodic diagonalization} to $\mcP\mcZ$ which allows for diagonalization of nondegenerate ergodic matrices of measures: $\mcP \mcZ$ is often degenerate, as $\mcP$ is a projection operator.  As $\mcP \mcZ = \mcP \mcZ \mcP$ by (\ref{P invariance of Z}),  $\mcP \mcZ$ is zero on the kernel of $\mcP,$ and will be diagonal on that subspace with respect to any basis for the kernel of $\mcP.$  Thus, we can restrict to the image of $\mcP,$ where $\mcP$ is the identity.  This implies that $\mcP \mcZ(\bbT^d) = \mcP \widehat{\mcZ}(\bZero) = \mcP^2 = \mcP$ is the identity on the image of $\mcP$, so that we may consider $\mcP \mcZ \in \mcM_n^\circ$ as an \textit{operator valued measure} on the image of $\mcP.$  As $\mcP\mcZ$ is ergodic for the $\bq$-shift, theorem \ref{ergodic diagonalization} tells us that $\mcP\mcZ$ is diagonalizable with respect to the image of $\mcP.$  As both the image and the kernel are invariant subspaces for $\mcP \mcZ$, it follows that $\mcP\mcZ$ is diagonalizable.  Ergodicity of the eigenmeasures then follows from the ergodic property of $\mcP\mcZ.$
\end{proof}
\begin{cor}
	Both $\mcZ$ and $\mcP$ preserve strong semipositivity on $\bbC^{\mcA^2}.$
\end{cor}
\begin{proof}
	As $\mcP = \widehat{\mcZ}(\bZero) = \mcZ(\bbT^d)$ and strong semipositivity of a vector-valued measure is determined by its value on sets, the statement for $\mcP$ is inherited from $\mcZ.$  Letting $\bv$ be strongly semipositive, we must show that $\mcZ \bv = \sum_{\gamma \delta \in \mcA^2} v_{\gamma \delta} \sigma_{\alpha \beta}^{\gamma \delta}$ gives rise to a positive definite matrix of measures, or
	$$\textstyle \bz^\ast \mathring{(\mcZ \bv)} \bz = \sum_{\alpha, \beta \in \mcA} \, \sum_{\gamma \delta \in \mcA^2} \, \overline{z_\alpha} \, z_\beta \, v_{\gamma \delta} \, \sigma_{\alpha \beta}^{\gamma \delta}$$
	\noindent	is a positive measure for nonzero $\bz \in \bbC^\mcA.$  As $\mathring{\bv} \stpos 0,$ there exists a orthonormal basis $\{\bw_\kappa\}_{\kappa \in \mcA}$ of eigenvectors with eigenvalues $\lambda_\kappa \geq 0$ (for $\kappa \in \mcA$) such that $\mathring{\bv} = \sum \lambda_\kappa \bw_\kappa \bw_\kappa^\ast.$ By the above and nonnegativity of the $\lambda_\kappa,$ it suffices to show for $\bw \in \bbC^\mcA$ the expression  $\sum_{\alpha,\beta,\gamma, \delta \in \mcA} \, \overline{z_\alpha} \, z_\beta \, \overline{w_\gamma}\, w_\delta \, \sigma_{\alpha \beta}^{\gamma \delta}$ is a positive measure.   Let $\tilde{\mcZ} = (\sigma_{\alpha \beta}^{\gamma \delta})_{\alpha \gamma, \beta \delta \in \mcA^2},$ noting the change in indexing from $\mcZ = (\sigma_{\alpha \beta}^{\gamma \delta})_{\alpha \beta, \gamma \delta \in \mcA^2}$.
\begin{lem}\label{Ztilde pos def}
	The matrix $\widetilde{\mcZ}$ is a positive definite matrix of measures.
\end{lem}
\begin{proof}
	Let $\{t_{\alpha \gamma}\}_{\alpha \gamma \in \mathcal{A}^2} \subset \bbC$ and denote $\nu := \sum_{\alpha \gamma, \beta \delta} t_{\alpha \gamma} \overline{t_{\beta \delta}} \sigma_{\alpha \beta}^{\gamma \delta}.$  We show that $\nu$ is a positive measure by showing that its Fourier coefficients form a positive definite sequence, and appealing to Bochner's theorem.  Fixing $n > 0$ and $\{ a_\bj \}_{\bj \in [\bZero,\bq^n)} \subset \bbC$, we must show that the measure
	$${\sum}_{\bj,\bk \in [\bZero,\bq^n)} a_\bk \, \overline{a_\bj} \, \widehat{\nu}(\bj - \bk) = {\sum}_{\alpha \gamma, \beta \delta} \, {\sum}_{\bj, \bk \in [\bZero,\bq^n)} a_\bk  \, \overline{a_\bj } \, t_{\alpha \gamma} \, \overline{t_{\beta \delta}} \, \widehat{\sigma_{\alpha \beta}^{\gamma \delta}}(\bj - \bk)$$
	\noindent is nonnegative.  For $\bj, \bk \in \Zd,$ one can use the Kronecker product to write
	{\small \begin{equation*}
		\widehat{\sigma_{\alpha \beta}^{\gamma\delta}}(\bj - \bk) = \lim_{n \to \infty} \frac{1}{Q^n} {\sum}_{\bi \in [\bZero,\bq^n)} \be_{\alpha \beta}^t \mcR_\bi^{(n)} \otimes \mcR_{\bi + \bj - \bk}^{(n)} \be_{\gamma \delta} = \lim_{n \to \infty} \frac{1}{Q^n} {\sum}_{\bi \in [\bZero,\bq^n)} (\be_\alpha^t \mcR_{\bi + \bk}^{(n)} \be_\gamma) (\be_\beta^t \mcR_{\bi + \bj}^{(n)} \be_\delta)
	\end{equation*}}
	\noindent as the definition of the $\mcR^{(n)}$ give invariance of the sum under translation in its index ($\bi \mapsto \bi + \bk$) and the defining property of the Kronecker product.  Thus, for any $n > 0$, if $\{a_\bj\}_{\bj \in [\bZero, \bq^n)} \subset \bbC$,
	{\small \begin{align*}
		\sum_{\bj,\bk \in [\bZero,\bq^n)} a_\bj \overline{a_\bk} \bump \widehat{\nu}(\bj - \bk) & = \lim_{n \to \infty} \frac{1}{Q^n} \sum_{\alpha \gamma, \beta \delta}\, \sum_{\bj, \bk \in [\bZero,\bq^n)} \, \sum_{\bi \in [\bZero,\bq^n)} \left( (a_\bk \, t_{\alpha \gamma} \, \be_\alpha^t \mcR_{\bi+\bk}^{(n)} \be_\gamma) (\overline{a_\bj \, t_{\beta \delta}} \, \be_\beta^t \mcR_{\bi + \bj}^{(n)} \be_\delta) \right) \\
		&= \lim_{n \to \infty} \frac{1}{Q^n} \Big\| {\sum}_{\bi, \bk \in [\bZero,\bq^n)} \, {\sum}_{\gamma \alpha \in \mcA^2} a_\bi \, t_{\gamma \alpha} \, \be_\alpha^t \mcR_{\bi+\bk}^{(n)}\be_\gamma \Big\|^2 \geq 0
	\end{align*}}
	\noindent and so $\{\widehat{\nu}(\bk)\}_{\bk \in \Zd}$ forms a positive definite $\Zd$-sequence.  By Bochner's theorem, $\nu$ is a positive measure on $\bbT^d,$ and it follows that $\widetilde{\mcZ}$ is a positive definite matrix of measures.
\end{proof}
\noindent As our goal was to show that $\sum_{\alpha,\beta,\gamma, \delta \in \mcA} \, \overline{z_\alpha} \, z_\beta \, \overline{w_\gamma}\, w_\delta \, \sigma_{\alpha \beta}^{\gamma \delta}$ is a positive measure, the result follows by setting $t_{\alpha \gamma} = z_\alpha w_\gamma$ and applying the above.
\end{proof}
\begin{cor}\label{positive diagonalization}
	 $\mcP \mcZ$ is diagonalizable with respect to strongly semipositive eigenvectors in $\bbC^{\mcA^2}$.
\end{cor}
\begin{proof}
	We know by proposition \ref{diagonalization of PZ} that $\mcP\mcZ$ can be diagonalized, all that remains is the choice of eigenvectors.  Note that \textit{a priori} the eigenmeasures are complex valued.  Being ergodic, however, they must be constant complex multiplies of positive $\bq$-shift invariant measures.  As $\widehat{\mcP\mcZ}(\bZero) = \mcP$ is the identity on the image of $\mcP,$ it follows that $\widehat{\lambda}(\bZero) = 1$ for every eigenmeasure $\lambda$ of $\mcP \mcZ.$  Thus the eigenmeasures are in fact probability measures, and therefore are either equal or mutually singular by ergodic decomposition.  Fix an eigenmeasure $\lambda$ of $\mcP \mcZ,$ and let $D_\lambda$ be the projection onto $\mcL(\lambda),$ the $\mcL$-space of measures absolutely continuous w.r.t. $\lambda$.    As $\mcP$ and $\mcZ$ preserve strong semipositivity, $D_\lambda(\mcP \mcZ)$ preserves strong semipositivity as well, as this property is determined pointwise on Borel sets.  Thus, the matrix of measures $D_\lambda(\mcP \mcZ)$ is similar (via the same similarity diagonalizing $\mcP \mcZ$) to a diagonal matrix of measures, with $\lambda$ or $\bZero$ on the diagonal, and so $D_\lambda(\mcP\mcZ) = \lambda P_\lambda$ for some projection operator $P_\lambda \in \bM_{\mcA^2}(\bbC).$  This implies that $P_\lambda$ preserves strong semipositivity, and as strongly semipositive vectors span $\bbC^{\mcA^2}$ (as $\bM_n(\bbC)$ is the $\bbC$-span of the (semi)positive definite matrices) it follows that the image of $P_\lambda$ is spanned by strongly semipositive vectors.  Thus the eigenspace corresponding to each eigenmeasure $\lambda$ is spanned by its strongly semipositive vectors, and we can therefore choose a basis of strongly semipositive eigenvectors for $\mcP\mcZ,$ as desired.
\end{proof}
\begin{prop}\label{Z lower triangular}
	There is a strongly semipositive basis in which $\mcZ$ is similar to
	\scalebox{0.9}{$\tiny{\begin{pmatrix} \Lambda & \bZero \\ \mcW & \bZero \end{pmatrix}}$} and
	where $\Lambda$ is a diagonal matrix of measures with $|\mcW| \ll \bomega_\bq \ast |\Lambda|$.
\end{prop}
\begin{proof}
	The specified basis is provided by corollary \ref{positive diagonalization}.  Let $\bw_1, \ldots, \bw_n$ be the strongly semipositive eigenvectors for nonzero eigenmeasures of $\mcP\mcZ$, and $\tilde{\bw}_{n+1}, \ldots, \tilde{\bw}_{\mcA^2}$ those corresponding to the zero eigenmeasures.   Then $\Lambda$ corresponds to $\mcZ$ on the span of the $\bw_j$ and $\mcW$ corresponds to $\mcZ$ on the span of the $\tilde{\bw}_j.$  That the last block column of $\mcZ$ is zero follows as $\mcZ = \mcZ \mcP.$  That $\Lambda$ is diagonal follows as $\mcP$ is the identity on the span of the $\bw_j,$ which diagonalize $\mcP \mcZ.$
	
	We now show that  $|\mcW| \ll \bomega_\bq \ast |\Lambda|$: let $\bL = \mcL(\bomega_\bq \ast |\Lambda|)$ be those measures absolutely continuous with respect to $ \sum \bomega_\bq \ast \lambda$, summed over the eigenmeasures $\lambda$ of $\mcP\mcZ$.  Recalling that $\bS_\bq$ is the $\bq$-shift on $\bbT^d$ sending $\bz \mapsto \bz^\bq$, we write $\bS_\bq \mu := \mu \circ \bS_\bq^{-1}$ for measures $\mu \in \mcM(\bbT^d).$  Note that the support of $|\Lambda| \ast \bomega_\bq$ is $\bS_\bq$-invariant, as the $\bq$-adic rationals and the support of $\Lambda$ are $\bS_\bq$-invariant.  Similarly, the null sets of $|\Lambda| \ast \bomega_\bq$ are also $\bS_\bq$-invariant, so that $\bL$ and $\bL^\perp$ are $\bS_\bq$-invariant $\mcL$-spaces of measures; see \S \ref{measures appendix} and also [\ref{queffelec}, Lemma 10.4]. Let $D$ and $D^\perp$ represent the projections onto the $\mcL$-spaces $\bL$ and $\bL^\perp$ respectively.  Writing $\bw := \sum \bw_i,$ then $\mcP \bw = \bw$ implies $C_\mcS \bw = Q \bw$ and
	$$\textstyle \mcZ \bw = \sum \lambda_i \bw_i + \mcW \bw \tbump{8px}{ $\implies$ } C_\mcS \mcZ \bw = Q \sum \lambda_i \bw_i + C_\mcS \mcW \bw \tbump{8px}{ and } \mcP \mcW \bw = \bZero$$
	\noindent Using identity (\ref{Z coefficients}) and $\widehat{\bS_\bq\mu}(\ba) = \widehat{\mu}(\ba \bq),$ one checks that $\bS_\bq \mcZ = \frac{1}{Q}C_\mcS \mcZ,$ and so
	$$\textstyle \bS_\bq \mcZ \bw = \sum \lambda_i \bw_i + \frac{1}{Q}\bS_\bq \mcW \bw \tbump{6px}{ $\implies$ } \frac{1}{Q}C_\mcS \mcW \bw = \bS_\bq \mcW \bw \tbump{6px}{ $\implies$ } \frac{1}{Q^n} C_\mcS^n D^\perp \mcW \bw = \bS_\bq^n D^\perp \mcW \bw$$
	\noindent as $\bL$ and $\bL^\perp$ are $\bS_\bq$-invariant.  As $\mcP \mcW \bw = \bZero,$ this implies that $\bS_\bq^n(D^\perp \mcW \bw) \to \bZero$ in norm.   Thus
	$$\widehat{\bS_\bq^n D^\perp \mcW \bw}(\bZero) = \widehat{D^\perp \mcW \bw}(\bZero) = \bZero$$
	\noindent As strong-positivity is determined pointwise and the $\lambda_i \in \bL,$ $D^\perp(\mcZ\bw) = D^\perp \mcW \bw$ is strongly semipositive, and thus $\bZero.$  It follows that $\mcW \bw = D(\mcW \bw).$  As the eigenvectors of the nonzero eigenmeasures of $\mcP \mcZ$ span the image of $\mcP,$ it follows that $|\mcW| \ll |\Lambda \ast \bomega_\bq|,$ completing the proof.
\end{proof}

The following lets us compute the eigenmeasures of $\mcP\mcZ$, see \S \ref{spectral hull} for a description of $\mcK^\ast.$
\begin{prop}\label{spectrum of PZ is given by K}
	The map $\bv \mapsto \lambda_\bv = \bv^t \Sigma$ takes $\mcK^\ast$ onto the eigenmeasures of $\mcP \mcZ.$
\end{prop}
\begin{proof}
	Let $\lambda$ be an eigenmeasure of $\mcP \mcZ$ and $\bv \stpos 0$ with $\bv^t \mcP \mcZ = \lambda \bv^t.$  As $\mcZ = \mcZ \mcP,$
	$$\lambda \bv^t = \bv^t \mcP \mcZ = \bv^t \mcP \mcZ \mcP = \lambda \bv^t \mcP \tbump{8px}{ $\implies$ } \bv^t \mcP = \bv^t$$
	\noindent so that $\bv$ is a left $Q$-eigenvector of $C_\mcS$ and $v_{\alpha \alpha}$ is constant for $\alpha \in \mcA$ by stochasticity of $C_\mcS$ and primitivity on ergodic classes.   As $\bv$ is strongly semipositive, this implies $\bv_{\alpha \alpha}\neq 0$ and thus $\bv \in \mcK$.  Now, if $\bu$ is the vector of initial weights $u_{\gamma} = \mu[\gamma]$,  then
	$$\lambda_{\bv} = \bv^t \Sigma = \bv^t \mcP \Sigma = \bv^t \mcP \mcZ \widehat{\Sigma}(\bZero) = \lambda \bv^t \widehat{\Sigma}(\bZero) = \lambda {\sum}_{\gamma \in \mcA} v_{\gamma \gamma} u_\gamma = \lambda v_{\alpha \alpha} {\sum}_{\gamma \in \mcA} u_\gamma = v_{\alpha \alpha} \lambda$$
	\noindent so that $\lambda = \lambda_\bw$ for some $\bw \in \mcK.$  As $\mcP\mcZ$ is strongly semipositively diagonalizable, we can write $\mcP \mcZ = \sum \lambda_j \mcP_j$ where $\mcP_j$ preserves strong semipositivity as in the proof of corollary \ref{positive diagonalization}.  If $\bv \in \mcK$
	$$\bv^t \Sigma = \bv^t \mcP \mcZ \widehat{\Sigma}(\bZero) = \sum \lambda_j \bv^t \mcP_j \widehat{\Sigma}(\bZero) = \sum c_j \lambda_j$$
	\noindent with $c_j \geq 0$ by lemma \ref{schur inner product}.  Thus the map $\bv \mapsto \lambda_\bv$ takes $\mcK$ onto the positive span of eigenmeasures of $\mcP\mcZ.$  As affine maps preserve convexity, this takes $\mcK^\ast$ onto the eigenmeasures of $\mcP\mcZ.$
\end{proof}
\noindent As theorem \ref{fourier recursion} gives $\Sigma = \mcZ \widehat{\Sigma}(\bZero),$ and theorem \ref{queffelec one} relates $\Sigma$ and $\sigmaX,$ the following is no surprise.  
\begin{cor}
	\label{trace Z generates sigmaX}
	The maximal spectral type of $(X_\mcS,\mu)$ is equivalent to $\Pi(|\mcZ|)$, with $\Pi$ as in (\ref{invariant contracted}). 
\end{cor}
\begin{proof}
	First, we note that for each $\alpha, \beta,\gamma \in \mathcal{A}$, $\sigma_{\alpha \beta}^{\gamma \gamma} = \sigma_{\alpha \beta}$, which can be seen immediately by comparing the respective Fourier coefficients and using the identities $\mcZ = \mcZ \mcP$ and $\mcP \be_{\gamma \gamma} = \widehat{\Sigma}(\bZero)$ for $\gamma \in \mcA.$  For $x, y \in \bbC$ not both $0$, let $\bv = \bv_{\alpha \beta}^{\gamma \delta} \in \bbC^{s^2}$ be the vector with components	
$${\small v_{\alpha' \beta'} = \scalebox{0.7}{$\begin{cases} \vspace{20px}\end{cases} \hspace{-9px} \begin{matrix} x & \text{ if  } \alpha' \beta' = \gamma \alpha \\ y & \text{ if  } \alpha' \beta' = \delta \beta \\ 0 & otherwise. \end{matrix}$}}$$
\noindent Then, as $\widetilde{\mcZ}$ is positive definite, we have for all measurable $A \subset \bbT^d$,
\begin{align*}
		\scalebox{0.8}{$\bv^\ast \widetilde{\mcZ}(A)$} \bv & \scalebox{0.8}{$= |x|^2 \sigma_{\alpha \alpha}^{\gamma \gamma}(A) + |y|^2 \sigma_{\beta \beta}^{\delta \delta}(A) + x \overline{y} \sigma_{\beta \alpha}^{\delta \gamma}(A) + y \overline{x} \sigma_{\alpha \beta}^{\gamma \delta}(A) \geq 0$}\\
			&= \scalebox{0.85}{$|x|^2 \sigma_{\alpha\alpha}(A) + |y|^2 \sigma_{\beta\beta}(A) + x \overline{y} \sigma_{\beta \alpha}^{\delta \gamma}(A) + y \overline{x} \sigma_{\alpha \beta}^{\gamma \delta}(A) \geq 0$}
\end{align*}
\noindent Let $A$ be such that $\sigma_{\alpha\alpha} (A) = 0,$ and let $y = -1$ in the above, so we obtain:
$$\sigma_{\beta\beta}(A) - x \sigma_{\beta \alpha}^{\delta \gamma} (A) - \overline{x} \sigma_{\alpha \beta}^{\gamma \delta}(A) \geq 0$$	
\noindent Letting $x \to \pm \infty$ along the real and imaginary axes we obtain $\sigma_{\beta\alpha}^{\delta \gamma}(A) = \sigma_{\alpha \beta}^{\gamma \delta}(A) = 0,$ so $\sigma_{\alpha \beta}^{\gamma \delta} \ll \sigma_{\alpha \alpha}$.  Similarly, fixing $x = -1,$ and letting $y \to \pm \infty$ along each axis shows the same for $\sigma_{\beta \beta} .$  Thus, 
$$\textstyle |\mcZ| = \sum |\sigma_{\alpha \beta}^{\gamma \delta}| \sim \sum_{\alpha = \beta} |\sigma_{\alpha \beta}^{\gamma \delta}| + \sum_{\alpha \neq \beta} |\sigma_{\alpha \beta}^{\gamma \delta}| \sim \sum |\sigma_{\alpha \alpha}| = \sum \sigma_{\alpha \alpha}$$
\noindent as the $\sigma_{\alpha \alpha}$ are positive measures, and the result follows as a corollary of theorem \ref{queffelec one}.
\end{proof}

With the above results established, we can prove Queff\'elec's theorem \ref{queffelec two}.  
\begin{theorem*}
	If $\mcS$ is an aperiodic $\bq$-substitution $\mcS$ and $\lambda_\bv := \bv^t \Sigma$, then the spectrum of $\mcS$ is
	$$\textstyle{\sigmaX \sim  \bomega_\bq \ast \sum_{\bw \in \mcK^\ast} \lambda_\bw} $$
	\noindent and the measures $\lambda_\bw$ for $\bw \in \mcK^\ast$ are ergodic $\bq$-shift invariant probability measures on $\bbT^d$.
\end{theorem*}
\begin{proof}
	Let $S \in \bM_{\mcA^2}(\bbC)$ be the similarity matrix of proposition \ref{Z lower triangular}, and $\lambda_i$ denote eigenmeasures of $\mcP\mcZ.$   Combining propositions \ref{nullity invariance} and \ref{Z lower triangular} with corollary \ref{trace Z generates sigmaX} and lemma \ref{invariant contracted map} gives equivalences
	$$\textstyle \sigmaX \sim \Pi(|\mcZ|) \sim \Pi( |S \mcZ S^{-1}| ) =  \Pi(|\mcP \mcZ|) = \Pi \left(|\mcW| + \sum_i \lambda_i \right) \sim \sum_i \Pi(\lambda_i) \sim \sum \bomega_\bq \ast \lambda_i$$
	\noindent Using proposition \ref{spectrum of PZ is given by K} and positivity of measures in $\lambda(\mcK^\ast)$ gives $\bomega_\bq \ast \sum_i \lambda_i \sim \bomega_\bq \ast \sum_{\bw \in \mcK^\ast} \lambda_\bw$.  
\end{proof}

\centerline{\Large{\textbf{Acknowledgement}}}

I would like to thank my advisor Boris Solomyak, Michael Baake, and an anonymous referee for their many valuable comments and insights that went into the preparation of this work.

\vspace{0.25in}

\centerline{\Large{\textbf{Notation Quick Reference}}}

{\small
\begin{multicols}{2}
{ {\large\underline{Substitutions}} \hfill \S \ref{sds section} \\ \vspace{0.07in} the alphabet $\alpha, \beta, \gamma, \delta \in \mcA$, and blocks $\omega, \eta \in \mcA^+$  \\ \vspace{0.07in} $T$ is the shift, and $\AZd$ the full $\Zd$-shift on $\mcA$ \\ \vspace{0.07in} $\mcL_\mcS$ is the language, $X_\mcS$ the substitution subshift \\ \vspace{0.07in} $[\alpha] = \{ \bA \in X_\mcS : \alpha = \bA(0) \}$ initial cylinders of $X_\mcS$ \\ \vspace{0.07in} $[\omega] = \{ \bA \in X_\mcS : \omega \leq \bA \}$ cylinders of $X_\mcS$ \\ \vspace{0.07in} $\mu \in \mcM(X_\mcS,T)$ an invariant probability measure \\ \vspace{0.07in} $\mcE, \mcT$: ergodic and transient classes of $\mcS$ \hfill (prop \ref{PRF})  \\ \vspace{0.07in} $\mcS^{-n}(\omega)$: the $n$-th desubstitutes of $\omega$ \hfill (def \ref{inverse substitution}) \\ \vspace{0.11in} }

{ {\large \underline{$\bq$-adic Arithmetic}} \hfill \S \ref{arithmetic} \\ \vspace{0.07in}   Write $\bk \in \Zd$ as $\qr{\bk}_n + \qq{\bk}_n \, \bq^n$ for $n \in \bbN$ \\ \vspace{0.07in}  $\qq{\bk}_n \in \Zd$ is the quotient of $\bk$ mod $\bq^n$ \\ \vspace{0.07in}  $\qr{\bk}_n \in [\bZero,\bq^n)$ is the remainder of $\bk$ mod $\bq^n$ \\ \vspace{0.07in}  $\bk = \bk_0 + \bk_1 \bq + \cdots + \bk_{n-1} \bq^{n-1} + \qq{\bk}_n \bq^n$ \\ \vspace{0.07in}  $\bk_n$ the $n$-th $\bq$-adic digit of $\bk$ \\ \vspace{0.07in}  $\mfp(\bk)$ the power of $\bk$ base $\bq$ \\ \vspace{0.07in}  Carry set $\Delta_n(\bk) = \{ \bj \in [\bZero,\bq^n) \, : \, \qq{\bj+\bk}_n \neq \bZero \}$ \\ \vspace{0.11in} }

{ {\large \underline{Configurations}} \hfill \S \ref{configurations} \\ \vspace{0.07in}  $\mcR_\bj : \mcA \to \mcA$, instructions of $\mcS$, $\mcR_\bj: \alpha \mapsto (\mcS \alpha)(\bj)$ \\ \vspace{0.07in}  $\mcR_\bj^{(n)}$ are the generalized instructions (of $\mcS^n$) \\ \vspace{0.07in}  $\mcR_\bj^{(n)} = \mcR_{\bj_0} \mcR_{\bj_1} \cdots \mcR_{\bj_{n-1}}$ for $\bj \in \Zd$ and $n > 0$ \\ \vspace{0.07in}  $\mcS \otimes \mcS$ is the bisubstitution, on the alphabet $\mcA^2$ \\ \vspace{0.11in} }

{ {\large\underline{Matrix Representation}} \hfill \S \ref{configurations} \\ \vspace{0.07in}  $\bM_\mcA(\bbC)$ the $\mcA \times \mcA$ indexed complex matrices \\ \vspace{0.07in}  $\otimes$ Kronecker product $(A \otimes B)_{\alpha \beta, \gamma \delta} = A_{\alpha \gamma} B_{\beta \delta}$ \\ \vspace{0.07in} $(\bA \otimes \bB) (\bC \otimes \bD) = (\bA \bC) \otimes (\bB \bD)$ \\ \vspace{0.07in}  Represent $\mcR\tbump{-2px}{$:$}\mcA \tbump{-2px}{$\to$} \mcA$ in $\bM_\mcA(\bbC)$ by $\mcR_{\alpha,\beta} \tbump{-2px}{$=$} \big\{ \scalebox{0.9}{${\tiny \begin{matrix} 1 \text{ if } \alpha = \mcR(\beta) \\   0 \text{ if } \alpha \neq \mcR(\beta) \end{matrix}}$}$ \\ \vspace{0.07in}  $M_\mcS = \sum \mcR_\bj$ the substitution matrix in $\bM_\mcA(\bbC)$ \\ \vspace{0.07in}  $\bu$ is a positive combination of Perron vectors of $M_\mcS$ \\ \vspace{0.07in}  $C_\mcS = \sum \mcR_\bj \otimes \mcR_\bj$ the coincidence matrix in $\bM_{\mcA^2}(\bbC)$ \\ }

\vspace{0.09in}

{ {\large \underline{Spectral Theory}}  \hfill \S \ref{spectral theory} \\ \vspace{0.07in} $\bnu_\mcL$ is Haar measure for $\mcL$-th roots of unity \hfill (\ref{lattice measure}) \\ \vspace{0.07in}  $\bomega_\bq$ is the $\bq$-adic support measure $\sum 2^{-n} \bnu_{\bq^n}$ \hfill \S \ref{height} \\ \vspace{0.07in} $\bz^\bq := (z_1^{q_1}, \ldots, z_d^{q_d})$ for $\bq \in \Zd$, $\bz \in \bbT^d$ \\ \vspace{0.07in}  $\bS_\bq$ the $\bq$-shift $\bz \mapsto \bz^\bq$ on $\bbT^d$ \\ \vspace{0.07in}  $\pi$ the invariant contracted map and $\Pi = \sum 2^{-n} \pi^n$ \\ \vspace{0.07in}   $\sigmaX$ the maximal spectral type of $\mcS$ \hfill \S \ref{spectral theory section} \\ \vspace{0.07in}  $\Sigma := (\sigma_{\alpha \beta})_{\alpha \beta \in \mcA^2}$ the correlation vector \hfill \S \ref{sigma section} \\ \vspace{0.07in}  $\widehat{\Sigma}(\bZero) = \sum_{\alpha \in \mcA} u_\alpha \be_{\alpha \alpha}$ \\ \vspace{0.07in}  $\widehat{\Sigma}(\bk) = \frac{1}{Q^n} {\sum}_{\bj \in [\bZero,\bq^n)} \mcR_\bj^{(n)} \otimes \mcR_{\bj+\bk}^{(n)} \widehat{\Sigma}(\qq{\bj+\bk}_n)$  \\ \vspace{0.07in}  For $\bv \in \bbC^{\mcA^2}$, write $\lambda_\bv := \bv^t \Sigma \ll \sigmaX$ \\ \vspace{0.07in}  For $\bv \in \bbC^{\mcA^2}$ its associated matrix $\mathring{\bv} \in \bM_\mcA(\bbC)$  \hfill \S \ref{spectral hull} \\ \vspace{0.07in}  $\bv \stpos 0$ is strongly positive if $\mathring{\bv}$ is positive semidefinite \\ \vspace{0.07in}  $\mcK^\ast$ are the extreme points of the spectral hull $\mcK(\mcS)$ \\ \vspace{0.07in}  $\mcK(\mcS) = \{ \bv \in \bbC^{\mcA^2} \hspace{-2px}: C_\mcS^t \bv = Q \bv, \, \bv \stpos 0, \, \bv^t \widehat{\Sigma}(\bZero) = 1 \}$ \\ \vspace{0.07in}}

\end{multicols}

}
\newpage

\singlespacing

\begin{center} {\Large References} \end{center}

{\footnotesize
\noindent \begin{enumerate}
	\item \label{BGG}  M. Baake, F. G\"ahler and U. Grimm, Examples of substitution systems and their factors \textit{J. Int.Seq.} \textbf{16} (2013) art. 13.2.14 (18 pp)
	\item \label{baake and grimm} M. Baake and U. Grimm, Squirals and beyond: Substitution tilings with singular continuous spectrum \textit{Ergodic Th. and Dynam. Syst.} \textbf{34} (2014) 1077-1102
	\item \label{TAO} M. Baake, and U. Grimm, Aperiodic Order. Vol 1. A Mathematical Invitation. \textit{Enc. of Math. and its Appl.}, \textbf{149}. Cambridge: Cambridge University Press, 2013
	\item \label{BLV} M. Baake, D. Lenz, and A. Van Enter, Dynamical Versus Diffraction Spectrum for Structures with Finite Local Complexity,  arXiv:1307.7518
	\item \label{bartlett} A. Bartlett (2015), Spectral Theory of $\Zd$ Substitutions (doctoral dissertation).  Seattle, University of Washington.
	\item \label{BKMS}  S. Bezuglyi, J. Kwiatkowski, K. Medynets and B. Solomyak, Invariant measures on stationary Bratteli diagrams. \textit{Ergodic Th. and Dynam. Syst.}, \textbf{30} (2010), pp 973-1007.
	\item \label{cortez and solomyak} M. Cortez, and B. Solomyak. Invariant Measures for Non-Primitive Tiling Substitutions. \textit{J. d'Analyse Math.} \textbf{115.1} (2011): 293-342.
	\item \label{dekking} F. Dekking, The Spectrum of a Dynamical System Arising from Substitutions of Constant Length.  \textit{Zeit. Wahr. Verw.  Gebiete} \textbf{41} (1978), 221-239.
	\item \label{dworkin} S. Dworkin, Spectral theory and X-ray diffraction, \textit{J. Math. Phys.} \textbf{34} (1993) 2965-2967.
	\item \label{frank zd} N.P. Frank, Multidimensional constant-length substitution sequences, \textit{Topol. Appl.}, \textbf{152} (1Ð2) (2005), pp. 44-69
	\item \label{frank lebesgue} N.P. Frank, Substitution sequences in $\Zd$ with a nonsimple Lebesgue component in the spectrum, \textit{Ergodic Th. and Dynam. Syst.} \textbf{23} (2)(2003) 519-532.
	\item \label{gantmacher} F.R. Gantmacher, Applications of the Theory of Matrices, New York: Dover Publications, Inc., 2005.
	\item \label{HJ} R. Horn and C. Johnson, Topics in Matrix Analysis. Cambridge: Cambridge University Press, 1991.
	\item \label{LMS} J.-Y. Lee, R.V. Moody and B. Solomyak, Pure point dynamical and diffraction spectra, \textit{Ann. Henri Poincar\'e}, \textbf{3} (2002) 1003-1018
	\item \label{michel} P. Michel, Stricte ergodicit\'e d'ensembles minimaux de substitutions, \textit{C.R. Acad. Sc. Paris} \textbf{278} (1974), 811-813.
	\item \label{mosse} B. Moss\'e, Reconnaissabilit\'e des substitutions et complexit\'e des suites automatiques, \textit{Bull. Soc. Math. France} \textbf{124} (1996), 329-346.
	\item \label{mozes} S. Mozes, Tilings, substitution systems and dynamical systems generated by them,  \textit{J. d'Analyse Math.} \textbf{53}, (1989), 139-186.
	\item \label{pansiot} J. Pansiot. Decidability of periodicity for infinite words. \textit{RAIRO Inform. Th\'eor. App.} \textbf{20} (1986), 43-46.
	\item \label{queffelec} M. Queff\'elec, Substitution Dynamical Systems, Spectral Analysis. 2nd Ed. Berlin: Springer-Verlag, 2010.
	\item \label{radin} C. Radin, Miles of Tiles, AMS, Providence, RI, 1999
	\item \label{rudin} W. Rudin, Fourier Analysis on Groups, Wiley, New York 1962
	\item \label{robinson} E.A. Robinson, On the table and the chair, \textit{Indag. Math.} \textbf{10} (4) (1999) 581-599.
	\item \label{solomyak tiling} B. Solomyak, Dynamics of self-similar tilings. \textit{Ergodic Th. Dynam. Syst.}, \textbf{17} (1997), pp. 695-738
	\item \label{solomyak aperiodicity} B. Solomyak, Nonperiodicity Implies Unique Composition for Self-Similar Translationally Finite Tilings. \textit{Disc. and Comp. Geom.}, \textbf{20.2} (1998): 265-279.
	\item \label{sreider} Y.A. \v{S}reider, The Structure of Maximal Ideals in Rings of Measures with Convolution, \textit{Mat. Sbornik N S} 27 (69), 297-318 (1950), Am. Math. Soc. Translation no. 81, Providence, 1953
	\item \label{walters} P. Walters, An Introduction to Ergodic Theory. New York: Springer-Verlag, 1982.
\end{enumerate}
}

\end{document}